\documentclass{article}

\usepackage{amsmath,amsthm,amssymb,amscd,xcolor,mathrsfs,verbatim,microtype}
\usepackage{graphicx,eurosym}
\usepackage{mathtools}
\usepackage{bm}
\usepackage{enumitem}
\usepackage[cyr]{aeguill}
\usepackage{amscd,wrapfig,mathrsfs,lipsum}
\usepackage{float}
\usepackage{tikz}
\usepackage{multicol}
\usepackage{caption}
\usepackage{capt-of}
\usepackage{hyperref}

\colorlet{darkblue}{blue!50!black}

\hypersetup{
	colorlinks,%
	citecolor=blue,%
	filecolor=red,%
	linkcolor=darkblue,%
	urlcolor=blue,%
	pdfnewwindow=true,%
	pdfstartview={FitH}
}

\usetikzlibrary{arrows}
\newcommand{\LLLLL}{\mathbb{L}}
\newcommand{\SSSSS}{\mathbb{S}}

\newcommand{\Q}{{\mathbb Q}}

\newcommand{\gggg}{{\mathbf g}}

\newcommand{\ppP}{{\mathsf P}}
\newcommand{\qqQ}{{\mathsf Q}}
\newcommand{\GGGG}{{\mathfrak G}}
\newcommand{\A}{{\mathcal A}}
\newcommand{\bmnu}{{\bm\nu}}

\DeclareMathOperator{\per}{Per}
\DeclareMathOperator{\ent}{Ent}

\newcommand{\F}{{\boldsymbol{F}}}

\newcommand{\R}{{\mathbb R}}

\newcommand{\I}{{\mathbb I}}
\newcommand{\E}{{\mathbb E}}

\newcommand{\bmlambda}{{\bm{\lambda}}}
\newcommand{\bmLambda}{{\bm{\Lambda}}}
\newcommand{\N}{{\mathbb N}}

\newcommand{\ty}{\infty}

\newcommand{\MMM}{\mathfrak{M}}

\newcommand{\LLLL}{{\mathfrak L}}

\newcommand{\wwww}{{\mathfrak w}}

\newcommand{\BH}{{\boldsymbol{\mathrm{H}}}}

\newcommand{\MU}{{\boldsymbol{\mathfrak{{A}}}}}
\newcommand{\uu}{{\boldsymbol{ {u}}}}
\newcommand{\vv}{{\boldsymbol{ {v}}}}

\newcommand{\III}{{\boldsymbol{I}}}

\newcommand{\bmzeta}{\bm{\zeta}}

\newcommand{\BB}{{\cal B}}
\newcommand{\CC}{{\cal C}}
\newcommand{\DD}{{\cal D}}
\newcommand{\EE}{{\cal E}}
\newcommand{\FF}{{\cal F}}
\newcommand{\GG}{{\cal G}}

\newcommand{\II}{{\cal I}}
\newcommand{\JJ}{{\cal J}}
\newcommand{\KK}{{\cal K}}
\newcommand{\LL}{{\cal L}}
\newcommand{\MM}{{\cal M}}

\newcommand{\PP}{{\cal P}}

\newcommand{\RR}{{\cal R}}

\newcommand{\TT}{{\cal T}}

\newcommand{\VV}{{\cal V}}
\newcommand{\WW}{{\cal W}}
\newcommand{\XX}{{\cal X}}
\newcommand{\YY}{{\cal Y}}
\newcommand{\ZZ}{{\cal Z}}

\newcommand{\PPPPPP}{{\mathbb{P}}}
\newcommand{\dd}{{\textup d}}
\newcommand{\bI}{\bm{I}}
\newcommand{\PPPP}{{\mathfrak P}}

\newcommand{\RRRR}{{\mathfrak R}}
\newcommand{\FFFF}{{\mathfrak F}}
\newcommand{\XXXX}{{\mathfrak X}}

\newcommand{\VVV}{\boldsymbol{V}}

\newcommand{\DDDD}{{\mathfrak D}}

\newcommand{\SSS}{{\mathscr S}}
\newcommand{\fff}{{\boldsymbol{\mathit f}}}

\newcommand{\uuu}{{\boldsymbol{\mathit u}}}

\newcommand{\www}{{\boldsymbol{\mathit w}}}

\newcommand*\mcap{\mathbin{\mathpalette\mcapinn\relax}}
\newcommand*\mcapinn[2]{\vcenter{\hbox{$\mathsurround=0pt
			\ifx\displaystyle#1\textstyle\else#1\fi\bigcap$}}}
\newcommand*\mcup{\mathbin{\mathpalette\mcupinn\relax}}
\newcommand*\mcupinn[2]{\vcenter{\hbox{$\mathsurround=0pt
			\ifx\displaystyle#1\textstyle\else#1\fi\bigcup$}}}
\newcommand{\lspan}{\mathop{\rm span}\nolimits}

\newcommand{\supp}{\mathop{\rm supp}\nolimits}
\newcommand{\diver}{\mathop{\rm div}\nolimits}

\newcommand{\dist}{\mathop{\rm dist}\nolimits}

\newcommand{\Osc}{\mathop{\rm Osc}\nolimits}

\theoremstyle{plain}

\newtheorem*{lemma*}{Lemma}
\newtheorem{theorem}{Theorem}[section]
\newtheorem{lemma}[theorem]{Lemma}
\newtheorem{proposition}[theorem]{Proposition}
\newtheorem{corollary}[theorem]{Corollary}
\theoremstyle{definition}

\newtheorem{remark*}{Remark}
\theoremstyle{remark}

\numberwithin{equation}{section}

\begin{document} 
	\author{Meng~Zhao\,\footnote{School of Mathematical Sciences, Shanghai Jiao Tong University, 200240 Shanghai, China; Universit\'e Paris Cit\'e and Sorbonne Universit\'e, IMJ-PRG, F-75013 Paris, France, e-mail: \href{mailto:mathematics_zm@sjtu.edu.cn}{mathematics$\_$zm@sjtu.edu.cn}}
	}

	\title{Level-3 large deviations for the white-forced 2D Navier--Stokes system in a bounded domain} 
	\date{\today}
	\maketitle
	\begin{abstract}
		We study the large deviations principle (LDP) of Donsker--Varadhan type for the white-forced Navier--Stokes system in a bounded domain. Under the assumption that the noise is non-degenerate, we establish level-2 and level-3 LDPs with rate functions given by the Donsker--Varadhan formulas. The proof relies on an improved version of Kifer's criterion, a lift argument inspired from \cite{DV1983}, an improved abstract result on the large-time asymptotics of generalized Markov semigroups, and a delicate approximation scheme utilizing the resolvent operators of the Markov semigroup.
	\end{abstract}
	\tableofcontents
	\section{Introduction}\label{Intro}
	In this paper, we establish both level-2 and level-3 large deviations principles (LDP) of Donsker--Varadhan type for the 2D white-forced Navier--Stokes system for incompressible fluids in a smooth bounded domain:
	\begin{align}
		\label{I1}
		\begin{cases}
			\partial_tu+ (u\cdot\nabla)u-\nu\Delta u+\nabla p=h+ \eta,\\
			\diver u=0,\\
			u(0)=u_0,\qquad u|_{\partial D}=0,
		\end{cases}
	\end{align}
	where $u$ is the velocity field of the fluid, $p$ is the pressure, and $\nu>0$ is the kinematic viscoity. Applying the Leray projection $\Pi$ on both sides of \eqref{I1}$_1$, setting $\nu=1$ for simplicity, and writing 
	\[Lu:=-\Pi\Delta u,\qquad B(u):=\Pi(u\cdot \nabla)u,\]
	we eliminate the pressure and consider the system as an abstract evolutionary equation 
	\[\partial_t u+B(u)+ Lu=h+\eta\]
	in the space of divergence-free vector fields
	\begin{align}
		\label{I3}
		H:=\{u\in L^2| \diver u=0, (u\cdot n)|_{\partial D}=0\}.
	\end{align}
    The space $H$ is endowed with the inner product $\langle u,v\rangle$
	and the norm $\|u\|^2:=\langle u,u\rangle$ inherited from $L^2$. The driving force consists of two parts: $h\in H$ is a given deterministic function and $\eta$ is a white-in-time noise of the form
	\begin{align}\label{I4}
		\eta(t,x):=\partial_t\sum_{j\ge 1} b_je_j(x)\beta_j(t),
	\end{align}
	where $\{b_j\}_{j\ge 1}$ is a sequence of real numbers satisfying 
	\begin{align}
		\label{I2}
		\BB_1:=\sum_{j\ge 1} \alpha_jb_j^2<\infty,
	\end{align}
	$\{e_j\}_{j\ge 1}$ is an orthonormal basis in $H$ formed by eigenfunctions of the Stokes operator $L$ associated with the eigenvalues $\{\alpha_j\}_{j\ge 1}$, and $\{\beta_j\}_{j\ge 1}$ are independent standard Brownian motions defined on a probability space $(\Omega,\mathscr{F},\mathscr{F}_t,\mathbb{P})$  satisfying the usual conditions (e.g., see Definition 2.25 in \cite{KS91}).
	
	\subsection{Main results}  
	Under the above assumptions, the system \eqref{I1} defines a Markov family $(u_t,\mathbb{P}_u)$ parameterized by the initial condition $u\in H$. The large-time asymptotic behaviour of this family is now well understood. In particular, it is known that there is a unique invariant measure which attracts exponentially the distribution of all solutions, provided that the noise is sufficiently non-degenerate. We refer the readers to the papers \cite{FM95,EMS01,KS02,BKL02,HM06} for the first results, and the book \cite{KS12} and reviews \cite{D13,KS17} for the subsequent developments. Besides, the central limit theorem (CLT) holds, which describes the fluctuations of time average of a functional around its mean value, cf. \cite{Shi06}.
    In the present paper, our main goal is to quantify the probabilities of large deviations from the mean value, which is a natural extension of the CLT. To formulate the main results of the paper, let us denote the Markov operators associated with $(u_t,\PPPPPP_u)$ by 
	\begin{align}
		\PPPP_tf(u)&:=\int_{H}f(v)P_t(u,\mathrm{d} v), &\PPPP_t&:C_b(H)\to C_b(H),\label{Markov1}\\\label{Markov2}\PPPP^*_t\lambda(\Gamma)&:=\int_{H}P_t(u,\Gamma)\lambda(\mathrm{d} u),  &\PPPP^*_t&:\PP(H)\to \PP(H),\end{align}
	where \begin{align}\label{Markov3}P_t(u,\Gamma):=\mathbb{P}_u\left\{u_t\in\Gamma\right\}\end{align}
	is the transition function and $\PP(H)$ denotes the space of all Borel probability measures on $H$ endowed with the weak topology. For $\sigma\in\PP(H)$, we write 
	\[\mathbb{P}_{\sigma}(\Gamma):=\int_H\mathbb{P}_u(\Gamma)\sigma(\dd u),\qquad \Gamma\in\mathscr{F},\]
	and consider the occupation measures 
	\begin{align}
		\label{MR1}
		\zeta_t:=\frac{1}{t}\int_0^t\delta_{u_s}\dd s,\qquad t>0
	\end{align}
	defined on $(\Omega,\mathscr{F},\mathbb{P}_{\sigma})$, where $\delta_u$ is the Dirac measure centered at $u$. Note that $\{\zeta_t\}_{t>0}$ is a family of random probability measures on $H$. Recall that a function $I: \PP(H)\to [0,\infty]$ is said to be a good rate function, if the level set $\{\lambda\in \PP(H)| I(\lambda)\le \alpha\}$ is compact for any $\alpha\ge 0$.
 
	Our first result establishes a level-2 LDP for $\{\zeta_t\}_{t>0}$ with a rate function $I$ given by the Donsker--Varadhan variational formula.
	\begin{theorem}
		\label{MT1}
		Suppose that $h\in U:=H_0^1\mcap H$, $b_j>0$ for all $j\ge 1$, and \eqref{I2} holds. Then, for any $\gamma,M>0$, the family $\{\zeta_t\}_{t>0}$ satisfies an LDP uniformly in 
		\begin{align}\label{MR0}\sigma\in \Lambda(\gamma,M):=\left\{\sigma\in \PP(H)\Big| \int_H e^{\gamma\|u\|^2}\sigma(\dd u)\le M\right\},\end{align}
		with a good rate function $I:\PP(H)\to [0,\infty]$ that is independent of $\gamma,M$. More precisely, the following upper and lower bounds hold:
		\begin{enumerate}
			\item for any closed set $F\subset \PP(H)$, we have 
			\begin{align}
				\label{MR2}
				\limsup_{t\to+\infty}\frac{1}{t}\log\left(\sup_{\sigma\in \Lambda(\gamma, M)}\mathbb{P}_{\sigma}\{\zeta_t\in F\}\right)\le -\inf_{\lambda\in F} I(\lambda);
			\end{align}
			\item for any open set $G\subset \PP(H)$, we have 
			\begin{align}
				\label{MR3}
				\liminf_{t\to+\infty}\frac{1}{t}\log\left(\inf_{\sigma\in \Lambda(\gamma, M)}\mathbb{P}_{\sigma}\{\zeta_t\in G\}\right)\ge -\inf_{\lambda\in G} I(\lambda).
			\end{align}
		\end{enumerate}
		Moreover, the rate function $I$ is given by the Donsker--Varadhan variational formula
		\begin{align}
			\label{MR4}
			I(\lambda)=\sup_{\substack{f\in \mathcal{D}(\LLLL)\\ \inf_H f>0}}\int_H\frac{-\LLLL f}{f}\dd\lambda,
		\end{align}
		where $\LLLL$ is the generator of the Markov semigroup $\{\PPPP_t\}_{t\ge0}$ defined on the domain
		\begin{align}\label{MR5}
			\mathcal{D}(\LLLL):=\left\{f\in C_b(H)\Big|\exists g\in C_b(H),\ \PPPP_tf-f=\int_0^t\PPPP_sg\dd s,\ \forall t\ge 0\right\},
		\end{align}
		and $\LLLL f:=g$ for $f\in \mathcal{D}(\LLLL)$.
	\end{theorem}
	\begin{remark*}\label{remark1}
		Let us note that $\{\PPPP_t \}_{t\ge 0}$ may fail to be a $C_0$-semigroup, due to the non-compactness of the phase space $H$. However, from the Feller property of $\PPPP_t$ and the continuity of the sample paths, we have the continuity of the map $(t,u)\mapsto \PPPP_tf(u)$, which, by the Fubini theorem, allows us to integrate $s\mapsto \PPPP_sg(u)$ for fixed $u\in H$. Therefore, the integral term on the right hand-side of \eqref{MR5} is well-defined.
	\end{remark*}
	To present our next result, let us fix any $T>0$ and set $X^{T}:=C([0,T];H)$. We define the empirical measures of finite-length trajectories by
	\begin{align} 
		\label{MR7} 
		\bmzeta^T_t:=\frac{1}{t}\int_0^t\delta_{u_{[s,s+T]}}\dd s,
	\end{align}
    where $u_{[t,t+T]}$ denotes the trajectory $\{u_r\}_{r\in[t,t+T]}$. Then, $\{\bmzeta^T_t\}_{t>0}$ is a family of random probability measures on $X^{T}$. The following result establishes an LDP for $\{\bmzeta^T_t\}_{t>0}$.
	\begin{theorem}
		\label{MT2}
		Under the assumptions of Theorem~\ref{MT1}, for any $\gamma,M>0$, the family $\{\bmzeta^T_t\}_{t>0}$ satisfies an LDP, uniformly in $\sigma\in \Lambda(\gamma,M)$, with a good rate function $I_T:\PP(X^{T})\to [0,\infty]$ that is independent of $\gamma,M$.
	\end{theorem}
	Finally, to formulate our result on level-3 LDP, we recall some basic definitions. Let us consider the empirical measures of full trajectories 
	\begin{align}\label{MR8}
		\bmzeta_t:=\frac{1}{t}\int_0^t\delta_{u_{[s,\infty)}}\dd s,
	\end{align}
	which are random probability measures on the space $X:=C([0,\infty);H)$ endowed with the topology of uniform convergence on bounded intervals. A measure $\bmlambda\in\PP(X)$ is said to be shift-invariant, if it is invariant under the shift operator $\theta_s, s>0$ defined by
	\begin{align}
    \label{MR11}
    \theta_su(\cdot):= u(s+\cdot),\qquad \theta_s: X\to X.
    \end{align}
    Let us denote by $\PP_s(X)$ the space of all shift-invariant measures on $X$. According to the Kolmogorov theorem, each shift-invariant measure $\bmlambda\in \PP_s(X)$ can be uniquely extended to a shift-invariant measure on $X^{\R}:=C(\R;H)$. With a slight abuse of notation, we still denote this extension by $\bmlambda$. Let $\bmlambda(u_{(-\infty,0]},\cdot)$ denote the regular conditional probability of $\bmlambda$ given $\mathscr{F}^{-\infty}_{0}$,
	where for any $-\infty\le s\le t\le \infty$, $\mathscr{F}^s_t$ is the $\sigma$-algebra generated by the family of projections $\{\pi_r\}_{r\in [s,t]}$ defined by 
	\[\pi_r u_{(-\infty,\infty)}:=u(r),\qquad \pi_r: X^\R\to H.\]
	For any interval $\JJ\subset \R$, we define the projection $\pi_\JJ$ by 
    \[\pi_{\JJ}u_{(-\infty,\infty)}:=u_{\JJ},\qquad \pi_\JJ: X^\R\to C(\JJ;H).\]
	\begin{theorem}\label{MT3}
		Under the assumptions of Theorem~\ref{MT1}, for any $\gamma,M>0$, the family $\{\bmzeta_t\}_{t>0}$ satisfies an LDP, uniformly in $\sigma\in \Lambda(\gamma,M)$, with a good rate function $\bI:\PP(X)\to [0,\infty]$ given by the Donsker--Varadhan entropy formula
		\begin{align}\label{MR9}
			\bI(\bmlambda)=\begin{cases}
				\int_{X^{\R}}\ent\left(\pi_{[0,1]}^*\bmlambda(u_{(-\infty,0]},\cdot)\Big| \pi_{[0,1]}^*\mathbb{P}_{u(0)}\right)\bmlambda(\dd u_{(-\infty,\infty)})&\mbox{if $\bmlambda\in \PP_s(X)$},\\
				+\infty &\mbox{otherwise},
			\end{cases}
		\end{align}
		where $\pi_{[0,1]}^*$ is the push-forward operator associated with $\pi_{[0,1]}$ and $\ent(\mu_1|\mu_2)$ is the relative entropy of $\mu_1$ with respect to $\mu_2$ defined by \eqref{relative-entropy}. 
	\end{theorem}
    \subsection{Literature review and scheme of the proof}\label{scheme}
	These types of LDPs were first established by Donsker and Varadhan \cite{DV1975,1DV1975,DV1976,DV1983} for general Markov processes under various assumptions. Later, this topic was studied by many authors, and we refer to the book \cite{DZ2010} for a detailed account of the main achievements in this field. 
    
    In the context of stochastic PDEs, only a few results have been obtained regarding the Donsker--Varadhan type LDP. The first results are due to Gourcy \cite{Gou07b,Gou07a}, who established a level-2 LDP for the stochastic Navier--Stokes and Burgers equations on the torus driven by white-in-time and irregular-in-space noise. In these papers, the spatial irregularity of the noise is utilized to ensure the strong Feller property of the corresponding Markov family. However, from a physical point of view, irregular random forces are not very natural. This assumption was later relaxed by Jak\v{s}i\'c et al. in \cite{JNPS15,JNPS18,JNPS21}, where they established a level-3 LDP for a class of dissipative PDEs perturbed by random kick forces and bounded noise. Their proof combines a non-compact version of Kifer's criterion with a study of the large-time behaviour of generalized Markov semigroups. This approach was further extended to white-forced PDEs, leading to level-2 LDP. For the case of non-degenerate noises, see \cite{MN18,Ner19}; for the degenerate case, we refer to \cite{NPX23}. Recently, by combining the asymptotic compactness property of the dynamics and the Kifer-type framework, Chen et al. \cite{CLXZ2024} proved a level-2 LDP for a nonlinear wave equation with localized damping and bounded noise. We also mention  \cite{LL23}, where the authors generalize the results in \cite{NPX23} to the quasi-periodic case. 
	
    To the best of our knowledge, Theorem~\ref{MT3} is the first result that establishes a level-3 LDP of Donsker--Varadhan type for white-forced stochastic PDEs. Compared to existing results, the most important new difficulty is the verification of the exponential growth condition required by the generalized Kifer's criterion, see (3.4) in \cite{JNPS18}. As shown in the proof of Theorem 3.3 therein, this condition is not only required for proving exponential tightness, but also plays a pivotal role in ensuring an important covering property for the level sets of the rate function. However, in the context of level-3 LDP for white-forced stochastic PDEs, such a condition is not necessarily satisfied. To be precise, one has to find a functional $\bm\Phi:X^T\to [0,\infty)$ with compact level sets such that
    \begin{align}\label{EGC}
        \E_{u}\exp\left(\int_0^t\bm{\Phi}(u_{[s,s+T]})\dd s\right)\lesssim e^{ct},
    \end{align}
    in order to verify the condition. In contrast, due to the a priori estimates \eqref{A5} and \eqref{A10}, only doubly-logarithmic moment estimates can be expected for the higher-order norms. Even in the case of periodic boundary conditions, one can only have an estimate similar to \eqref{EGC} but with the right hand-side replaced by $\exp(ct^2)$.
    
    To address this challenge, our approach combines two ingredients. First, using a crucial lower bound estimate for the rate function, see (3.6) in \cite{CLXZ2024}, we improve the generalized Kifer's criterion by replacing the exponential growth condition with a weaker condition, that is, exponential tightness along any subsequence, see Proposition~\ref{proposition-kifer}. Second, in order to verify this weaker condition, we develop a lift argument by utilizing some ideas from Donsker and Varadhan \cite{DV1983}, where the authors established a level-3 LDP for general Markov processes whose transition function and generator satisfy seven conditions. In our current situation, verifying these conditions seems to be impossible, as, for example, one of these conditions requires the existence of density of the transition function with respect to a reference measure. Nevertheless, we observe that by lifting the level-2 exponential growth condition \eqref{A2} to the trajectory level, see Lemma \ref{lemma4-1}, a level-3 LD upper bound holds for the empirical measures of periodized trajectories
	\begin{align}\label{MR10}
		\bmzeta_t^{\per}:=\frac{1}{t}\int_0^t\delta_{\theta_su_{\per,t}}\dd s,
	\end{align}
    where $u_{\per,t}$ is the periodized trajectory defined by 
	\begin{align*}
		u_{\per,t}(r)&=u(r),\qquad 0\le r<t,\\
		u_{\per,t}(r+t)&=u_{\per,t}(r),\qquad r\in\R.
	\end{align*}
    Furthermore, with the help of a contraction inequality \eqref{contraction_inequality} and the goodness of the level-2 rate function $I$ obtained in Theorem \ref{MT1}, the aforementioned LD upper bound admits a good rate function $\BH$ given by the Donsker--Varadhan entrophy \eqref{ET1}. This, together with the exponential equivalence of $\{\bmzeta_t^{\per}\}_{t>0}$ and $\{\bmzeta_t\}_{t>0}$, implies this weaker condition. In addition, we mention that with the help of the doubly logarithmic moment estimates \eqref{A5} and \eqref{A10}, the exponential tightness holds for the whole family $\{\zeta_t\}_{t>0}$. As discussed previously, this property is unlikely to be attainable via purely PDE-based methods, although a weaker form is still sufficient for the level-3 LDP.

    The verification of the remaining two conditions in Proposition~\ref{proposition-kifer} relies on an analysis of large-time asymptotics of the Feynman--Kac semigroup. However, due to the presence of a physical boundary, the system \eqref{I1} exhibits weaker dissipation: the $H^1$-norm cannot be fully dissipated, posing a significant challenge. Specifically, the weak dissipativity leads to the existence of eigenvectors satisfying only a doubly-logarithmic moment estimate for the $H^1$-norm, while the uniqueness guaranteed by the abstract criteria in \cite{JNPS18,MN18} holds in a class requiring higher-order polynomial moment estimates. Consequently, the existence is established in a broader class than the uniqueness. To resolve this discrepancy, we expand the class in which the uniqueness holds by utilizing a stronger uniform irreducibility property ensured by the parabolic smoothing effect. This leads to an improved abstract criterion for large-time asymptotics of generalized Markov semigroups, see Proposition~\ref{propositionE2}, which is well-suited for the Feynman--Kac semigroup with weak dissipation.

    Let us emphasize that Theorems~\ref{MT1} and \ref{MT3} are the first results establishing the Donsker--Varadhan formulas \eqref{MR4} and \eqref{MR9} in the context of stochastic PDEs with non-compact phase spaces. These results can be viewed as a generalization of the compact case studied in \cite{JNPS21}. In our setting, the main difficulty arises from the potential growth of the eigenfunctions of the Feynman--Kac semigroup. More precisely, the growth of the eigenfunctions leads to a variational formula similar to \eqref{MR4}, but with the supremum taken over a class of unbounded continuous test functions, as detailed in \eqref{lvl2-24}. To address this issue, we approximate the unbounded test function through carefully designed cut-off and regularization techniques, utilizing the resolvent operators of the Markov semigroup. This approach enables the justification of the Donsker--Varadhan variational formula \eqref{MR4}, see Section~\ref{level2} for further details. The Donsker--Varadhan entropy formula \eqref{MR9} is proved by generalizing the approach of \cite{JNPS21} to the continuous-time setting, combined with the varational formula and a contradiction argument. Further details can be found in Section~\ref{proofMT3}.

    Finally, our method is flexible and can be applied to other stochastic dissipative PDEs on compact domains with or without boundary. We believe that by combining our framework and the Malliavin calculus method in \cite{NPX23}, this level-3 LDP holds for the stochastic Navier--Stokes system on the torus with degenerate white-in-time noise. 
    \subsection*{Structure of the paper}
    The paper is organized as follows. Section~\ref{level2} is devoted to the proof of Theorem~\ref{MT1}. Sections~\ref{ET} and \ref{largetimeFK} focus on the verification of the conditions in the improved version of Kifer's criterion, cf. Proposition~\ref{proposition-kifer}. This verification relies on Theorem~\ref{MT1}, a lift argument inspired by \cite{DV1983}, and an improved abstract criterion for large-time asymptotics of generalized Markov semigroups, cf. Proposition~\ref{propositionE2}. The proofs of Theorem~\ref{MT2} and \ref{MT3} are presented in Sections~\ref{proofMT2} and \ref{proofMT3}, respectively. Finally, the Appendix gathers a priori estimates for the solutions of the system \eqref{I1}, some properties of Markov semigroups on weighted phase spaces, basic properties of the Donsker--Varadhan entropy, and the proofs of the abstract criteria stated in Propositions~\ref{proposition-kifer} and \ref{propositionE2}.
 
    \subsection*{Acknowledgement}  

    The author's research was partially supported by the National Natural Science Foundation of China under Grant Nos. 12171317, 12331008, 12250710674, and 12161141004. This paper was completed during the author's visit to Universit\'e Paris Cit\'e. He would like to express sincere gratitude to the China Scholarship Council for its financial support, as well as to Professors Sergei Kuksin, Armen Shirikyan, and Vahagn Nersesyan, for the invitation to Paris, and for the valuable discussions and support throughout the visit.

	\subsection*{Notation}
	We shall use the following standard notations.
	\begin{itemize}
		\item  For any Polish space $Y$, let $C_b(Y)$ be the space of all bounded continuous functions $V:Y\to\R$ endowed with the norm
		\[\|V\|_{\infty}:=\sup_{y\in Y}|V(y)|,\]
		and let $L_b(Y)$ be the space of bounded Lipschitz functions endowed with the norm
		\begin{align}
		    \label{lipnorm}\|V\|_{L_b}:=\|V\|_{\infty}+\sup_{\substack{y_1,y_2\in Y\\ y_1\neq y_2}}\frac{|V(y_1)-V(y_2)|}{\dist(y_1,y_2)}.
		\end{align}
		\item By $\mathscr{B}(Y)$, we denote the Borel $\sigma$-algebra of $Y$. For $A\in\mathscr{B}(Y)$, $\I_A$ is the indicator function of the set $A$. For $R>0$ and $y\in Y$, let $B_Y(y,R)$ be the open ball centered at $y$ of radius $R$, and we write $B_Y(R):=B_Y(0,R)$, if $Y$ is a separable Banach space.
	\item Let $\MM^+(Y)$ be the set of non-negative finite Borel measures on $Y$ endowed with topology of weak convergence, and let $\PP(Y)$ denote the space of all Borel probability measures on $Y$. For a measurable function $f:Y\to\R$ and a measure $\mu\in\MM^+(Y)$, we write
		\[\langle f,\mu\rangle:=\int_Y f(y)\dd\mu.\]
        The weak convergence in $\PP(Y)$ is equivalent to the convergence in the following dual-Lipschitz metric:
        \[\|\mu_1-\mu_2\|^*_{L(Y)}:=\sup_{\substack{f\in L_b(Y)\\\|f\|_{L_b}\le 1}}|\langle f,\mu_1\rangle-\langle f,\mu_2\rangle|,\qquad \mu_1,\mu_2\in \PP(Y).\]
		\item Let us define
		\[H:=\{u\in L^2| \diver  u=0, (u\cdot n)|_{\partial D}=0\}\]
		with the norm $\|\cdot\|$ and inner product $\langle\cdot,\cdot\rangle$ inherited from $L^2$. For any integer $j\ge 1$, we write $U^j:=H^j\mcap H_0^1\mcap H$ with the norm $\|\cdot\|_j$ and inner product $\langle\cdot,\cdot\rangle_j$ inherited from the usual Sobolev space $H^j$. Let $U^*$ be the dual space of $U:=U^1$ and denote by $\|\cdot\|_{-1}$ its norm. Let $\ppP_N$ be the orthogonal projection in $H$ onto the space 
		\[H_N:=\lspan\{e_1,e_2,\ldots,e_N\},\]
		where $\{e_j\}_{j\ge 1}$ is the orthonormal basis entering \eqref{I4}. Let $\qqQ_N:=I-\ppP_N$.
		\item For any interval $\JJ\subset \R$, we define the space
		\[X^{\JJ}:=C(\JJ;H)\]
        endowed with the topology of uniform convergence on any bounded intervals. Furthermore, by $\DDDD^{\JJ}$ we denote the space of functions $f:\JJ\to H$ with discontinuities only of the first kind and normalized to be right-continuous. The convergence in $\DDDD^{\JJ}$ is induced by the Skorokhod topology, cf. (16.4) in \cite{Bil1999}. For simplicity, we write \[\DDDD:=\DDDD^{[0,\infty)},\qquad\DDDD^T:=\DDDD^{[0,T]},\]
		and 
		\[X:=X^{[0,\infty)},\qquad X^T:=X^{[0,T]}.\]
		\item For any $r\in \R$, we introduce the projection from $\DDDD^\R$ to $H$:
		\[\pi_ru_{(-\infty,\infty)}=u(r),\qquad u_{(-\infty,\infty)}\in \DDDD^\R.\]
        The associated push-forward operator $\pi_r^*: \PP(\DDDD^\R)\to \PP(H)$ is defined as
        \begin{align}\label{pushforward1}
            \pi_r^*\bmlambda(A)=\bmlambda(\pi_r^{-1}(A)),\qquad\forall A\in \mathscr{B}(H).
        \end{align}
		For any interval $\JJ\subset \R$, we introduce the projection from $\DDDD^\R$ to $\DDDD^\JJ$:
		\[\pi_{\JJ}u_{(-\infty,\infty)}=u_{\JJ},\qquad u_{(-\infty,\infty)}\in \DDDD^\R.\]
        The associated push-forward operator $\pi_\JJ^*: \PP(\DDDD^\R)\to \PP(\DDDD^\JJ)$ is defined as
        \begin{align}\label{pushforward2}\pi_\JJ^*\bmlambda(\mathbf{A})=\bmlambda(\pi_\JJ^{-1}(\mathbf{A})),\qquad\forall \mathbf{A}\in \mathscr{B}(\DDDD^\JJ).
        \end{align}
		For any $-\infty\le s\le t\le \infty$, $\mathscr{F}^s_t$ is the $\sigma$-algebra generated by the family of projections $\{\pi_r,r\in [s,t]\}$.
		\item For Banach spaces $X,Y$, we denote by $\LL(X\to Y)$ the space of bounded linear operators from $X$ to $Y$ and write $\LL(X):=\LL(X\to X)$.
		\item Given two measures $\mu_1,\mu_2$ defined on a measurable space $(\Omega,\mathscr{F})$, the relative entropy of $\mu_1$ with respect to $\mu_2$ is defined by
		\begin{align}\label{relative-entropy}
			\ent (\mu_1| \mu_2):=\sup_{F\in \BB(\mathscr{F})}\left(\langle F,\mu_1\rangle-\log \langle e^F,\mu_2\rangle\right),
		\end{align}
		where $\BB(\mathscr{F})$ denotes the space of all bounded $\mathscr{F}$-measurable functions $F:\Omega\to\R$. 
        \item For $a,b\in \R$, we write 
		\[a\wedge b:=\min\{a,b\},\qquad a\vee b:= 
		\max\{a,b\}.\]
		By $[a]$, we denote the largest integer that does not exceed $a$.
		\item By $C$, $C_a$, $C_T$, and so on, we denote positive constants that are not important to the analysis, with subscripts indicating dependence on specific parameters. For simplicity, we shall frequently use the notation $\lesssim$, $\lesssim_a$, $\lesssim_T $, etc., to indicate inequalities that hold up to an unessential multiplicative constant, such as  $C$, $C_a$, $C_T$, and so on.
	\end{itemize}

	\section{Proof of Theorem~\ref{MT1}}
	\label{level2}
	The LD upper and lower bounds with a good rate function $I$ can be established as in \cite{Ner19}. The proof relies on a Kifer-type criterion, cf. Theorem 3.3 in \cite{JNPS18}, combined with the $H^1$-doubly logarithmic moment estimate \eqref{A5}. We will only briefly outline the proof of the LD upper and lower bounds, and focus our attention on the proof of the Donsker--Varadhan variational formula~\eqref{MR4}.

    The exponential growth condition (3.4) in \cite{JNPS18} is guaranteed by \eqref{A2}. The remaining two conditions, namely existence of pressure and uniqueness of equilibrium required by the Kifer-type criterion are verified through a multiplicative ergodic theorem (MET), cf. Theorem 2.1 in \cite{Ner19}. However, as discussed earlier, the weak dissipativity of the system \eqref{I1}, caused by the presence of a physical boundary, prevents the direct application of the proof of Theorem 2.1 in \cite{Ner19} to establish the MET. This issue is resolved by employing Proposition~\ref{propositionE2}, which is an improved abstract criterion for large-time asymptotics of generalized Markov semigroups. All conditions in this criterion can be verified similarly to the arguments in Sections 3 and 4 of \cite{Ner19}; see also Section~\ref{largetimeFK} in this paper for a more general case, where this criterion is applied to establish a level-3 MET. Therefore, we omit the details here. By applying the Kifer-type criterion, the LD upper and lower bounds are established, with a good rate function $I$ given by
	\begin{align}\label{lvl2-1}
		I(\lambda)=\sup_{V\in C_b(H)}\left(\int_HV\mathrm{d}\lambda-Q(V)\right),\qquad \lambda\in\PP(H),
	\end{align}
	where 
	\begin{align}\label{lvl2-2}
		Q(V):=\lim_{t\to+\infty}\frac{1}{t}\log\E_{\sigma}\exp\left(\int_0^tV(u_s)\dd s\right),\qquad V\in C_b(H),
	\end{align}
	and the limit on the right hand-side does not depend on $\sigma\in \Lambda(\gamma,M)$. Here and throughout the rest of the paper, we fix the positive parameters $\gamma,M>0$.
	
	Let us now turn to the justification of the Donsker--Varadhan variational formula \eqref{MR4}. Let $\VV$ be the space of all functions $V: H\to \R$ for which there exists an integer $N\ge 1$ and a function $F\in L_b(H)$ such that
	\[V(u)=F(\ppP_Nu),\qquad u\in H.\]
	Here, $\ppP_N$ is the orthogonal projection in $H$ onto the space $H_N$ spanned by $e_1,e_2,\ldots,e_N$. We say that $\{V_n\}_{n\ge 1}\subset C_b(H)$ buc-converges to $V\in C_b(H)$, if 
	\begin{align}
		\label{buc}\sup_{n\ge 1}\|V_n\|_{\infty}<\infty,\qquad\lim_{n\to+\infty}\|V-V_n\|_{L^{\infty}(K)}=0
	\end{align}
	for any compact set $K$ in $H$. It is known that any function $V\in C_b(H)$ can be buc-approximated by functions in $\VV$, cf. Section 5.6 in \cite{JNPS18}. 

    We need the following auxiliary result, which shows that the rate function $I$ can be recovered from the pressure $Q$ via the Legendre transform, but with the supremum taken over the class $\VV$.
	\begin{lemma}\label{lemma1}
		Let $I$ be the rate function given by \eqref{lvl2-1}. Then,
		\begin{align}\label{lvl2-3}
			I(\lambda)=\sup_{V\in \VV}\left(\int_HV\mathrm{d}\lambda-Q(V)\right)
		\end{align}
		for any $\lambda\in\PP(H)$.
	\end{lemma}
	\begin{proof}
		By definition, we have 
		\begin{align}\label{lvl2-4}
			I(\lambda)\ge\sup_{V\in \VV}\left(\int_HV\mathrm{d}\lambda-Q(V)\right).
		\end{align}
		To prove the inequality in the other direction, let us fix any $\lambda\in \PP(H)$ and $V\in C_b(H)$. Let $\{V_n\}_{n\ge 1}\subset \VV$ be a sequence buc-converging to $V$. Then, 
		\begin{align}
		    \label{lvl2-4-1}\int_HV\dd\lambda=\lim_{n\to+\infty}\int_HV_n\dd\lambda.
		\end{align}
		We claim that 
		\begin{align}\label{lvl2-5}
			Q(V)\ge \limsup_{n\to+\infty} Q(V_n).
		\end{align}
		Indeed, by combining \eqref{A2} with Lemma 3.2 in \cite{JNPS18}, we infer that the family of empirical measures $\{\zeta_t\}_{t>0}$, defined in \eqref{MR1}, is exponentially tight, that is, for any $a>0$, there is a compact set $\KK_a\subset \PP(H)$ such that
		\begin{align}
			\label{lvl2-7}
			\limsup_{t\to+\infty}\frac{1}{t}\log\PPPPPP_{\sigma}\{\zeta_t\in \KK_{a}^c\}\le-a, \qquad \forall \sigma\in \Lambda(\gamma,M).
		\end{align}
        Moreover, by the Prokhorov theorem, for any $\epsilon>0$, there exists a compact set $K_{a,\epsilon}\subset H$ satisfying
        \begin{align}
            \label{lvl2-6}
            \mu(K_{a,\epsilon})\ge 1-\epsilon
        \end{align}
        for any $\mu\in\KK_a$. Let us take 
		\begin{align}\label{lvl2-8}
			a:=\sup_{n\ge 1}\|V_n\|_{\infty}+\|V\|_{\infty}+|Q(V)|.
		\end{align}
		Then, 
		\begin{align}\label{lvl2-11}
			Q(V_n)&=\lim_{t\to+\infty}\frac{1}{t}\log\E_{\sigma}\exp\left(t\langle V_n,\zeta_t\rangle\right)\notag\\
            &=\lim_{t\to+\infty}\frac{1}{t}\log\E_{\sigma}\left((\I_{\{\zeta_t\in \KK_a\}}+\I_{\{\zeta_t\in \KK_a^c\}})\exp\left(t\langle V_n,\zeta_t\rangle\right)\right)\notag\\&
            \le \max\{I_1,I_2\},
		\end{align}
		where 
		\begin{align*}
			I_1&:=\limsup_{t\to+\infty}\frac{1}{t}\log\E_{\sigma}\left(\I_{\{\zeta_t\in \KK_a\}}\exp\left(t\langle V_n,\zeta_t\rangle\right)\right),\\
			I_2&:=\limsup_{t\to+\infty}\frac{1}{t}\log\E_{\sigma}\left(\I_{\{\zeta_t\in \KK_a^c\}}\exp\left(t\langle V_n,\zeta_t\rangle\right)\right).
		\end{align*}
        To estimate $I_1$, we use \eqref{lvl2-6} and \eqref{lvl2-8} to get 
		\begin{align*}
			\E_{\sigma}&\left(\I_{\{\zeta_t\in \KK_a\}}\exp\left(t\langle V_n,\zeta_t\rangle\right)\right)\\&\qquad\qquad\le \E_{\sigma}\left(\I_{\{\zeta_t\in \KK_a\}}\exp\left(t\langle \I_{K_{a,\epsilon}}V_n,\zeta_t\rangle+t\langle \I_{K^c_{a,\epsilon}}V_n,\zeta_t\rangle\right)\right)\\&\qquad\qquad\le \E_{\sigma}\left(\I_{\{\zeta_t\in \KK_a\}}\exp\left(t\langle \I_{K_{a,\epsilon}}V_n,\zeta_t\rangle+a\epsilon t\right)\right)\\
			&\qquad\qquad\le e^{a\epsilon t}\E_{\sigma}\left(\I_{\{\zeta_t\in \KK_a\}}\exp\left(t\langle \I_{K_{a,\epsilon}}(V_n-V),\zeta_t\rangle+t\langle \I_{K_{a,\epsilon}}V,\zeta_t\rangle\right)\right)\\&\qquad\qquad\le e^{a\epsilon t}e^{t\|V_n-V\|_{L^{\infty}(K_{a,\epsilon})}}\E_{\sigma}\left(\I_{\{\zeta_t\in \KK_a\}}\exp\left(t\langle \I_{K_{a,\epsilon}}V,\zeta_t\rangle\right)\right)\\&\qquad\qquad\le e^{a\epsilon t}e^{t\|V_n-V\|_{L^{\infty}(K_{a,\epsilon})}}\E_{\sigma}\left(\I_{\{\zeta_t\in \KK_a\}}\exp\left(t\langle V,\zeta_t\rangle-t\langle\I_{K_{a,\epsilon}^c} V,\zeta_t\rangle\right)\right)\\&\qquad\qquad\le e^{2a\epsilon t}e^{t\|V_n-V\|_{L^{\infty}(K_{a,\epsilon})}}\E_{\sigma}\left(\exp\left(t\langle V,\zeta_t\rangle\right)\right),
		\end{align*}
		which implies
		\begin{align}
			\label{lvl2-9}
			I_1\le 2a\epsilon+\|V_n-V\|_{L^{\infty}(K_{a,\epsilon})}+Q(V).
		\end{align}
		As for $I_2$, we apply \eqref{lvl2-7} and \eqref{lvl2-8} to have 
		\begin{align}
			\label{lvl2-10}
			I_2&\le \sup_{n\ge 1}\|V_n\|_{\infty}+\lim_{t\to+\infty}\frac{1}{t}\log\PPPPPP_{\sigma}\left\{\zeta_t\in \KK_a^c\right\}\notag\\&\le \sup_{n\ge 1}\|V_n\|_{\infty}-a\le -|Q(V)|\le Q(V).
		\end{align}
		Plugging \eqref{lvl2-9} and \eqref{lvl2-10} into \eqref{lvl2-11}, we obtain
		\[Q(V_n)\le  2a\epsilon+\|V_n-V\|_{L^{\infty}(K_{a,\epsilon})}+Q(V)\]
		for any $n\ge1$ and $\epsilon>0$. Therefore, by first taking $n\to+\infty$ and then $\epsilon\to0^+$, we arrive at \eqref{lvl2-5} as claimed. This, together with \eqref{lvl2-4-1}, implies
		\[\int_H V\dd\lambda-Q(V)\le \liminf_{n\to+\infty}\left(\int_HV_n\dd \lambda-Q(V_n)\right)\le\sup_{V\in \VV}\left(\int_HV\mathrm{d}\lambda-Q(V)\right),\]
		thus completing the proof.
	\end{proof}
	For $V\in C_b(H)$, we define the Feynman--Kac semigroup by
	\begin{align}
		\label{MR6}
		\PPPP_t^V f(u):=\mathbb{E}_u\left(f(u_t)\exp\left(\int_0^t V(u_s)\dd s\right)\right),\qquad f\in C_b(H).
	\end{align} 
    Let $\{\PPPP^{V*}_t\}_{t\ge0}$ be its dual semigroup, and let $C_{\www,\kappa}(H)$ be defined by \eqref{B0}. Here and throughout this section, we shall always take $\kappa$ as in Proposition~\ref{propositionB1}. We need the following Duhamel-type formula.
	\begin{lemma}\label{lemma3}
		Let $V\in C_b(H)$ and $f\in C_{\www,\kappa}(H)$. Then,
		\[\PPPP_t^Vf(u)-\PPPP_t f(u)=\int_0^t\PPPP_{t-s}(V\PPPP^V_sf)(u)\dd s\]
		for any $t\ge 0$ and $u\in H$.
	\end{lemma}
	\begin{proof}
		By the Newton--Leibniz formula, the Fubini theorem, and the Markov property, we derive
		\begin{align*}
			\PPPP_t^Vf(u)-\PPPP_t f(u)&=\E_u\left(f(u_t)\left(\exp\left(\int_0^tV(u_s)\dd s\right)-1\right)\right)\\&=\E_u\left(f(u_t)\int_0^t\left(\exp\left(\int_r^tV(u_s)\dd s\right)V(u_r)\right)\dd r\right)\\&=\int_0^t\E_u\left(f(u_t)\exp\left(\int_r^tV(u_s)\dd s\right)V(u_r)\right)\dd r\\&=\int_0^t\E_u\left(\E_{u_r}\left(f(u_{t-r})\exp\left(\int_0^{t-r}V(u_s)\dd s\right)\right)V(u_r)\right)\dd r\\&=\int_0^t\PPPP_r(V\PPPP_{t-r}^Vf)(u)\dd r=\int_0^t\PPPP_{t-s}(V\PPPP^V_sf)(u)\dd s.
		\end{align*}
		Thus, the proof is complete.
	\end{proof}
	By Proposition~\ref{propositionB1}, for each $V\in C_b(H)$, $\{\PPPP^V_t\}_{t\ge0}$ forms a semigroup of bounded linear operators in $C_{\www,\kappa}(H)$. A number $c_V\in\R$ is said to be an eigenvalue of $\{\PPPP^V_t\}_{t\ge0}$, if there is a function $h_V\in C_{\www,\kappa}(H)$ (called eigenfunction) such that
	\begin{align}\label{lvl2-12}\PPPP^V_t h_V=c_V^th_V\end{align}
	for any $t>0$. Using the argument in \cite{Ner19} (see Theorem 2.1) combined with the $H^1$-doubly logarithmic moment estimate \eqref{A5}, we see that for any $V\in\VV$, 
	\begin{align}
		\label{lvl2-13}c_V:=\exp({Q(V)})
	\end{align}
	is the unique principal eigenvalue of $\{\PPPP^V_t\}_{t\ge0}$, and there is a positive eigenfunction $h_V$, which is unique up to multiplying a constant, associated with $c_V$.
	
	Now let us prove the Donsker--Varadhan variational formula \eqref{MR4}.
    \begin{proof}[Proof of \eqref{MR4}] The proof uses several auxiliary results from Appendix \ref{AppenB} and is divided into three steps.
    
	\noindent\textit{Step 1.} First, we claim that 
	\begin{align}
		\label{lvl2-27}
		I(\lambda)\ge \sup_{\substack{f\in \mathcal{D}(\LLLL)\\ \inf_H f>0}}\int_H\frac{-\LLLL f}{f}\dd\lambda,
	\end{align}
	where $\LLLL$ is the generator of the Markov semigroup $\{\PPPP_t\}_{t\ge0}$, and $\DD(\LLLL)$ is defined by \eqref{MR5}. Indeed, let $f\in\DD(\LLLL)$ be such that $\inf_H f>0$. Then, for $V:=\frac{-\LLLL f}{f}\in C_b(H)$, we have 
	\begin{align}\label{lvl2-15}\int_0^t\PPPP_{t-s}(Vf)\dd s=-\int_0^t\PPPP_{t-s}(\LLLL f)\dd s=f-\PPPP_t f.\end{align}
	On the other hand, using Lemma~\ref{lemma3}, we have 
	\begin{align}
		\label{lvl2-16}
		\PPPP_t^Vf-\PPPP_tf=\int_0^t\PPPP_{t-s}(V\PPPP_s^Vf)\dd s.
	\end{align}
	Taking difference between \eqref{lvl2-15} and \eqref{lvl2-16}, we conclude 
	\begin{align*}
		\|\PPPP_t^V f-f\|_{\infty}&\le \int_0^t\|\PPPP_{t-s}(V(\PPPP_s^V f-f))\|_{\infty}\dd s\\&\lesssim_{V}\int_0^t\|\PPPP_s^V f-f\|_{\infty}\dd s.
	\end{align*}
	An application of the Gronwall inequality yields 
	\begin{align}
		\label{lvl2-21}
		\PPPP_t^V f=f,\qquad \forall t\ge 0.
	\end{align}
	Combining this with the definition \eqref{lvl2-2} and the fact that $f\in C_b(H)$ satisfies $\inf_H f>0$, we obtain 
	\begin{align*}
		Q(V)\le \frac{1}{\inf_H f} \left(\lim_{t\to+\infty}\frac{1}{t}\log \PPPP_t^Vf\right)=\frac{1}{\inf_H f} \left(\lim_{t\to+\infty}\frac{1}{t}\log f\right)=0,
	\end{align*}
	and 
	\begin{align*}
		Q(V)\ge \|f\|_{\infty}^{-1} \left(\lim_{t\to+\infty}\frac{1}{t}\log \PPPP_t^Vf\right)=\|f\|_{\infty}^{-1} \left(\lim_{t\to+\infty}\frac{1}{t}\log f\right)=0,
	\end{align*}
	which implies $Q(V)=0$ and thus
	\[\int_H \frac{-\LLLL f}{f}\dd \lambda=\int_{H}V\dd \lambda= \int_{H}V\dd \lambda-Q(V)\le I(\lambda).\]
	This leads to \eqref{lvl2-27} as claimed.
	
	\noindent\textit{Step 2.} It remains to prove the inequality \eqref{lvl2-27} in the other direction. To this end, we claim that 
	\begin{align}\label{lvl2-24}
		I(\lambda)\le \sup\left\{\int_{H}\frac{-\LLLL f}{f}\dd \lambda\Big|f\in\DD_{\www,\kappa}(\LLLL),\ f>0,\ \frac{\LLLL f}{f}\in C_b(H)\right\},
	\end{align}
	where $\DD_{\www,\kappa}(\LLLL)$ is defined by \eqref{B4} and $f>0$ is understood as $f(u)>0$ for any $u\in H$. To see this, let us notice that for each $V\in C_b(H)$, 
	\[Q(V-Q(V))=0.\]
	Then, from Lemma~\ref{lemma1}, we derive 
	\[I(\lambda)=\sup_{\substack{V\in\VV\\Q(V)=0}}\int_HV\dd\lambda.\]
	Let $\tilde I(\lambda)$ denote the right hand-side of \eqref{lvl2-24}. Let $V\in \VV$ be such that $Q(V)=0$, then by \eqref{lvl2-13}, the unique principal eigenvalue of $\{\PPPP_t^V\}_{t\ge0}$ is given by $c_V=1.$ Let $h_V$ be the positive eigenfunction associated with $c_V$. Combining Lemma~\ref{lemma3}, $c_V=1$, and \eqref{lvl2-12}, we derive
	\begin{align*}
		\int_0^t\PPPP_{s}(Vh_V)\dd s=\int_0^t\PPPP_{t-s}(V\PPPP^V_sh_V)\dd s=\PPPP_t^Vh_V-\PPPP_th_V=h_V-\PPPP_th_V,
	\end{align*}
	which, by the definition \eqref{B4}, implies that $h_V\in \DD_{\www,\kappa}(\LLLL)$ and
	\[\LLLL h_V=-Vh_V.\]
	Therefore,
	\[\int_H V\dd \lambda=\int_H\frac{-\LLLL h_V}{h_V}\dd\lambda\le \tilde I(\lambda),\]
	which implies \eqref{lvl2-24} as claimed.

	\noindent\textit{Step 3.} Now, it suffices to show 
	\begin{align}
		\label{lvl2-17}
		\tilde{I}(\lambda)\le\sup_{\substack{f\in\DD(\LLLL)\\\inf_H f>0}}\int_{H}\frac{-\LLLL f}{f}\dd \lambda.
	\end{align}
	Let $\RR_{\alpha}$ be the resolvent operator defined by \eqref{B1}, and let $f\in\DD_{\www,\kappa}(\LLLL)$ satisfy 
	\begin{align*}
		f>0,\qquad \frac{\LLLL f}{f}\in C_b(H).
	\end{align*} 
    From Corollary~\ref{corollaryB1}, it follows that 
	\[\int_H\frac{-\LLLL f}{f}\dd \lambda\le \lim_{\alpha\to +\infty}\liminf_{N\to+\infty}\int_H-\frac{\LLLL\RR_{\alpha}f_N}{\RR_{\alpha }f_N}\dd \lambda,\]
    where $f_N:=(f\wedge N)\vee \frac{1}{N}$. Notice that for any $\alpha>0$ and $N\ge 1$,
	\[\RR_{\alpha}f_N\in \DD(\LLLL),\qquad \inf_H\RR_{\alpha}f_N>0,\]
    due to Proposition~\ref{propositionB3}. Therefore, we have 
    \[\int_H\frac{-\LLLL f}{f}\dd \lambda\le \sup_{\substack{g\in\DD(\LLLL)\\\inf_H g>0}}\int_H \frac{-\LLLL g}{g}\dd\lambda.\]
    This implies \eqref{lvl2-17} and thus completes the proof.  
    \end{proof}

    \section{Exponential tightness}\label{ET}
    Let $\{\bmzeta_t^T\}_{t\ge0}$ be the empirical measure defined in \eqref{MR7}. The main result of this section is Proposition \ref{proposition4-4}, which shows the exponential tightness of $\{\bmzeta^T_{t_n}\}_{n\ge 1}$ for any sequence $\{t_n\}_{n\ge 1}$ such that $\lim_{n\to+\infty}t_n=+\infty$.
	
    As discussed in Section~\ref{scheme}, to obtain this result, we consider the periodized empirical measures $\{\bmzeta^{\per}_t\}_{t>0}$ defined by \eqref{MR10}. Let $\DDDD^{\R}$ denote the space of functions $u:\R\to H$ with discontinuities only of the first kind and normalized to be right-continuous. Let $\PP_s(\DDDD^\R)$ denote the space of all Borel probability measures which are invariant under the shift operators $\{\theta_t\}_{t>0}$. Notice that $\{\bmzeta^{\per}_t\}_{t>0}$ is a family of random probability measures valued in $\PP_s(\DDDD^\R)$. To see this, let us take any $s>0$ and $A\in\mathscr{B}(\DDDD^\R)$. Utilizing the periodicity of $u_{\per,t}$, we derive
	\begin{align*}
		\theta^{*}_s\bmzeta^{\per}_{t}(A)&=\frac{1}{t}\int_0^t\I_{\{\theta_r u_{\per,t}\in\theta^{-1}_s A\}}\dd r
		=\frac{1}{t}\int_0^t\I_{\{\theta_{r+s} u_{\per,t}\in A\}}\dd r\\&=\frac{1}{t}\int_{s}^{t+s}\I_{\{\theta_{r} u_{\per,t}\in A\}}\dd r=\frac{1}{t}\int_{0}^{t}\I_{\{\theta_{r} u_{\per,t}\in A\}}\dd r=\bmzeta^{\per}_{t}(A).
	\end{align*}
    Let us introduce the Donsker--Varadhan entropy
	\begin{align}\label{ET1}
		\BH(\bmlambda)\!:=\begin{cases}
			\int_{\DDDD^{\R}}\ent\left(\pi_{[0,1]}^*\bmlambda(u_{(-\infty,0]},\cdot)\Big| \pi_{[0,1]}^*\mathbb{P}_{u(0)}\right)\bmlambda(\dd u_{(-\infty,\infty)})\!\!&\mbox{if $\bmlambda\in \PP_s(\DDDD^\R)$},\\
			+\infty & \mbox{otherwise},
		\end{cases}
	\end{align}
    where and in what follows, for $r\in\R$ and $\JJ\subset\R$, $\pi^*_r$ and $\pi_{\JJ}^*$ are the push-forward operators defined in \eqref{pushforward1} and \eqref{pushforward2}, respectively. We gather basic properties of $\BH$ in the Appendix \ref{propertyDVE}.
	\subsection{LD upper bound for periodized empirical measures}
	In this subsection, we prove the following uniform LD upper bound for $\{\bmzeta^{\per}_t\}_{t>0}$.
	\begin{proposition}\label{proposition4-1}
		For any closed set $F\subset \PP(\DDDD^\R)$, we have 
		\begin{align}
			\label{ET2}
			\limsup_{t\to+\infty}\frac{1}{t}\log\left(\sup_{\sigma\in\Lambda(\gamma,M)}\PPPPPP_{\sigma}\{\bmzeta^{\per}_t\in F\}\right)\le-\inf_{\bmlambda\in F}\BH(\bmlambda).
		\end{align}
	\end{proposition}
	To establish this proposition, we need the following auxiliary lemma, which provides an estimate analogous to the exponential tightness. The proof is carried out by lifting the level-2 exponential growth estimate \eqref{A2} to the trajectory level.
	\begin{lemma}\label{lemma4-1}
		For any $a>0$, there is a closed set $F_a\subset\PP_s(\DDDD^{\R})$ such that the family $\{\pi^*_0\bmlambda| \bmlambda\in F_a\}$ is tight in $\PP(H)$ and 
		\begin{align}
		    \label{lvl2-25-0}\limsup_{t\to+\infty}\frac{1}{t}\log\left(\sup_{\sigma\in\Lambda(\gamma,M)}\PPPPPP_{\sigma}\{\bmzeta^{\per}_t\in F_a^c\}\right)\le-a.
		\end{align}
	\end{lemma}
	\begin{proof}
		Let us define 
		\begin{align}\label{lvl2-25}
			\Phi(u):=\kappa\|u\|^2_1,
		\end{align}
		where $\kappa$ is as in Proposition~\ref{propositionB1}. By the definition \eqref{MR10},
		\[\int_0^t\Phi(u_s)\dd s=t\int_{\DDDD^{\R}}\Phi(u(0))\bmzeta^{\per}_t(\dd u_{(-\infty,\infty)}).\]
  		We define $\MU_t:=(\bmzeta^{\per}_t)^*\PPPPPP_{\sigma}$, so $\MU_t$ is a probability measure on the space $\PP_s(\DDDD^{\R})$ for any $t>0$. Then,
		\[\E_{\sigma}\exp\left(\int_0^t\Phi(u_s)\dd s\right)=\int_{\PP_s(\DDDD^\R)}\exp(t\langle\Phi,\pi^*_0\bmlambda\rangle)\MU_t(\dd \bmlambda).\]
        This, together with the level-2 exponential growth estimate \eqref{A2} and the Chebyshev inequality, implies that there is $C>0$ depending on $M,h$, and $\BB_0:=\sum_{j=1}^{\infty}b_j^2$ such that 
		\begin{align*}
			Ce^{Ct}&\ge\E_{\sigma}\exp\left(\int_0^t\Phi(u_s)\dd s\right)\ge\int_{\PP_s(\DDDD^\R)}\exp(tl\pi_0^*\bmlambda(K_l^c))\MU_t(\dd \bmlambda)\\&\ge \exp\left(\frac{tl}{n}\right)\MU_t\left\{\bmlambda\in \PP_s(\DDDD^{\R})\Big|\pi_0^*\bmlambda(K_l^c)> \frac{1}{n}\right\},
		\end{align*}
		where $K_l:=\{u| \Phi(u)\le l\}$. By choosing $l_{n,a}:=Cn+an^2$, we get
		\begin{align*}
			\MU_t\left\{\bmlambda\in \PP_s(\DDDD^{\R})\Big|\pi_0^*\bmlambda(K_{l_{n,a}}^c)>\frac{1}{n}\right\}\le Ce^{-ant}.
		\end{align*}
		Let us introduce
		\[F_a:=\bigcap_{n=1}^{\infty}\left\{\bmlambda\in \PP_s(\DDDD^{\R})\Big|\pi_0^*\bmlambda(K_{l_{n,a}}^c)\le \frac{1}{n}\right\}.\]
		Then, we have
		\[ \PPPPPP_{\sigma}\{\bmzeta^{\per}_t\in F_a^c\}=\MU_t(F_a^c)\le C\frac{e^{-at}}{1-e^{-at}},\]
        which implies \eqref{lvl2-25-0}. Using the Portmanteau theorem and the Prokhorov theorem, we see that $F_a$ satisfies the desired properties. This completes the proof.
	\end{proof}
	Let us turn to the proof of Proposition~\ref{proposition4-1}.
	\begin{proof}[Proof of Proposition~\ref{proposition4-1}] Let $F\!\subset\!\PP(\DDDD^\R)$ be a closed set. By applying Lemma~\ref{lemma4-1} and Proposition~\ref{propositionC4}, we see that for each $a>0$,
		\begin{align*}
			\limsup_{t\to+\infty}\frac{1}{t}&\log\left(\sup_{\sigma\in\Lambda(\gamma,M)}\PPPPPP_{\sigma}\{\bmzeta^{\per}_t\in F\}\right)\\&=\limsup_{t\to+\infty}\frac{1}{t}\log\left(\sup_{\sigma\in\Lambda(\gamma,M)}\PPPPPP_{\sigma}\left\{\bmzeta^{\per}_t\in F\mcap \PP_s(\DDDD^\R)\right\}\right)\\&\le\max\left\{\limsup_{t\to+\infty}\frac{1}{t}\log\left(\sup_{\sigma\in\Lambda(\gamma,M)}\PPPPPP_{\sigma}\left\{\bmzeta^{\per}_t\in F\mcap F_a\right\}\right),-a\right\}\\&\le\max\left\{-\inf_{\bmlambda\in F}\BH(\bmlambda),-a\right\}.
		\end{align*}
		By choosing $a$ sufficiently large, we obtain \eqref{ET2}.
	\end{proof}
	\subsection{Contraction inequality and exponential tightness}
	In this subsection, we establish the exponential tightness of $\{\bmzeta^T_t\}_{t>0}$. To this end, we first prove the goodness of the Donsker--Varadhan entropy $\BH$ defined in \eqref{ET1} by using the following contraction inequality.
	\begin{lemma}\label{lemma4-2}
		Let $I$ be the level-2 rate function given by \eqref{MR4}. Then,
		\begin{align}
			\label{contraction_inequality}
			I(\lambda)\le \inf\{\BH(\bmlambda)|\pi_0^*\bmlambda=\lambda\},\qquad \forall \lambda\in\PP(H).
		\end{align}
	\end{lemma}
	\begin{proof}According to the definition \eqref{ET1}, it suffices to establish 
		\[I(\lambda)\le \inf\{\BH(\bmlambda)|\bmlambda\in\PP_s(\DDDD^\R), \pi_0^*\bmlambda=\lambda\}.\]
		To see this, we take $\bmlambda\in\PP_s(\DDDD^\R)$ and $f\in \DD(\LLLL)$ such that $\inf_Hf>0$. Notice that
		\[\Psi_t(u_{(-\infty,\infty)}):=\log\frac{f(u(t))}{\PPPP_tf(u(0))},\qquad \Psi_t: \DDDD^\R\to\R\]
        is a bounded $\mathscr{F}^0_t$-measurable function satisfying
		\[\log \left(\int_\DDDD e^{\Psi_t}\PPPPPP_{u}\right)=\log\left(\frac{\int_\DDDD f(u(t))\dd\PPPPPP_u}{\PPPP_t f(u)}\right)=0.\]
		From the definition \eqref{C2} and the relation \eqref{C2-1}, we derive
		\[\int_{\DDDD^\R}\Psi_t \dd\bmlambda\le t\BH(\bmlambda).\]
		This, together with the translation-invarianc of $\bmlambda$, implies 
		\begin{align*}
			t\BH(\bmlambda)&\ge \int_{\DDDD^\R}\log\frac{f(u(t))}{\PPPP_tf(u(0))}\bmlambda(\dd u_{(-\infty,\infty)})\\&=\int_{\DDDD^\R}\log f(u(t))\bmlambda(\dd u_{(-\infty,\infty)})-\int_{\DDDD^\R}\log\PPPP_tf(u(0))\bmlambda(\dd u_{(-\infty,\infty)})\\&=\int_{H}\log f(u)\lambda(\dd u)-\int_{H}\log\PPPP_tf(u)\lambda(\dd u),
		\end{align*}
		where $\lambda:=\pi_0^*\bmlambda$. Therefore, we have
		\begin{align*}
			\BH(\bmlambda)&\ge -\lim_{t\to0^+}\frac{1}{t}\int_{H}\int_0^t\frac{\PPPP_s\LLLL f}{\PPPP_sf}\dd s\dd\lambda\\&=-\lim_{t\to0^+}\frac{1}{t}\int_0^t\int_{H}\frac{\PPPP_s\LLLL f}{\PPPP_sf}\dd\lambda\dd s=\int_{H}\frac{-\LLLL f}{f}\dd\lambda,
		\end{align*}
		which implies $\BH(\bmlambda)\ge I(\lambda)$ and thus \eqref{contraction_inequality}.
	\end{proof}
    Now we prove the goodness of the Donsker--Varadhan entrophy $\BH$.
	\begin{proposition}\label{proposition4-2}
		For any $\alpha\ge 0$, $A_\alpha:=\{\bmlambda\in\PP(\DDDD^\R)| \BH(\bmlambda)\le \alpha\}$ is compact.
	\end{proposition}
	\begin{proof}The proof is divided into three steps.
    
    \noindent\textit{Step 1: Tightness of marginals.} We claim that the family $A_{\MM,\alpha}:=\{\pi_0^*\bmlambda|\bmlambda\in A_{\alpha}\}$ is tight in $\PP(H)$. In fact, by Lemma~\ref{lemma4-2}, for any $\bmlambda\in A_\alpha$, we have $I(\lambda)\le \alpha$, where $\lambda:=\pi_0^*\bmlambda$. Therefore,
		\[A_{\MM,\alpha}\subset\{\lambda\in\PP(H)| I(\lambda)\le \alpha\}.\]
		Utilizing the goodness of $I$, we obtain the tightness of $A_{\MM,\alpha}$ as claimed.
		
		\noindent \textit{Step 2: Construction of an auxiliary test function.} Let us take any $\epsilon,\delta\in(0,1)$. By the tightness of $A_{\MM,\alpha}$, we find a compact set $K\subset H$ such that 
		\begin{align}
			\lambda(K^c)< \frac{\delta}{2}\label{ET5},\qquad \forall\lambda\in A_{\MM,\alpha}.
		\end{align}
		Combining the compactness of $K$ and the continuity of the map $u\mapsto \PPPPPP_u$, see Proposition 2.4.7 in \cite{KS12}, we derive the tightness of $\{\PPPPPP_u\}_{u\in K}$ in $\PP(\DDDD^{[0,1]})$. Thus, there is a compact set $\KK\subset \DDDD^{[0,1]}$ satisfying $\PPPPPP_u(\KK^c)<\epsilon$ for any $u\in K$. We introduce the following auxiliary test function
		\[\Psi(u_{(-\infty,\infty)}):=a\int_0^1\I_{K}(u(s))\I_{\KK^c}(u_{[s,s+1]})\dd s,\qquad\Psi: \DDDD^\R\to\R,\]
        where $a>0$. Then, for any $u\in H$, 
		\begin{align}\label{ET3}
			\E_{u}\exp\left(a\I_{K}(u(0))\I_{\KK^c}(u_{[0,1]})\right)&\le \E_{u}\exp\left(a\I_{\KK^c}(u_{[0,1]})\right)\notag\\&=(e^{a}-1)\mathbb{P}_u(\KK^c)+1\le (e^{a}-1)\epsilon+1.
		\end{align}
         This, together with the Jensen inequality and the Markov property, implies
        \begin{align}
		    \label{ET3-1}
            \E_ue^{\Psi}&\le  \int_0^1\E_u \exp\left(a\I_{K}(u(s))\I_{\KK^c}(u_{[s,s+1]})\right)\dd s\notag\\&= \int_0^1\E_u \left(\E_{u(s)} \exp\left(a\I_{K}(u(0))\I_{\KK^c}(u_{[0,1]})\right)\right)\dd s\le (e^{a}-1)\epsilon+1
        \end{align}
        for any $u\in H$.
        
        \noindent\textit{Step 3: Tightness of} $A_\alpha$.  Let us fix any $\bmlambda\in A_\alpha$. Combining \eqref{ET3-1} with the definition \eqref{C2} of $\bar\BH$ and the relation \eqref{C2-1}, we have
		\begin{align}\label{ET4}
			\int_{\DDDD^{\R}}\Psi\dd\bmlambda&\le \bar{\BH}(2 ,\bmlambda)+\int_{\DDDD^\R}\log\left(\int_{\DDDD}e^{\Psi}\dd\PPPPPP_{u(0)}\right)\bmlambda(\dd u_{(-\infty,\infty)})\notag\\& \le 2\alpha  +\log((e^{a}-1)\epsilon+1).
		\end{align}
        On the other hand, by the Fubini theorem and the translation-invariance of $\bmlambda$, we have 
		\begin{align*}
			\int_{\DDDD^{\R}}\Psi\dd\bmlambda&=a\int_0^1\int_{\DDDD^{\R}}\I_{K}(u(s))\I_{\KK^c}(u_{[s,s+1]})\dd\bmlambda\dd s\\&=a\int_{\DDDD^{\R}}\I_{K}(u(0))\I_{\KK^c}(u_{[0,1]})\dd\bmlambda\ge a\pi^*_{[0,1]}\bmlambda(\KK^c)-a\lambda(K^c),
		\end{align*}
        where $\lambda=\pi_0^*\bmlambda$. This, together with \eqref{ET5} and \eqref{ET4}, implies 
		\[\pi^*_{[0,1]}\bmlambda(\KK^c)< \frac{\delta}{2}+\frac{1}{a}(2\alpha +\log((e^{a }-1)\epsilon+1)).\]
		Choosing $\epsilon:=\frac{1}{e^{a}-1}$	with $a$ sufficiently large so that $\frac{2\alpha +\log 2}{a  }< \frac{\delta}{2}$, we see that
		\[\pi^*_{[0,1]}\bmlambda(\KK^c)<\delta.\]
    Therefore, the projections of measures in $A_\alpha$ to $\DDDD^{[0,1]}$ form a tight family. This, together with Lemma~\ref{lemmaC5-0}, leads to the tightness of $A_\alpha$. Finally, utilizing the lower semi-continuity of $\BH$, see Proposition \ref{lscDVE}, combined with the Prokhorov theorem, we obtain the compactness of $A_\alpha$.
	\end{proof}
	
	Combining Propositions~\ref{proposition4-1} and \ref{proposition4-2}, we infer that a uniform LD upper bound with the good rate function $\BH$ is valid for the periodized empirical measures $\{\bmzeta^{\per}_t\}_{t>0}$. This, together with the exponential equivalence between $\{\pi_{[0,\infty)]}^*\bmzeta^{\per}_t\}_{t>0}$ and $\{\bmzeta_t\}_{t>0}$, ensures the exponential tightness of $\{\bmzeta^T_{t_n}\}_{n\ge 1}$ for any sequence $\{t_n\}_{n\ge 1}$ such that $\lim_{n\to\infty}t_n=+\infty$, as detailed below.
	\begin{lemma}\label{LemmaET1}
		The families $\{\pi_{[0,\infty)}^*\bmzeta^{\per}_t\}_{t>0}$ and $\{\bmzeta_t\}_{t>0}$ are uniformly exponentially equivalent, that is, for any $\delta>0$,
		\begin{align}\label{ET6}
			\lim_{t\to+\infty}\frac{1}{t}\log\left(\sup_{\sigma\in\Lambda(\gamma,M)}\PPPPPP_{\sigma}\left\{\|\pi_{[0,\infty)}^*\bmzeta^{\per}_t-\bmzeta_t\|^*_{L(\DDDD)}>\delta\right\}\right)=-\infty.
		\end{align}
	\end{lemma}
	\begin{proof} Let $L_b(\DDDD)$ denote the space of all bounded Lipschitz functions on $\DDDD$ with the norm defined by \eqref{lipnorm}. For any $\VVV\in L_b(\DDDD)$ such that $\|\VVV\|_{L_b}\le 1$,
		\begin{align*}
			|\langle \VVV,\bmzeta_t\rangle-\langle \VVV,\pi_{[0,\infty)}^*\bmzeta^{\per}_t\rangle|&\le\frac{1}{t}\int_0^t|\VVV(u_{[s,\infty)})-\VVV(\pi_{[s,\infty)}u_{\per,t})|\dd s\\&\le\frac{1}{t}\int_0^t\dist_{\DDDD} (u_{[s,\infty)},\pi_{[s,\infty)}u_{\per,t})\dd s.
		\end{align*}
		Since
		\[u_{[0,\infty)}(s+r)=u_{\per,t}(s+r),\qquad  \forall r\in[0, t-s),\]
		we have 
		\[\dist_{\DDDD}(u_{[s,\infty)},\pi_{[s,\infty)}u_{\per,t})\le \sum_{m=[t-s]+1}^{\infty}\frac{1}{2^m}=\frac{1}{2^{[t-s]}},\]
		where $[x]$ denotes the largest integer not exceeding $x$, and we refer to (16.4) in \cite{Bil1999} for the definition of $\dist_\DDDD$. Hence, we derive
		\begin{align*}
			|\langle V,\bmzeta_t\rangle-\langle V,\pi_{[0,\infty)}^*\bmzeta^{\per}_t\rangle|\le\frac{1}{t}\int_0^t\frac{1}{2^{[s]}}\dd s\le\frac{2\log 2}{t}.
		\end{align*}
		This implies 
		\[\|\pi_{[0,\infty)}^*\bmzeta^{\per}_t-\bmzeta_t\|^*_{L(\DDDD)}<\delta,\qquad \forall t>\frac{2\log 2}{\delta},\]
		and thus the uniform exponential equivalence \eqref{ET6}.
	\end{proof}
	The above exponential equivalence ensures the following uniform LD upper bound for $\{\bmzeta_t\}_{t>0}$ with a good rate function $\II$.
	\begin{proposition}\label{proposition4-3}
		For any closed set $F\subset \PP(X)$, we have 
		\begin{align*}
			\limsup_{t\to+\infty}\frac{1}{t}\log\left(\sup_{\sigma\in\Lambda(\gamma,M)}\PPPPPP_{\sigma}\{\bmzeta_t\in F\}\right)\le-\inf_{\bmlambda\in F}\II(\bmlambda),
		\end{align*}
    where $\II$ is a good rate function satisfying the Donsker--Varadhan entropy formula~\eqref{MR9}.
	\end{proposition}
	\begin{proof} Recall that any probability measure $\bmlambda\in \PP_s(\DDDD)$ can be uniquely extended to a probability measure $\bmlambda\in \PP_s(\DDDD^\R)$. Combining this with Proposition~\ref{proposition4-1} and the contraction principle (Theorem 4.2.1 in \cite{DZ2010}), we see that a uniform LD upper bound holds for $\{\pi_{[0,\infty)}^*\bmzeta^{\per}_t\}_{t>0}$ with a good rate function
    \begin{align*}
		\II(\bmlambda):=\begin{cases}
			\int_{\DDDD^{\R}}\ent\left(\pi_{[0,1]}^*\bmlambda(u_{(-\infty,0]},\cdot)\Big| \pi_{[0,1]}^*\mathbb{P}_{u(0)}\right)\bmlambda(\dd u_{(-\infty,\infty)}) &\mbox{if $\bmlambda\in \PP_s(\DDDD)$},\\
			+\infty & \mbox{otherwise}.
		\end{cases}
	\end{align*} 
    Let $F\subset\PP(\DDDD)$ be a closed subset. For $\delta>0$, we write 
		\[F_{\delta}:=\{\bmlambda\in\PP(\DDDD)| \dist_{\PP(\DDDD)}(\bmlambda,F)\le\delta\}.\]
		Applying Lemma~\ref{LemmaET1}, we have 
		\begin{align*}
			\limsup_{t\to+\infty}\frac{1}{t}\log&\left(\sup_{\sigma\in\Lambda(\gamma,M)}\PPPPPP_{\sigma}\{\bmzeta_t\in F\}\right)&\\&\le\limsup_{t\to+\infty}\frac{1}{t}\log\left(\sup_{\sigma\in\Lambda(\gamma,M)}\PPPPPP_{\sigma}\{\pi^*_{[0,\infty)}\bmzeta^{\per}_t\in F_{\delta}\}\right)\le -\inf_{\bmlambda\in F_{\delta}}\II(\bmlambda).
		\end{align*}
		This, together with Lemma 4.1.6 in \cite{DZ2010}, yields
		\begin{align*}
			\limsup_{t\to+\infty}\frac{1}{t}\log\left(\sup_{\sigma\in\Lambda(\gamma,M)}\PPPPPP_{\sigma}\{\bmzeta_t\in F\}\right)\le -\lim_{\delta\to0^+}\inf_{\bmlambda\in F_{\delta}}\II(\bmlambda)=-\inf_{\bmlambda\in F}\II(\bmlambda).
		\end{align*}
		As $\bmzeta_t$ is supported in $\PP(X)$ for all $t>0$ and $\II$ satisfies \eqref{MR9} when restricted on $\PP(X)$, this completes the proof.
	\end{proof}
	Using the contraction principle again, we see that a uniform LD upper bound holds for $\{\bmzeta_t^T\}_{t>0}$ with a good rate function. It is well-known that in the discrete-time case, an LD upper bound with a good rate function implies the exponential tightness, see Lemma 2.6 in \cite{LS1987}. Therefore, we have the following result.
	\begin{proposition}\label{proposition4-4}
		For each sequence $\{t_n\}_{n\ge 1}$ such that $\lim_{n\to+\infty}t_n=+\infty$, the family $\{\bmzeta^T_{t_n}\}_{n\ge 1}$ is uniformly exponentially tight in $\sigma\in \Lambda(\gamma,M)$, that is, for any $a>0$, there is a compact set $\KK_a$ such that 
		\[\limsup_{n\to\infty}\frac{1}{t_n}\log\left(\sup_{\sigma\in \Lambda(\gamma,M)}\PPPPPP_{\sigma}\{\bmzeta^T_{t_n}\in \KK_a^c\}\right)\le -a.\]
	\end{proposition}

	\section{Large-time asymptotics for the Feynman--Kac semigroup}\label{largetimeFK}
    The goal of this section is to study the large-time behavior of the Feynman--Kac semigroup by verifying the conditions of Proposition~\ref{propositionE2}.
    
	We begin with some definitions. Throughout this section, let us fix a time $T>0$ and introduce a Markov $T$-process associated with the Navier--Stokes system \eqref{I1}: the phase space is $X^{T}:=C([0,T];H)$, and the transition function is defined in the following way.  Let us take $\uu\in X^{T}$, and consider 
	\[\Gamma:=\mcap_{j=1}^n\pi_{t_j}^{-1}(A_j),\]
    where $0\le t_1<\ldots<t_n\le T$, and $A_j\in\mathscr{B}(H)$ for $j=1,2,\ldots,n$. If $t+t_1\ge T$, then the transition function is defined by
    \begin{align}
        \label{FK1-1}
        P_{t}^T(\uu,\Gamma
		):=\PPPPPP_{\uu(T)}\left(\mcap_{j=1}^n\{u_{t+t_j-T}\in A_j\}\right),
    \end{align}
    where $\PPPPPP_{\uu(T)}:=\PPPPPP_{u}|_{u=\uu(T)}$, and $(u_t,\PPPPPP_u)$ is the Markov family associated with \eqref{I1}. If $0\le t+t_1<T$, we define
	\begin{align}
		\label{FK1}
		P_{t}^T(\uu,\Gamma
		):=\prod_{j=1}^k\I_{A_j}(\uu(t+t_j))\PPPPPP_{\uu(T)}\left(\mcap_{j=k+1}^n\{u_{t+t_j-T}\in A_j\}\right),
	\end{align}
	where the integer $k$ is such that $t+t_k< T\le t+t_{k+1}$. By the Kolmogorov theorem, the transition function is well defined. Let $(\uu_t,\PPPPPP_{\uu})$ denote the Markov family corresponding to this   transition function. For $\VVV\in C_b(X^{T})$, we define the Feynman--Kac semigroup by the formula
	\begin{align}\label{FK2}\PPPP_t^{T,\VVV}\fff(\uu):=\E_{\uu}\left(\fff(\uu_t)\exp\left(\int_0^t\VVV(\uu_s)\dd s\right)\right),\qquad \fff\in C_b(X^{T}).\end{align}
    By $\PPPP_{t}^{T,\VVV*}$, we denote the dual semigroup of $\PPPP_{t}^{T,\VVV}$, and for simplicity, we write 
	\begin{align}
	    \label{FK3}\PPPP_t^{T}:=\PPPP_t^{T,0},\qquad \PPPP_t^{T*}:=\PPPP_t^{T,0*}.
	\end{align}

	\subsection{Hyper-exponential recurrence}
	Let us introduce
	\[X^T_{\infty}:=C([0,T];U)\bigcap C^{\frac{1}{4}}([0,T];U^*)\]
	endowed with the norm
	\[\|\uu\|_{X^T_{\infty}}:=\|\uu\|_{C([0,T];U)}+\|\uu\|_{C^{1/4}([0,T];U^*)}.\]
	For any integer $R\ge 1$, let $X^T_R$ denote the closure of $B_{X^T_\infty}(R)$ in $X^T$. According to the Aubin--Lions lemma, $X^T_R$ is a compact subset of $X^T$. Moreover, $X^T_{\infty}$ is dense in $X^T$. Let 
	\begin{align}\label{HER0}\tau(R):=\inf\left\{t\ge0 \big|\uu_t\in X_{R}^T\right\}\end{align}
	be the first hitting time of $X^T_R$. 
	\begin{proposition}\label{propositionHER}
		For any $\kappa>0$, there are integers $m,R\ge 1$ and $C>0$ such that 
		\begin{align}
			\label{HER1}
			\mathbb{E}_{\uu}\exp(\kappa\tau(R))\le C(\|\uu\|_{X^T}^{2m}+1),\qquad \forall \uu\in X^T.
		\end{align}
	\end{proposition}
	\begin{proof}
		\textit{Step 1: Hyper-exponential recurrence in $X^T$.} For $r>0$, let us define
		\begin{align}\label{HER1-1}\tau_0(r):=\inf\{t\ge0| \uu_t\in \overline{B_{X^T}(r)}\}.\end{align}
		We claim that for any $\kappa>0$, there are integers $m,r\ge 1$ and $C>0$ such that 
		\begin{align}
			\label{HER2}
			\mathbb{E}_{\uu}\exp(\kappa\tau_0(r))\le C(\|\uu\|_{X^T}^{2m}+1),\qquad \forall \uu\in X^T.
		\end{align}
		 Indeed, by utilizing \eqref{A4-2}, we have
		\begin{align}
			\label{HER3}
			\E_{\uu}\|\uu_{T+1}\|^{2m}_{X^T}\le 4e^{-m\alpha_1}\|\uu\|^{2m}_{X^T}+C_{m,\BB_0,h,T}.
		\end{align}
		Then, by choosing $m$ sufficiently large so that 
        \begin{align}\label{HER4-1}
            q:=8e^{-m\alpha_1}<1,\qquad e^{\kappa(T+1)}q<1,
        \end{align}
        and setting $r:=\frac{1}{4}e^{m\alpha_1}C_{m,\BB_0,h,T}$, we have
		\begin{align}
			\label{HER4}
			\E_{\uu}\|\uu_{T+1}\|^{2m}_{X^T}\le q(\|\uu\|^{2m}_{X^T}\vee r).
		\end{align}
		Combining this and the Markov property, we derive 
		\begin{align*}
			p_k(\uu):&=\E_{\uu}\left(\I_{\{\tau_0(r)>k(T+1)\}} \|\uu_{k(T+1)}\|^{2m}_{X^T}\right)\\&\le \E_{\uu}\left(\I_{\{\tau_0(r)>(k-1)(T+1)\}} \|\uu_{k(T+1)}\|^{2m}_{X^T}\right)\\&= \E_{\uu}\left(\I_{\{\tau_0(r)>(k-1)(T+1)\}}\E_{\uu_{(k-1)(T+1)}}\|\uu_{T+1}\|^{2m}_{X^T}\right)\\&\le q\E_{\uu}\left(\I_{\{\tau_0(r)>(k-1)(T+1)\}}\left(\|\uu_{(k-1)(T+1)}\|^{2m}_{X^T}\vee r\right)\right)=qp_{k-1}(u),
		\end{align*}
		which implies
		\[p_k(\uu)\le q^k\|\uu\|^{2m}_{X^T}\]
		for any integer $k\ge 1$ and $\uu\in X^T$. Hence, 
		\begin{align*}
			\PPPPPP_{\uu}\{\tau_0(r)>k(T+1)\}\le r^{-2m}p_{k}(\uu) \le r^{-2m}q^k\|\uu\|^{2m}_{X^T}.
		\end{align*}
        Combining this with \eqref{HER4-1}, we see that
		\begin{align*}
			\E_{\uu}\exp(\kappa\tau_0(r))&\le 1+\E_{\uu}\left(\exp(\kappa \tau_0(r))\sum_{k=0}^{\infty}\I_{\{\tau_0(r)\in(k(T+1),(k+1)(T+1)]\}}\right)\\&\le 1+\sum_{k=0}^{\infty}e^{\kappa(k+1)(T+1)}\PPPPPP_{\uu}\{\tau_0(r)>k(T+1)\}\\&\le 1+\sum_{k=0}^{\infty}e^{\kappa(k+1)(T+1)}r^{-2m}q^k\|\uu\|^{2m}_{X^T}\le C_{\kappa,T}(\|\uu\|^{2m}_{X^T}+1).
		\end{align*}
		
		\noindent\textit{Step 2:
  			Auxiliary stopping times and estimates.} Let us introduce
		\[\tilde\tau_0(r):=\tau_0(r)+T+1,\]
		where $\tau_0(r)$ is defined in \eqref{HER1-1}, and 
		\[\tau_n(r):=\inf\{t>\tilde\tau_{n-1}(r)| \uu_t\in \overline{B_{X^T}(r)}\},\qquad \tilde\tau_n(r):=\tau_n(r)+T+1,\qquad n\ge 1.\]
		We define
		\[\tilde n:=\min\{n\in\N| \uu_{\tilde\tau_{ n}(r)}\in X^T_R\}.\]
            Combining \eqref{A5} and Corollary~\ref{corollaryA1}, we get
		\begin{align}
			\label{HER5}
			\E_{\uu}\log\left(1+\log\left(1+\|\uu_{T+1}\|_{X^T_{\infty}}\right)\right)\lesssim_{\BB_1,h,T}\|\uu\|^2_{X^T}+1.
		\end{align}
		This, together with the Chebyshev inequality, implies
		\begin{align*}
			\PPPPPP_{\uu}\left\{\uu_{T+1}\in X^T_R\right\}&\ge1-\PPPPPP_{\uu}\left\{\uu_{T+1}\notin B_{X^T_\infty}(R)\right\}\ge1-\frac{C_{\BB_1,h,T}(\|\uu\|^2_{X^T}+1)}{\log(1+\log(1+R))}.
		\end{align*}
		Therefore, for any $p\in (0,1)$, $r\ge 1$ and $\uu\in \overline{B_{X^T}(r)}$, there is $R=R(p,r)>0$ such that 
		\begin{align*}
			\PPPPPP_{\uu}\{\uu_{T+1}\in X^T_R\}\ge 1-p.
		\end{align*}
		Combining this with the strong Markov property, we derive 
		\begin{align}\label{HER6-1}
			\PPPPPP_{\uu}\{\tilde n>k\}&\le \E_{\uu}\left(\prod_{j=0}^k\I_{\left\{\uu_{\tilde\tau_j(r)}\notin X^{T}_{R}\right\}}\right)\notag\\&\le \E_{\uu}\left(\prod_{j=0}^{k-1}\I_{\left\{\uu_{\tilde\tau_j(r)}\notin X^{T}_{R}\right\}}\PPPPPP_{u_{\tau_k(r)}}\{\uu_{T+1}\notin X^T_R\}\right)\notag\\&\le p\E_{\uu}\left(\prod_{j=0}^{k-1}\I_{\left\{\uu_{\tilde\tau_j(r)}\notin X^{T}_{R}\right\}}\right) \le \ldots\le p^{k}
		\end{align}
		for any $\uu\in X^T$ and integer $k\ge 1$. On the other hand, by using \eqref{HER2} with $3\kappa$, the estimate \eqref{HER4}, and the Markov property, for 
		\[\tau_0':=\{t\ge T+1| \uu_t\in \overline{B_{X^T}(r)}\},\]
		we have 
		\begin{align*}
			\E_{\uu}(\exp(3\kappa \tau_0'))&\le e^{3\kappa(T+1)}\E_{\uu}\left(\mathbb{E}_{\uu_{T+1}}(\exp(3\kappa\tau_0))\right)\\&\le Ce^{3\kappa(T+1)}\E_{\uu}(1+\|\uu_{T+1}\|^{2m}_{X^T})\lesssim_{\kappa,T}\left(1+\|\uu\|^{2m}_{X^T}\vee r\right),
		\end{align*}
		which, together with the strong Markov property and the definition of $\tau_{k-1}$, yields
		\begin{align*}
			\E_{\uu}\exp(3\kappa \tau_k)&=\E_{\uu}\left(\mathbb{E}_{\uu}(\exp(3\kappa\tau_k)|\mathscr{F}_{\tau_{k-1}})\right)\\&=\E_{\uu}\left(\exp(3\kappa \tau_{k-1})\mathbb{E}_{\uu_{\tau_{k-1}}}(\exp(3\kappa\tau_0')\right)\\&\lesssim_{\kappa,T}\E_{\uu}\left(\exp(3\kappa \tau_{k-1})\right).
		\end{align*}
		Utilizing this and \eqref{HER2}, we obtain
		\begin{align}
			\label{HER7}
			\E_{\uu}\exp(3\kappa\tilde\tau_k)\le C_{\kappa,T}^{k}\left(\|\uu\|^{2m}_{X^T}+1\right),\qquad\forall \uu\in X^T.
		\end{align}
        
        \noindent\textit{Step 3:
  			Hyper-exponential recurrence in $X^{T}_{\infty}$.} Applying \eqref{HER6-1} and \eqref{HER7}, we see that
		\begin{align*}
			\PPPPPP_{\uu}\{\tau(R)\ge l(T+1)\}&\le \PPPPPP_{\uu}\{\tau(R)>\tilde{\tau_k}\}+\PPPPPP_{\uu}\{\tilde\tau_k\ge l(T+1)\}\\&\le \PPPPPP_{\uu}\{\tilde n>k\}+\PPPPPP_{\uu}\{\tilde\tau_k\ge l(T+1)\}\\&\le p^k+e^{-3\kappa l(T+1)}C_{\kappa,T}^k\left(\|\uu\|^{2m}_{X^T}+1\right).
		\end{align*}
		By choosing $k=cl$, where $c$ is such that 
		\[c\log(C_{\kappa,T})\le \kappa(T+1),\]
		and take $p$ sufficiently small so that 
		\[c\log p\le -2\kappa(T+1),\]
		we obtain
		\[\PPPPPP_{\uu}\{\tau(R)\ge l(T+1)\}\le 2e^{-2\kappa l(T+1)}\left(\|u\|_{X^T}^{2m}+1\right).\]
		As shown previously, this implies \eqref{HER1} and thus completes the proof.
	\end{proof}

	\subsection{Concentration property and growth condition}
	Let us define the weight function \begin{align}\label{GC1}
		\wwww_m(\uu):=\|\uu\|_{X^T}^{2m}+1
	\end{align} 
	and denote by $C_{\wwww,m}(X^T)$ (respectively, $L^{\infty}_{\wwww,m}(X^T)$) the space of all continuous (measureable) functions $\fff: X^T\to \mathbb{R}$ such that   
	\[|\fff(\uu)|\le C\wwww_m(\uu),\qquad\forall \uu\in X^T.\]
	Let us endow the spaces $C_{\wwww,m}(X^T)$ and $L^\infty_{\wwww,m}(X^T)$ with the norm
	\begin{align}\label{GC1-1}\|\fff\|_{L^\infty_{\wwww,m}}:=\sup_{\uu\in X^T}\left|\frac{\fff(\uu)}{\wwww_m(\uu)}\right|.\end{align}
	\begin{proposition}\label{propositionGC}
		For any $\VVV\in C_b(X^T)$, $\uu \in X^T$, and $t>T$, the measure 
		\[P^{T,\VVV}_{t}(\uu,\cdot):= \PPPP_{t}^{T,\VVV*}\delta_{\uu}\]
		is concentrated on $X^T_{\infty}$. Moreover, for any $\VVV\in C_b(X^T)$, there are integers $m,R_0\ge 1$ such that\footnote{The expression $\PPPP_t^{T,\VVV}\wwww_m$ is understood in the sense of \eqref{FK2}, though $\wwww_m$ is not bounded.}  
		\begin{align}
			&\sup_{t\ge0}
			\frac{\|\PPPP^{T,\VVV}_t\wwww_l\|_{L_{\wwww,l}^\infty}}{\|\PPPP^{T,\VVV}_t{\mathbf1}\|_{R_0}}<\infty,\label{GC2}\\
			&\sup_{t\in [0,1]} \|\PPPP^{T,\VVV}_t{\mathbf1}\|_{\infty}<\ty\label{GC3}
		\end{align}
        for any $l\ge m$, where $\mathbf{1}$ denotes the function on $X^T$ identically equal to 1, and $\|\cdot\|_{\infty}$, $\|\cdot\|_R$ are the $L^{\infty}$ norms on $X^T$ and $X^T_R$, respectively.
	\end{proposition}
	\begin{proof} Combining the estimate \eqref{A5} and Corollary~\ref{corollaryA1}, we see that the measure $P^{T,\VVV}_{t}(\uu,\cdot)$ is concentrated on $X^{T}_{\infty}$ for any $\VVV\in C_b(X^T)$, $\uu \in X^T$, and $t>T$. The estimate \eqref{GC3} follows from the definition \eqref{FK2}. Thus, it remains to verify \eqref{GC2}, whose proof is divided into the following four steps. 
		
		\noindent{\it Step 1.} We claim that for any $\VVV\in C_b(X^T)$, $t\ge0$, and integer $m\ge 1$, $\PPPP_t^{T,\VVV}$ is a bounded linear operator in $L^{\infty}_{\wwww,m}$. Indeed, from \eqref{A4-2}, we have 
		\begin{align}\label{GC11}
			\E_{\uu}\wwww_m(\uu_t)\lesssim_{m,\BB_0,h,T}\wwww_m(\uu)
		\end{align}
		for any $t\ge 0$ and $m\ge 1$. Then, for any $\fff\in L^{\infty}_{\wwww,m}$,
		\begin{align*}
			\PPPP_t^{T,\VVV}\fff(u)&=\E_{\uu}\left(\fff(\uu_t)\exp\left(\int_0^t\VVV(\uu_s)\dd s\right)\right)\\&\le \|\fff\|_{L^{\infty}_{\wwww,m}}e^{t\|\VVV\|_{\infty}} \E_{\uu}\wwww_m(\uu_t)\le C_{m,\BB_0,h,T}e^{t\|\VVV\|_{\infty}}\|\fff\|_{L^{\infty}_{\wwww,m}}\wwww_m(\uu),
		\end{align*}
		which implies 
		\begin{align}\label{GC4}\|\PPPP_t^{T,\VVV}\|_{\LL(L^{\infty}_{\wwww,m})}\le C_{m,\BB_0,h,T}e^{t\|\VVV\|_{\infty}}\end{align}
		as claimed.

		\noindent\textit{Step 2.} Notice that by replacing $\VVV$ by $\VVV-\inf_{X^T}\VVV$, we can always assume that $\VVV$ is non-negative. To prove \eqref{GC2}, it suffices to establish 
		\begin{align}
			\label{GC9}
			\sup_{k\in\N}\frac{\|\PPPP_{k(T+1)}^{T,\VVV}\wwww_m\|_{L^{\infty}_{\wwww,m}}}{\|\PPPP_{k(T+1)}^{T,\VVV}\mathbf{1}\|_{R_0}}<\infty.
		\end{align}
		To see this, for any $t\ge0$, let us write 
		\[\alpha(t):=\left[\frac{t}{(T+1)}\right](T+1).\]
		Then, from \eqref{GC4} and the non-negativity of $\VVV$, we have 
		\begin{align*}
			\|\PPPP^{T,\VVV}_t\wwww_m\|_{L^{\infty}_{\wwww,m}}&\le \|\PPPP_{t-\alpha(t)}^{T,\VVV}\|_{\LL(L^{\infty}_{\wwww,m})}\|\PPPP_{\alpha(t)}^{T,\VVV}\wwww_m\|_{L^{\infty}_{\wwww,m}}\\&\lesssim_{\VVV,m,\BB_0,h,T}\|\PPPP_{\alpha(t)}^{T,\VVV}\wwww_m\|_{L^{\infty}_{\wwww,m}},
		\end{align*}
		and 
		\begin{align*}
			\|\PPPP^{T,\VVV}_{t}\mathbf{1}\|_{R_0}\ge \|\PPPP^{T,\VVV}_{\alpha(t)}\mathbf{1}\|_{R_0}.
		\end{align*}
		This implies 
		\[\sup_{t\ge0}
		\frac{\|\PPPP^{T,\VVV}_t\wwww_m\|_{L_{\wwww,m}^\infty}}{\|\PPPP^{T,\VVV}_t{\mathbf1}\|_{R_0}}\le C_{\VVV,m,\BB_0,h,T}\sup_{k\in\N}\frac{\|\PPPP_{k(T+1)}^{T,\VVV}\wwww_m\|_{L^{\infty}_{\wwww,m}}}{\|\PPPP_{k(T+1)}^{T,\VVV}\mathbf{1}\|_{R_0}}.\]
		\noindent\textit{Step 3.} Let us show that there are integers $m_0,R_0\ge 1$ such that 
		\begin{align}
			\label{GC5}
			\sup_{t\ge0}
			\frac{\|\PPPP^{T,\VVV}_t\mathbf{1}\|_{L_{\wwww,m}^\infty}}{\|\PPPP^{T,\VVV}_t{\mathbf1}\|_{R_0}}<\infty
		\end{align}
        for any $m\ge m_0$. To see this, let $\tau(R)$ be defined in \eqref{HER0}, and let the integers $m_0,R_0\ge1$ and $C_0>0$ be such that \eqref{HER1} holds with $\kappa:=\|\VVV\|_{\infty}$. Then, 
		\begin{align}
			\label{GC6}
			\PPPP_t^{T,\VVV}\mathbf{1}(\uu)&=\E_{\uu}\left(\I_{A_t}\exp\left(\int_0^t\VVV(\uu_s)\dd s\right)\right)+\E_{\uu}\left(\I_{A^c_t}\exp\left(\int_0^t\VVV(\uu_s)\dd s\right)\right)\notag\\&=: I_1+I_2,
		\end{align}
		where $A_t:=\{\tau(R_0)>t\}.$ By Proposition~\ref{propositionHER}, we have
		\begin{align}
			\label{GC7}
			I_1\le \E_{\uu}\exp({\kappa \tau(R_0)})\le C_0\wwww_{m_0}(\uu).
		\end{align}
		As for $I_2$, by using the strong Markov property, Proposition~\ref{propositionHER}, and the non-negativity of $\VVV$, we see that
		\begin{align}
			\label{GC8}
			I_2&\le \E_{\uu}\left(\I_{A_t^c}\exp\left(\int_0^{\tau(R_0)}\VVV(\uu_s)\dd s\right)\E_{\uu_{\tau(R_0)}}\left(\int_0^{t}\VVV(\uu_s)\dd s\right)\right)\notag\\&\le  C_0\|\PPPP_t^{T,\VVV}\mathbf{1}\|_{R_0}\wwww_{m_0}(\uu).
		\end{align}
		Notice that $\PPPP_t^{T,\VVV}\mathbf{1}(\uu)\ge 1$ for any $\uu\in X^T$. Combining this, the estimates \eqref{GC6}--\eqref{GC8}, and the fact
        \[\wwww_{m}(\uu)\le 2\wwww_{m+1}(\uu),\qquad \forall m\ge 1,\]
    we obtain \eqref{GC5}.
		
		\noindent\textit{Step 4.} We turn to the proof of \eqref{GC2}. Utilizing the Markov property, the estimate \eqref{HER3}, and the non-negativity of $\VVV$, we obtain
		\begin{align}
			\label{GC10}\notag&\PPPP_{k(T+1)}^{T,\VVV}\wwww_m(\uu)\\\notag&\qquad\le e^{(T+1)\|\VVV\|_{\infty}}\E_{\uu}\left(\exp\left(\int_0^{(k-1)(T+1)}\VVV(\uu_s)\dd s\right)\E_{\uu_{(k-1)(T+1)}}\wwww_m(\uu_{T+1})\right)\\&\qquad\le 4e^{(T+1)\|\VVV\|_{\infty}-m\alpha_1}\PPPP^{T,\VVV}_{(k-1)(T+1)}\wwww_m(u)+C_{\VVV,m,\BB_0,h,T}\PPPP_{k(T+1)}^{T,\VVV}\mathbf{1}(\uu).
		\end{align}
		Therefore, by choosing $m$ sufficiently large so that 
		\[q:=4e^{(T+1)\|\VVV\|_{\infty}-m\alpha_1}<1\]
		and iterating the inequality \eqref{GC10}, we obtain 
		\begin{align*}
			\PPPP^{T,\VVV}_{k(T+1)}\wwww_m(\uu)\le q^k\wwww_m(\uu)+\frac{1}{1-q}C_{\VVV,m,\BB_0,h,T}\PPPP_{k(T+1)}^{T,\VVV}\mathbf{1}(\uu).
		\end{align*}
		This, together with \eqref{GC5}, implies \eqref{GC9}.
	\end{proof}

	\subsection{Time-continuity and dissipativity condition}
	The following result verifies the time-continuity condition in Proposition~\ref{propositionE2}.
	\begin{proposition}\label{propositionTC}
		For any integer $m\ge 1$, $\fff\in C_{\wwww,m}(X^T)$, $\VVV\in C_b(X^T)$, and $\uu\in X^T$, the map $t\mapsto \PPPP^{T,\VVV}_t \fff(\uu)$ is continuous from $[0,\infty)$ to $\R$.
	\end{proposition}
	\begin{proof}
		By the continuity of the sample paths, we have 
		\[\lim_{t\to t_0}\fff(\uu_t)\exp\left(\int_0^t \VVV(\uu_s)\dd s \right)=\fff(\uu_{t_0})\exp\left(\int_0^{t_0} \VVV(\uu_s)\dd s \right)\]
		almost surely. Moreover, we claim that the family
		\[\A:=\left\{\fff(\uu_t)\exp\left(\int_0^t \VVV(\uu_s)\dd s \right)\Big| 0\le t\le t_0+1\right\}\]
		is uniformly integrable. Indeed, by using \eqref{GC11}, we have
		\begin{align*}
			\E_\uu \left(\fff(\uu_t)\exp\left(\int_0^t \VVV(\uu_s)\dd s \right)\right)^2&\le 2\|\fff\|_{L^{\infty}_{\wwww,m}}^2 e^{2(t_0+1)\|\VVV\|_{\infty}}\E_{\uu}\wwww_{2m}(\uu_t)\\&\lesssim_{m,\BB_0,h,T}\|\fff\|_{L^{\infty}_{\wwww,m}}^2 e^{2(t_0+1)\|\VVV\|_{\infty}}\wwww_{2m}(\uu),
		\end{align*}
		which implies the uniform integrability of $\A$ as claimed. Thus, an application of the Vitali convergence theorem gives the continuity of $t\mapsto \PPPP^{T,\VVV}_t \fff(\uu)$.
	\end{proof}
	Let us turn to the dissipativity condition in Proposition~\ref{propositionE2}.
	\begin{proposition}\label{propositionDC}
		For any $\VVV\in C_b(X^T)$, integer $R\ge 1$, and $\TT>0$, there is a positive integer $\RR$ depending only on $R$ such that 
		\begin{align}\label{DC1}
			\inf_{t\in[0,\TT]}\inf_{\uu\in \overline{B_{X^T}(R)}}P^{T,\VVV}_t(\uu,\overline{B_{X^T}(\RR)})>0.
		\end{align}
	\end{proposition}
	\begin{proof}
		\textit{Step 1.} Let us establish \eqref{DC1} for $P_t^T(\uu,\cdot)$. Applying \eqref{GC11}, we have 
		\[\E_{\uu}\|\uu_t\|^2_{X^T}\lesssim_{\BB_0,h,T} \|\uu\|^2_{X^T}+1,\qquad\forall t\ge 0.\]
        In particular, for $\uu\in \overline{B_{X^T}(R)}$, 
		\[\E_{\uu}\|\uu_t\|^2_{X^T}\le C_{\BB_0,h,T}(R^2+1),\]
		which, by the Chebyshev inequality, implies
		\begin{align*}
			P^{T}_t(\uu,\overline{B_{X^T}(\RR)})\ge1-\RR^{-2}\E_{\uu}\|\uu_t\|^2_{X^T}\ge 1-\frac{C_{\BB_0,h,T}(R^2+1)}{\RR^2}.
		\end{align*}
		Therefore, for any integer $R\ge 1$, by choosing $\RR\ge 1$ sufficiently large and taking infimum on both sides of the above estimate, we have
		\begin{align}
			\label{DC2}
			\inf_{t\ge 0}\inf_{\uu\in \overline{B_{X^T}(R)}}P^{T}_t(\uu,\overline{B_{X^T}(\RR)})\ge \frac{1}{2}.
		\end{align}
		\textit{Step 2.} To conclude the proof, we notice that 
		\begin{align*}
			P^{T,\VVV}_t(\uu,\overline{B_{X^T}(\RR)})\ge e^{-t\|\VVV\|_{\infty}}P^{T}_t(\uu,\overline{B_{X^T}(\RR)}).
		\end{align*}
		Thus, by choosing $\RR\ge 1$ sufficiently large so that \eqref{DC2} holds, we obtain 
		\begin{align*}
			\inf_{t\in[0,\TT]}\inf_{\uu\in \overline{B_{X^T}(R)}}P^{T,\VVV}_t(\uu,\overline{B_{X^T}(\RR)})\ge \frac{1}{2}e^{-\TT\|V\|_{\infty}}>0.
		\end{align*}
		This completes the proof.
	\end{proof}
	\subsection{Uniform irreducibility}
	In this subsection, our goal is to establish the uniform irreducibility property for the Feynman--Kac semigroup. Motivated by the relation
	\begin{align}
		\label{UI0}P^{T,\VVV}_t(\uu,\dd \vv )\ge e^{-t\|\VVV\|_{\infty}}P^{T}_t(\uu,\dd \vv),
	\end{align}
	we start with the transition function $P^{T}_t(\uu,\cdot)$.
	\begin{lemma} \label{UILEMMA0}For any integers $\rho,R\ge 1$ and $r>0$, there are numbers $l=l(R)>0$ and $p=p(\rho,r)>0$ such that 
		\begin{align}\label{UI1}
			P^T_l\left(\uu,X^T_{\rho}\mcap B_{X^T}(\hat \uu,r)\right)\ge p, 
		\end{align}
		where $\uu\in \overline{B_{X^T}(R)}$ and $\hat{\uu}\in X^T_{\rho}$.
	\end{lemma}
	\begin{proof}\textit{Step 1.} Let us first study the irreducibility property of the transition function $P_{T+1}^T(\uu,\cdot)$. Notice that from the definition \eqref{FK1-1},
		\begin{align}
			\label{UI2}
			P_{T+1}^T(\uu,\Gamma)=\PPPPPP_{\uu(T)}\{u_{[1,T+1]}\in\Gamma\},\qquad \forall \Gamma\in \mathscr{B}(X^T).
		\end{align}
		Therefore, it suffices to consider
		\[\tilde P_{T+1}^T(u,\Gamma):=\PPPPPP_{u}\{u_{[1,T+1]}\in\Gamma\},\qquad u\in H,\]
		as we have $P_{T+1}^T(\uu,\cdot)=\tilde P_{T+1}^T(u,\cdot)|_{u=\uu(T)}$. We claim that for any integer $\rho\ge 1$ and $r,d>0$,
		\begin{align}\label{UI3}
			p_1(\rho,r,d):=\inf_{(u,\hat\uu)\in \overline{B_{U}(d)}\times X^T_{\rho} }\tilde P_{T+1}^T\left(u,X^T_{\rho}\mcap B_{X^T}(\hat \uu,r)\right)>0.
		\end{align}
		Before proceeding, let us reduce the verification of the claim to some simplified situations. First, notice that 
		\begin{align}\label{UI4}
			\tilde P_{T+1}^T\left(u,X^T_{\rho}\mcap B_{X^T}(\hat \uu,r)\right)\ge \tilde P_{T+1}^T\left(u,B_{X^T_{\infty}}(\rho)\mcap B_{X^T}(\hat \uuu,r)\right),
		\end{align}
		where the right hand-side, considered as a function of $(u,\hat{\uu})\in H\times X^T$, is lower semi-continuous, due to Proposition~\ref{propositionlip}. Combining this with the compactness of $\overline{B_{U}(d)}\times X^T_{\rho}$, we see that \eqref{UI4} holds, provided that 
		\[\tilde P_{T+1}^T\left(u,B_{X^T_{\infty}}(\rho)\mcap B_{X^T}(\hat \uu,r)\right)>0\]
		for each fixed $(u,\hat\uu)\in \overline{B_{U}(d)}\times X^T_{\rho}$. Moreover, for any $\hat\uu\in X_\rho^T$, there is $\tilde\uu\in B_{X^T_{\infty}}(\rho)$ and $\tilde{r}>0$ such that 
		\[B_{X^T_\infty}(\tilde\uu, \tilde r)\subset B_{X^T_{\infty}}(\rho)\mcap B_{X^T}(\hat \uu,r).\]
		Hence, the claimed result \eqref{UI3} holds, if we can prove 
		\begin{align}
			\label{UI5}
			\tilde P_{T+1}^T\left(u,B_{X^T_\infty}(\hat\uu,  r)\right)>0
		\end{align}
		for each fixed $(u,\hat\uu)\in \overline{B_{U}(d)}\times X^T_{\infty}$ and $d,r>0$.
		
		\noindent\textit{Step 2.} Let us prove \eqref{UI5} by using a control-type argument, cf. Section 3.6.1 in \cite{KS12}. We need the following auxiliary lemmas whose proofs are postponed until the end of this subsection.
		\begin{lemma}\label{UILEMMA1}
			For any $\hat\uu \in C^1([0,T];U^2)$, there is a local Lipschitz nonlinear operator
			\begin{align*}
				Z_{\hat\uu}:\ &U\to L^2(0,T+1;U^2)\mcap H^1(0,T+1;H)\mcap \XX^1(0,T+1),\notag\\
				&u\mapsto Z_{\hat\uu}(u;t)
			\end{align*}
			such that the solution of \eqref{I1}, with the initial data $u$ and the profile of the noise $\eta$ given by $\partial_t Z_{\hat\uu}+LZ_{\hat\uu}$, is equal to $\hat\uu$ when $t\in[1,T+1]$, that is,  
			\begin{align}
				\label{UI6}\hat\uu=Z_{\hat\uu}(u;t)|_{[1,T+1]}+\RRRR_1^{T+1}(u,Z_{\hat\uu}(u)),
			\end{align}
			where the space $\XX^1(0,T+1)$ and the operator $\RRRR_1^{1+T}$ are defined in \eqref{A13-1} and \eqref{A13}, respectively.
		\end{lemma}
		\begin{lemma}\label{UILEMMA2}
			For any $r>0$ and 
			\[\hat z\in L^2(0,T+1;U^2)\mcap H^1(0,T+1;H)\mcap 
			\XX^1(0,T+1),\]
			we have
			\begin{align}\label{UI7}
				\PPPPPP\left\{\|z-\hat z\|_{\XX^1(0,T+1)}+\|z-\hat z\|_{C^{\frac{1}{4}}([0,T+1];U^*)}<r \right\}>0,
			\end{align}
			where $z$ is the solution of the stochastic Stokes system \eqref{A8}.
		\end{lemma}
		Without loss of generality, we may assume that $\hat \uu\in C^1([0,T];U^2)$, as the latter is dense in $X^T_\infty$. Let $Z_{\hat\uu}$ be the nonlinear operator given by Lemma~\ref{UILEMMA1} and choose $C_d>0$ such that the image of $\overline{B_{U}(d)}$ under $Z_{\hat \uu}$ is contained in $B_{\XX^1(0,T+1)}(C_d)$. Let $M_d$ denote the Lipschitz constant of $\RRRR_1^{T+1}$ on the bounded set $\overline{B_{U}(d)}\times B_{\XX^1(0,T+1)}(C_d)$. Then, by using Proposition~\ref{propositionlip}, we have
		\[\|\RRRR_1^{T+1}(u,Z_{\hat \uu}(u))-\RRRR_1^{T+1}(u,z)\|_{X^T_\infty}< \frac{r}{2},\]
		provided that $\|Z_{\hat \uu}(u)-z\|_{\XX^1(0,T+1)}<\frac{r}{2M_d}$. Let $u(t)$ denote the solution of \eqref{I1} issued from $u$. Then, by the definition of $\tilde P_{T+1}^T$ and the decomposition \eqref{A6}, we have 
		\[\tilde P_{T+1}^T\left(u,B_{X^T_\infty}(\hat\uu,  r)\right)=\PPPPPP_u\{u_{[1,T+1]}\in B_{X^T_\infty}(\hat\uu,  r)\}\]
		and 
		\begin{align*}
			&\left\{\|z-Z_{\hat \uu}(u)\|_{\XX^1(0,T+1)}+\|z-Z_{\hat \uu}(u)\|_{C^{\frac{1}{4}}([0,T+1];U^*)}<\frac{r}{2M_d} \wedge\frac{r}{2}\right\}\\&\qquad\qquad\qquad\qquad\qquad\qquad\qquad\qquad\qquad\qquad\qquad\subset \{u_{[1,T+1]}\in B_{X^T_\infty}(\hat\uu,  r)\}.
		\end{align*}
		Combining this with Lemma~\ref{UILEMMA2}, we obtain \eqref{UI5} and thus \eqref{UI3}.
		
		\noindent\textit{Step 3.} Let us prove \eqref{UI1}. To this end, we introduce the set 
		\[\textbf{A}^T:=\{\uu\in X^T| \uu(T)\in \overline{B_{U}(d)}\}\]
	   and claim that there is a universal constant $d>0$ and $t=t(R)>0$ such that 
		\begin{align}\label{UI11}
			P^T_{t}(\uu, \textbf{A}^T)\ge \frac{1}{2},\qquad\forall\uu\in \overline{B_{X^T}(R)}.
		\end{align}
		Indeed, by the definition \eqref{FK1-1}, we have
		\[P^T_{t}(\uu, \textbf{A}^T)=\PPPPPP_{\uu(T)}\{u_{t}\in \overline{B_{U}(d)}\}.\]
        Combining this with the estimates \eqref{A4} and \eqref{A5}, for $t>1$, we have
		\[\E_{\uu(T)}\log(1+\log(1+\|u_t\|_1^2))\lesssim_{\BB_1,h}e^{-\alpha_1 t}\|\uu(T)\|^2+1\lesssim_{\BB_1,h}e^{-\alpha_1 t}R^2+1.\]
        This, together with the Chebyshev inequality, implies  
		\begin{align*}
			P^T_{t}(\uu, \textbf{A}^T)=1-\PPPPPP_{\uu(T)}\{u_{t}\notin \overline{B_{U}(d)}\}\ge 1-\frac{C_{\BB_1,h}(e^{-\alpha_1 t}R^2+1)}{\log(1+\log(1+d^2))}.
		\end{align*}
		Choosing $d>0$ sufficiently large so that 
		\[\log(1+\log(1+d^2))=4C_{\BB_1,h}\]
		and 
		\begin{align*}
			t=t(R):=\frac{2\log R}{\alpha_1}\vee 1,
		\end{align*}
		we obtain \eqref{UI11} as claimed. Finally, an application of the Kolmogorov--Chapman relation combined with \eqref{UI3} and \eqref{UI11} yields \eqref{UI1} with 
		\[p=p(\rho,r):=\frac{1}{2}p_1(\rho,r,d)>0\]
		and 
		\begin{align}\label{UI12}
			l=l(R):=\frac{2\log R}{\alpha_1}\vee 1+T+1.
		\end{align}
		This completes the proof.
	\end{proof}
	Combining the relation \eqref{UI0}, Lemma~\ref{UILEMMA0}, and \eqref{UI12}, we obtain the uniform irreducibility property for the Feynman--Kac semigroup.
	\begin{proposition}\label{propositionUI}
		For any integers $\rho,R\ge 1$ and $r>0$, there are numbers $l=l(R)>0$ and $p=p(\rho,r)>0$ such that
		\begin{align}\label{UI13}
			P^{T,\VVV}_l\left(\uu,X^T_{\rho}\mcap B_{X^T}(\hat \uu,r)\right)\ge C_{\VVV}R^{-C_{\VVV}}p, 
		\end{align}
		where $\uu\in \overline{B_{X^T}(R)}$, $\hat{\uu}\in X^T_{\rho}$, and $C_{\VVV}>0$ is a constant depending only on $\VVV$.
	\end{proposition} 
	Let us present the proofs of Lemmas~\ref{UILEMMA1} and \ref{UILEMMA2}.
	\begin{proof}[Proof of Lemma~\ref{UILEMMA1}] Let $\chi:\R\to[0,1]$ be a smooth cut-off function such that $\chi(t)\equiv 1$ for $t\le  \frac{1}{2}$ and $\chi(t)\equiv 0$ for $t\ge 1$. We define
		\[w_{\hat\uu}(u;t):=\!\begin{cases}
			\chi(t)e^{-L t}u+(\partial_t\hat\uu(0)-\hat\uu(0))t^2+(2\hat\uu(0)-\partial_t\hat\uu(0))t,\!\!\!&t\in[0,1],\\
			\hat\uu(t-1),\! \!\!&t\in[1,T+1],
		\end{cases}\]
		and
		\[\xi_{\hat\uu}(u):=\partial_t w+Lw+B(w)-h,\]
		where $w=w_{\hat\uu}(u)$. Let $Z_{\hat\uu}(u)$ be the solution of the following forced Stokes system
		\[\partial_t z+Lz=\xi_{\hat\uu}(u)\]
		with zero initial data $z(0)=0$. Applying the standard regularity theory for parabolic equations, we see that 
		\[u\mapsto Z_{\hat{\uu}}(u)\]
		is a local Lipschitz map. Notice that $w_{\hat \uu}(u)$ is the solution of \eqref{I1} which satisfies all aforementioned properties. This completes the proof.
	\end{proof}
	\begin{proof}[Proof of Lemma~\ref{UILEMMA2}] Let us define
		\[\tilde \XX^1(0,T+1):=L^2(0,T+1;U^2)\mcap H^1(0,T+1;H)\mcap 
		\XX^1(0,T+1)\]
		equipped with the equivalent norm
		\[\|z\|_{\tilde\XX^1(0,T+1)}:=\|z\|_{L^2(0,T+1;U^2)}+\|\partial_tz\|_{L^2(0,T+1;H)}.\]
		Using the Sobolev embedding 
		\[H^1(0,T+1;H)\subset C^{\frac{1}{4}}([0,T+1];H),\]
		we have 
		\begin{align}\label{UI8}
			\|\qqQ_N \hat z\|_{\XX^1(0,T+1)}+\|\qqQ_N \hat z\|_{C^{\frac{1}{4}}([0,T+1];U^*
				)}\lesssim \|\qqQ_N\hat z\|_{\tilde\XX^1(0,T+1)}\to 0,
		\end{align}
		as $N\to+\infty$. Notice that 
		\[z(t)=\sum_{j\ge 1}b_jz_j(t)e_j,\]
		where 
		\[z_j(t):=\int_0^te^{-\alpha_j(t-s)}\dd \beta_j(s)=\beta_j(t)-\alpha_j\int_0^te^{-\alpha_j(t-s)}\beta_j(s)\dd s,\]
		from which, it follows that
		\begin{align*}
			\|z_j\|_{C^{\frac{1}{4}}([0,T+1])}\lesssim\alpha_j \|\beta_j\|_{C^{\frac{1}{4}}([0,T+1])},\qquad \forall j\ge 1.
		\end{align*}
		For any $s,t\in[0,T+1]$ satisfying $s\neq t$, this implies
		\begin{align*}
			\frac{\|\qqQ_N z(t)-\qqQ_Nz(s)\|^2_{U^*}}{\sqrt{|t-s|}}&=\sum_{j\ge N}b_j^2\alpha_j^{-1}\frac{|z_j(t)-z_j(s)|^2}{\sqrt{|t-s|}}\\&\le \sum_{j\ge N}b_j^2\alpha_j^{-1}\|z_j\|^2_{C^{\frac{1}{4}}([0,T+1])}\le \sum_{j\ge N}b_j^2\alpha_j\|\beta_j\|_{C^{\frac{1}{4}}([0,T+1])}^2,
		\end{align*}  
            which leads to
		\begin{align}
			\label{UI9}\E \|\qqQ_N z\|_{C^{\frac{1}{4}}([0,T+1];U^*)}^2\lesssim \sum_{j\ge N}b_j^2\alpha_j\to0,\qquad \mbox{as $N\to+\infty$},
		\end{align}
        Similarly, 
		\begin{align}\label{UI10}
			\E\|\qqQ_N z\|_{\XX^1(0,T+1)}^2\lesssim \sum_{j\ge N}b_j^2\alpha_j\to 0,\qquad \mbox{as $N\to+\infty$}.
		\end{align}

		Let $r>0$ be fixed. By applying \eqref{UI8}--\eqref{UI10}, there is an integer $N=N(r)\ge 1$ such that 
		\[\|\qqQ_N \hat z\|_{\XX^1(0,T+1)}+\|\qqQ_N \hat z\|_{C^{\frac{1}{4}}([0,T+1];U^*
			)}<\frac{r}{3}\]
		and 
		\[\PPPPPP\left\{\|\qqQ_N  z\|_{\XX^1(0,T+1)}+\|\qqQ_N  z\|_{C^{\frac{1}{4}}([0,T+1];U^*
			)}<\frac{r}{3}\right\}>0.\]
		Notice that
		\begin{align*}
			\|\ppP_Nz-\ppP_N\hat z\|_{\XX^1(0,T+1)}&+\|\ppP_Nz-\ppP_N\hat z\|_{C^{\frac{1}{4}}([0,T+1];U^*)}\\&\qquad\qquad\le C(r)\sup_{j<N} \|b_jz_j-\hat z_j\|_{C^{\frac{1}{4}}([0,T+1])},
		\end{align*}
		where $\hat z_j(t):=\langle \hat z(t),e_j\rangle$. For $j<N$, let us define 
		\[A_j:=\left\{\|b_jz_j-\hat z_j\|_{C^{\frac{1}{4}}([0,T+1])}<\frac{r}{3C(r)}\right\}\]
		and 
		\[A_{\ge N}:=\left\{\|\qqQ_N  z\|_{\XX^1(0,T+1)}+\|\qqQ_N  z\|_{C^{\frac{1}{4}}([0,T+1];U^*
			)}<\frac{r}{3}\right\}.\]
		As $z_j$ has a Gaussian distribution and $b_j>0$, we have $\PPPPPP (A_j)>0$ for all $j<N$. Hence, using the independence of the processes $\{z_j\}_{j\ge 1}$, we have 
		\[\PPPPPP\left(\mcap_{j<N} A_j\mcap A_{\ge N}\right)>0.\]
		As 
		\[\mcap_{j<N} A_j\mcap A_{\ge N}\subset \left\{\|z-\hat z\|_{\XX^1(0,T+1)}+\|z-\hat z\|_{C^{\frac{1}{4}}([0,T+1];U^*)}<r \right\},\]
		the proof is complete.
	\end{proof}
	\subsection{Uniform Feller property}
	Let $\VV^T$ be the space of functions $\VVV: X^T\to \R$ for which there is an integer $N\ge 1$ and a function $\F\in L_b(X^T)$ such that
	\begin{align}
		\label{UF0}\VVV(\uu)=\F(\ppP_N\uu),\qquad \uu\in X^T.
	\end{align}
	Here, $\ppP_N$ is the orthogonal projection in $H$ understood as the natural extension:
	\[\ppP_N \uu(t):=\sum_{j\le N}\langle \uu(t),e_j\rangle e_j,\qquad t\in[0,T].\]
	The main result is stated as follows.
	\begin{proposition}\label{propositionUF}
		For any $\VVV\in\VV^T$, there is an integer $R_0\ge 1$ such that the family $\{\|\PPPP_t^{T,\VVV}\mathbf{1}\|^{-1}_R\PPPP^{T,\VVV}_t \fff\}_{t\ge0}$ is uniformly equicontinuous on $X^T_{R}$ for any $\fff\in \VV^T$ and $R\ge R_0$. 
	\end{proposition}
	\begin{proof}\textit{Step 1: Construction of coupling processes.} For any $u,u'\in H$, let $\{u_t\}_{t\ge 0}$ and $\{u'_t\}_{t\ge0}$ be the solutions of \eqref{I1} issued from $u$ and $u'$, respectively, and let $\{v_t\}_{t\ge0}$ be the solution of the following auxiliary problem:
		\begin{align*}
			\begin{cases}
				\partial_t v+B(v)+Lv+\ppP_N[a(v-u)+B(u)-B(v)]=h+\eta,\\
				v|_{t=0}=u'.
			\end{cases}
		\end{align*}
		Let us denote by $\lambda(u,u')$ and $\lambda'(u,u')$ the distributions of processes $\{v_t\}_{0\le t\le 1}$ and $\{u'_t\}_{0\le t\le 1}$, respectively. Using Theorem 1.2.28 in \cite{KS12}, we can construct a maximal coupling $(\nu(u,u'),\nu'(u,u'))$ defined on some probability space $(\tilde{\Omega},\tilde{\mathscr{F}}, \tilde{\mathbb{P}})$ for the pair $(\lambda(u,u'),\lambda'(u,u'))$. Let $\{\tilde{v}_t\}_{0\le t\le 1}$ and $\{\tilde{u}'_t\}_{0\le t\le 1}$ be the flows of this maximal coupling. Then, $\{\tilde{v}_t\}_{0\le t\le 1}$ solves 
		\begin{align*}
			\begin{cases}\partial_t \tilde{v}+B(\tilde{v})+L\tilde{v}+\ppP_N[a \tilde{v}-B(\tilde{v})]=h+\psi,\\
				\tilde{v}|_{t=0}=u',
			\end{cases}
		\end{align*}
		where the process $\left\{\int_0^{t}\psi(s)\mathrm{d} s\right\}_{0\le t\le 1}$ has the same distribution as
		\[
		\left\{\int_0^{t}\left(\eta(s)+\ppP_N (a u-B(u))\right)\mathrm{d} s\right\}_{0\le t\le 1}.\]
		Let $\{\tilde{u}_t\}_{0\le t\le 1}$ be the solution of 
		\begin{align*}
			\begin{cases}
				\partial_t \tilde{u}+B(\tilde{u})+L\tilde{u}+\ppP_N[a \tilde{u}-B(\tilde{u})]=h+\psi,\\
				\tilde{u}|_{t=0}=u,
			\end{cases}
		\end{align*}
		then by the uniqueness-in-law for the above equation, $\{\tilde{u}_t\}_{0\le t\le 1}$ is distributed as that of $\{u_t\}_{0\le t\le 1}$. Let us define  
		\[\GGGG_t(u,u',\omega):=\tilde{u}_t,\quad \GGGG'_t(u,u',\omega):=\tilde{u}_t',\quad \hat\GGGG_t(u,u',\omega):=\tilde{v}_t\]
		for $u, u'\in H$,  $\omega\in\tilde\Omega$, and $t\in [0,1]$. Let $\{(\Omega^k,\mathscr{F}^k,\mathbb{P}^k)\}_{k\ge0}$ be independent copies of the probability space $(\tilde{\Omega}, \tilde{\mathscr{F}},\tilde{\mathbb{P}})$, and denote by $(\Omega,\mathscr{F},\mathbb{P})$ their direct product. For $\omega=(\omega^0,\omega^1,\ldots)\in\Omega$, we define $\tilde u_0=u$, $\tilde u_0'=u'$, and 
		\begin{gather}
			\label{UF1}
			\tilde u_t(\omega):=\GGGG_s(\tilde{u}_{k}(\omega),\tilde{u}'_{k}(\omega),\omega^k), \quad\tilde{u}'_t(\omega):=\GGGG_s'(\tilde{u}_{k}(\omega),\tilde{u}'_{k}(\omega),\omega^k),\\
			\tilde v_t(\omega):=\hat\GGGG_s(\tilde{u}_{k}(\omega),\tilde{u}'_{k}(\omega),\omega^k),\label{UF2}
		\end{gather}
		where $t=s+k$ with $s\in [0,1)$ and $k\ge0$. We call the pair $(\tilde u_t,\tilde u_t')$ a coupled trajectory at level $(N,a)$ issued from $(u,u')$.
		
        \noindent\textit{Step 2: Stratification.} In what follows, we shall simply denote by $(u_t,u'_t)$ and $v_t$ the coupled trajectory and the associated process constructed above. For any integers $r,\rho\ge0 $, we define the events $\hat G_r:=\mcap_{j=0}^rG_j$ with
		\[G_j:=\{v_t=u'_t,\ \forall t\in[j,j+1)\},\]
		and $F_{r,0}=\emptyset,$
		\begin{align*}
			F_{r,\rho}&:=\left\{\sup_{t\in[0,r]}\left(\int_0^t\|u_s\|_1^2+\|u'_s\|^2_1\dd s-2Kt\right)\le \|u\|^2+\|u'\|^2+\rho\right\}\\&\qquad\quad\mcap\{\|u_r\|^2+\|u'_r\|^2\le\rho\},\qquad \rho\ge 1,
		\end{align*}
		where $K$ is the constant in \eqref{A1}. Furthermore, we introduce the following disjoint partitions of $\Omega$: $A_0:=G_0^c$, $ A_\infty:=\hat G_{+\infty}$, and 
		\[A_{r,\rho}:=\left(\hat G_{r-1}\mcap G_r^c\right)\mcap\left(F_{r,\rho}\mcap F^c_{r,\rho-1}\right),\qquad r,\rho\ge 1.\]
		We need the following probabilistic estimates for the events $A_0$ and $A_{r,\rho}$ whose proof can be found in \cite{Ner19}, see Lemma 4.3 therein.
		\begin{lemma}\label{lemma-coupling}
			Let $u,u'\in \overline{B_{H}(R)}$ be such that $d:=\|u-u'\|\le1$ and $R\ge 1$. For any $\alpha>0$, there is an integer $N_0=N_0(\alpha)$ and universal constants $c_1,c_2,c_3>0$ such that 
			\begin{align}
				\label{UF3}
				&\PPPPPP(A_0)\lesssim_{R,N,a}d^{\frac{c_1}{2}},\\\label{UF4}&\PPPPPP(A_{r,\rho})\lesssim_R e^{-c_2\rho}\wedge\left(d^{c_1}e^{-c_1\alpha r}+\left[\exp\left(C_{R,N,a}d^{c_1}e^{c_3\rho-c_1\alpha r}\right)-1\right]^{\frac{1}{2}}\right),
			\end{align}
			provided that the coupling parameter $(N,a)$ satisfies $N\ge N_0(\alpha)$ and $a\ge \frac{N^2}{2}$.
		\end{lemma}
		
		\noindent\textit{Step3: Proof of the uniform equicontinuity.} Let $m,R_0$ be given in Proposition~\ref{propositionGC}, and let $\VVV,\fff\in\VV^T$, $R\ge R_0$. We aim to prove the uniform equicontinuity of 
		\[g_t:=\|\PPPP_t^{T,\VVV}\mathbf{1}\|^{-1}_R\PPPP^{T,\VVV}_t \fff,\qquad t\ge0\]
		on $X^T_R$. Without loss of generality, we assume that $\VVV,\fff$ are non-negative, $\fff\le 1$, and the integer $N$ in \eqref{UF0} is the same for $\VVV,\fff$. Let us choose $\uu,\uu'\in X^T_R$ such that $d:=\|\uu-\uu'\|_{X^T}\le 1$. Since $X^T_R\subset \overline{B_{X^T}(R)}$, we have 
		\[\uu(T),\uu'(T)\in \overline{B_{H}(R)},\qquad \|\uu(T)-\uu'(T)\|\le d\le 1.\]
		Let $(u_t,u'_t)$ be the coupled trajectory at level $(N,a)$ issued from $(\uu(T),\uu'(T))$ and let $v_t$ be the associated process. Let
		\[\uu_t:=u_{[t-T,t]},\qquad \uu'_t:=u'_{[t-T,t]},\qquad t\ge 0,\]
        where we set
		\[u_t:=\uu(T+t),\qquad u'_t:=\uu'(T+t),\qquad t\in[-T,0].\]
        Let 
		\[\Xi^{\VVV,\uu}_t:=\exp\left(\int_0^t \VVV(\uu_s)\dd s\right),\qquad \Xi^{\VVV,\uu'}_t:=\exp\left(\int_0^t \VVV(\uu'_s)\dd s\right).\]
        Under the above notations, we have 
		\begin{align}\label{UF8}
			\PPPP^{T,\VVV}_t\fff(\uu)-\PPPP^{T,\VVV}_t\fff(\uu')&=\E\left(\Xi^{\VVV,\uu}_t\fff(\uu_t)-\Xi_t^{\VVV,\uu'}\fff(\uu'_t)\right)\notag\\&=\E\left(\I_{A_0}\left(\Xi^{\VVV,\uu}_t\fff(\uu_t)-\Xi_t^{\VVV,\uu'}\fff(\uu'_t)\right)\right)\notag\\&\quad+\E\left(\I_{A_\infty}\left(\Xi^{\VVV,\uu}_t\fff(\uu_t)-\Xi_t^{\VVV,\uu'}\fff(\uu'_t)\right)\right),\notag\\&\quad+\sum_{r,\rho\ge 1}\E\left(\I_{A_{r,\rho}}\left(\Xi^{\VVV,\uu}_t\fff(\uu_t)-\Xi_t^{\VVV,\uu'}\fff(\uu'_t)\right)\right)\notag\\&=:I_0^t+I_\infty^t+\sum_{r,\rho\ge 1}I^t_{r,\rho}.
		\end{align}
  
		We claim that 
		\begin{align}
			\label{UF5}
			|I_0^t|&\lesssim_{R,\VVV,T}\|\PPPP^{T,\VVV}_t\mathbf{1}\|_R\PPPPPP(A_0)^\frac{1}{2},\\
			\label{UF6}
			|I_{r,\rho}^t|&\lesssim_{R,\VVV,T}e^{r\|\VVV\|_{\infty}}\|\PPPP^{T,\VVV}_t\mathbf{1}\|_R\PPPPPP(A_{r,\rho})^\frac{1}{2}
		\end{align}
		for any integers $r,\rho\ge 1$ and $R\ge R_0$. Let us prove \eqref{UF6}, as the other follows similarly. Indeed, if $t\le r+T+1$, then by the positivity and boundedness of $\fff$, we have 
		\begin{align*}
			I_{r,\rho}^t\le \E\left(\I_{A_{r,\rho}}\Xi^{\VVV,\uu}_t\fff(\uu_t)\right)\lesssim_{\VVV,T} e^{r\|\VVV\|_{\infty}}\PPPPPP(A_{r,\rho}),
		\end{align*}
		which, by symmetry, implies \eqref{UF6}. On the other hand, for $t>r+T+1$, by using the Markov property, \eqref{GC11}, and \eqref{GC5}, we obtain
		\begin{align*}
			I_{r,\rho}^t&\le \E\left(\I_{A_{r,\rho}}\Xi^{\VVV,\uu}_t\right)\\&\lesssim_{\VVV,T}e^{r\|\VVV\|_\infty}\E\left(\I_{A_{r,\rho}}\E\left(\exp\left(\int_{r+T+1}^{t}\VVV(\uu_s)\dd s\right)\Big|\mathscr{F}_{r+1}\right)\right)\\&\lesssim_{\VVV,T} e^{r\|\VVV\|_\infty}\E\left(\I_{A_{r,\rho}}\PPPP^{T,\VVV}_{t-r-1}\mathbf{1}(\uu_{r+1})\right)\\&\lesssim_{\VVV,T}e^{r\|\VVV\|_\infty}\|\PPPP^{T,\VVV}_{t}\mathbf{1}\|_{L^\infty_{\wwww,m}}\E\left(\I_{A_{r,\rho}}\wwww_{m}(\uu_{r+1})\right)\\&\lesssim_{\VVV,T}e^{r\|\VVV\|_\infty}\|\PPPP^{T,\VVV}_{t}\mathbf{1}\|_{R_0}(\E\wwww_{2m}(\uu_{r+1}))^{\frac{1}{2}}\PPPPPP(A_{r,\rho})^\frac{1}{2}\\&\lesssim_{R,\VVV,T}e^{r\|\VVV\|_\infty}\|\PPPP^{T,\VVV}_{t}\mathbf{1}\|_{R}\PPPPPP(A_{r,\rho})^\frac{1}{2}.
		\end{align*}
		This, together with symmetry again, implies \eqref{UF6}. 
		
		It remains to estimate $I_\infty^t$. Let us decompose 
		\[I_\infty^t=I_{\infty,1}^t+I_{\infty,2}^t,\]
		where 
		\begin{align*}
			I_{\infty,1}^t&:=\E\left(\I_{A_\infty}\Xi^{\VVV,\uu}_t\left(\fff(\uu_t)-\fff(\uu'_t)\right)\right),\\
			I_{\infty,2}^t&:=\E\left(\I_{A_\infty}\left(\Xi^{\VVV,\uu}_t-\Xi_t^{\VVV,\uu'}\right)\fff(\uu'_t)\right).
		\end{align*}
		Notice that on the event $A_\infty$, 
		\[v_s=u'_s,\qquad\forall s\ge0.\]
		Then, we see that $w_s:=u_s-u_s'$ satisfies the equation 
		\[\partial_sw +Lw+a\ppP_Nw+\qqQ_N(B(u)-B(u'))=0,\]
		which gives
		\begin{align*}
			\|\ppP_N(u_s-u_s')\|\le e^{-as}d,\qquad \forall s\ge0.
		\end{align*}
		This, together with the Lipschitz property of $\fff $ and $\VVV$, gives
		\begin{align*}
			|I_{\infty,1}^t|\lesssim_{\fff} \E\left(\I_{A_\infty}\Xi^{\VVV,\uu}_t\|\ppP_N(\uu_t-\uu_t')\|_{X^T}\right)\lesssim_{\fff}d\|\PPPP_t^{T,\VVV}\mathbf{1}\|_R,
		\end{align*}
		and 
		\begin{align*}
			|I_{\infty,2}^t|&\le \E\left(\I_{A_\infty}\left(\exp\left(\int_0^t |\VVV(\uu_s)-\VVV(\uu'_s)|\dd s\right)-1\right)\Xi_t^{\VVV,\uu'}\fff(\uu'_t)\right)\notag\\&\le \left(\exp\left(C_{\VVV,a}d\right)-1\right)\|\PPPP_t^{T,\VVV}\mathbf{1}\|_R\lesssim_{\VVV,a}d\|\PPPP_t^{T,\VVV}\mathbf{1}\|_R. 
		\end{align*}
		Hence, we obtain
		\begin{align}
			\label{UF7}
			|I^t_\infty|\lesssim_{\fff,\VVV,a}d\|\PPPP_t^{T,\VVV}\mathbf{1}\|_R. 
		\end{align}
		
		Let us plug the estimates \eqref{UF5}--\eqref{UF7} into \eqref{UF8} to get
		\begin{align}
			\label{UF9}
			\frac{|\PPPP^{T,\VVV}_t\fff(\uu)-\PPPP^{T,\VVV}_t\fff(\uu')|}{\|\PPPP^{T,\VVV}_t\mathbf{1}\|_R}\lesssim_{R,\fff,\VVV,T,a}d+\PPPPPP(A_0)^\frac{1}{2}+\sum_{r,\rho\ge1}e^{r\|\VVV\|_{\infty}}\PPPPPP(A_{r,\rho})^\frac{1}{2}.
		\end{align}
		Combining this with Lemma \ref{lemma-coupling}, it remains to prove the uniform convergence in $d$ of the series on the right hand-side of \eqref{UF9} by appropriately choosing the parameters $\alpha$ and $(N,a)$. Notice that each term in the series is positive. Thus, we only need to consider the case $d=1$:
		\[\sum_{r,\rho\ge1}e^{r\|\VVV\|_{\infty}}\left(e^{-\frac{c_2}{2}\rho}\wedge\left(e^{-\frac{c_1}{2}\alpha r}+\left[\exp\left(C_{R,N,a}e^{c_3\rho-c_1\alpha r}\right)-1\right]^{\frac{1}{4}}\right)\right)<\infty.\]
		Let us define $\SSSSS:=\{(r,\rho)| \rho\le \frac{c_1\alpha r}{2c_3}\}$. By choosing 
		\[\alpha=\frac{16\|\VVV\|_{\infty}}{c_1}\vee \frac{8c_3\|\VVV\|_\infty}{c_1c_2},\qquad N=N_0(\alpha),\qquad a=\frac{N^2}{2},\]
		we see that
		\begin{align*}
			\sum_{(r,\rho)\in \SSSSS}&e^{r\|\VVV\|_{\infty}}\left(e^{-\frac{c_1}{2}\alpha r}+\left[\exp\left(C_{R,N,a}e^{c_3\rho-c_1\alpha r}\right)-1\right]^{\frac{1}{4}}\right)\\&\lesssim_{R,N,a} \sum_{r,\rho\in \SSSSS}e^{r\|\VVV\|_{\infty}} e^{-\frac{c_1\alpha r}{8}}\lesssim_{R,\VVV,N,a}\sum_{r\ge 1}re^{-r\|\VVV\|_{\infty}} <\infty,
		\end{align*}
		and 
		\begin{align*}
			\sum_{(r,\rho)\notin \SSSSS}e^{r\|\VVV\|_{\infty}}e^{-\frac{c_2}{2}\rho}\le \sum_{(r,\rho)\notin \SSSSS} e^{-\frac{c_2\rho}{4}}\lesssim_{\VVV}\sum_{\rho\ge 1}\rho e^{-\frac{c_2\rho}{4}}<\infty. 
		\end{align*}
		This shows the uniform equicontinuity of $\{g_t\}_{t\ge 0}$ on $X_{R}^T$.
	\end{proof}
	\subsection{Uniform tail probability estimate}
	The goal of this subsection is to establish the uniform tail probability estimate required in Proposition~\ref{propositionE2} for the Feynman--Kac semigroup.
	\begin{proposition}\label{propositionUTPE}
		Let $\VVV\in C_b(X^T)$ and $m\ge1$. Then, for any integer $\RR\ge 1$ and $t\ge T+1$,
		\[\sup_{\uu\in \overline{B_{X^T}(\RR)}}\int_{(X_R^T)^c}\wwww_m(\vv)P^{T,\VVV}_t(\uu,\dd\vv)\to0,\]
		as $R\to+\infty$.
	\end{proposition}
	\begin{proof}
		Let us fix $\VVV\in C_b(X^T)$, $t\ge T+1$, and $m,\RR\ge 1$. Applying the definition of $P_t^{T,\VVV}$ and the Cauchy--Schwartz inequality, we get
		\begin{align}
			\label{UTPE1}\int_{(X_R^T)^c}\wwww_m(\vv)P^{T,\VVV}_t(\uu,\dd\vv)&\le \frac{e^{t\|\VVV\|_\infty}}{\sqrt{g(R)}}\int_{(X^T_R)^c}\wwww_m(\vv) \sqrt{g\left(\|\vv\|_{X^T_\infty}\right)}P_t^T(\uu,\dd \vv)\notag\\&\le \frac{e^{t\|\VVV\|_\infty}}{\sqrt{g(R)}}\left(\E_{\uu}\wwww_{2m}(\uu_t)\right)^{\frac{1}{2}}\left(\E_{\uu}g\left(\|\uu_t\|_{X^T_\infty}\right)\right)^{\frac{1}{2}},
		\end{align}
		where $g(x):=\log(1+\log(1+x))$. Moreover, utilizing the estimates \eqref{HER5} and \eqref{GC11} combined with the Markov property, we have 
		\begin{align}\label{UTPE2}\E_{\uu}\log\left(1+\log\left(1+\|\uu_{t}\|_{X^T_{\infty}}\right)\right)&= \E_{\uu}\left(\E_{\uu_{t-T-1}}\log\left(1+\log\left(1+\|\uu_{T+1}\|_{X^T_{\infty}}\right)\right)\right)\notag\\&\lesssim_{\BB_1,h,T}\E_{\uu}\wwww_1(\uu_{t-T-1})\lesssim_{\BB_1,h,T} \wwww_1(\uu),
		\end{align}
		which implies 
		\begin{align*}
			\left(\E_{\uu}\wwww_{2m}(\uu_t)\right)^{\frac{1}{2}}\left(\E_{\uu}g\left(\|\uu_t\|_{X^T_\infty}\right)\right)^{\frac{1}{2}}\lesssim_{m,\BB_1,h,T} \sqrt{\wwww_{2m}(\uu)\wwww_1(\uu)}.
		\end{align*}
		This, together with \eqref{UTPE1}, yields
		\[\sup_{\uu\in \overline{B_{X^T}(\RR)}}\int_{(X_R^T)^c}\wwww_m(\vv)P^{T,\VVV}_t(\uu,\dd\vv)\lesssim_{\VVV,t,\RR,m,\BB_1,h,T}\frac{1}{\sqrt{g(R)}}\to0,\]
		as $R\to+\infty$, thus completing the proof.
	\end{proof}
 
	\subsection{Existence of an eigenvector}
	Here, we show that the dual operator $\PPPP^{T,\VVV*}_t$ of the Feynman--Kac semigroup has at least one eigenvector and establish an exponential moment estimate for it.
	\begin{proposition}\label{propositionEE}
		For any $\VVV\in C_b(X^T)$ and $t>0$, there is a measure $\mu_{t,\VVV}\in\PP(X^T)$ and a positive constant $c_{t,\VVV}$ such that 
		\begin{align}
			\label{EE1}\PPPP^{T,\VVV*}_t\mu_{t,\VVV}=c_{t,\VVV}\mu_{t,\VVV}.
		\end{align}
		Moreover, any such measure satisfies
		\begin{align}\label{EE3}
			\int_{X^T}\exp\left(\frac{\gamma_0}{2}\|\uu\|_{X^T}^2\right)\mu_{t,\VVV}(\dd \uu)<\infty,
		\end{align}
		where $\gamma_0$ is the constant given in Lemma~\ref{lemmaA1}.
	\end{proposition}
	\begin{proof}\textit{Step 1: Exponential moment estimate.} Let us fix $\VVV\in C_b(X^T)$ and $t>0$. Suppose that $\mu_{t,\VVV}\in\PP(X^T)$ satisfies \eqref{EE1} and $k\ge 1$ is an integer such that $kt\ge T+1$. From \eqref{A4-3}, we have
		\begin{align}\label{EE2}
			\E_{\uu}\left(\|\uu_{kt}\|^2_{X^T}\exp\left(\frac{\gamma_0}{2}\|\uu_{kt}\|^2_{X^T}\right)\right)\lesssim_{\BB_0,h,t,T}\exp\left(\frac{\gamma_0}{2}\|\uu\|_{X^T}^2\right).	
		\end{align}
		For simplicity, let us write
		\[\fff(\uu):=\exp\left(\frac{\gamma_0}{2}\|\uu\|_{X^T}^2\right).\]
		Using \eqref{EE1} and \eqref{EE2}, we obtain
		\begin{align*}
			\langle \fff ,\mu_{t,\VVV}\rangle&=c_{t,\VVV}^{-k}\langle\PPPP_{kt}^{T,\VVV}\fff,\mu_{t,\VVV}\rangle\\&\le c_{t,\VVV}^{-k}e^{kt\|\VVV\|_\infty}\int_{X^T}\E_{\uu}\fff(\uu_{kt}) \mu_{t,\VVV}(\dd \uu)\\&\lesssim_{t,T,\VVV}\int_{X^T}\E_{\uu}\left(\fff(\uu_{kt}) \I_{\{\|\uu_{kt}\|_{X^T}\le R\}}+\fff(\uu_{kt}) \I_{\{\|\uu_{kt}\|_{X^T}> R\}}\right)\mu_{t,\VVV}(\dd \uu)\\&\lesssim_{t,T, \VVV}\int_{X^T}\left(e^{\frac{\gamma_0 R^2}{2}}+R^{-2}\E_{\uu}\left(\|\uu_{kt}\|^2_{X^T}\fff(\uu_{kt})\right)\right)  \mu_{t,\VVV}(\dd \uu)\\&\lesssim_{t,T,\VVV,\BB_0,h}\int_{X^T}\left(e^{\frac{\gamma_0 R^2}{2}}+R^{-2}\fff(\uu)\right)  \mu_{t,\VVV}(\dd \uu).
		\end{align*}
		Therefore, by choosing $R$ sufficiently large, we see that 
		\[\langle \fff,\mu_{t,\VVV}\rangle\le C_{t,T,\VVV,\BB_0,h}<\infty.\]
		\noindent\textit{Step 2: Existence of an eigenvector for $t\ge T+1$.} For integers $l,m\ge 1$, we define
		\[D_{m,l}:=\{\bmnu\in\PP(X^T)|\langle\wwww_m,\bmnu\rangle\le l\}.\]
		By the Fatou lemma, it is clear that $D_{m,l}$ is a closed and convex subset of $\PP(X^T)$. For $t\ge T+1$ and $\VVV\in C_b(X^T)$, we introduce the following continuous map:
		\begin{align*}
			G_{t,\VVV}:&\ D_{m,l}\to \PP(X^T),\\
			&\bmnu\mapsto \frac{\PPPP_t^{T,\VVV*}\bmnu}{\PPPP_t^{T,\VVV*}\bmnu(X^T)}.
		\end{align*}
		We claim that for appropriately chosen $m,l\ge1$, $G_{t,\VVV}(D_{m,l})\subset D_{m,l}$. Indeed, for $\bmnu\in D_{m,l}$, by using \eqref{A4-2}, we have
		\begin{align*}
			\langle \wwww_m,G_{t,\VVV}\bmnu\rangle&\le \exp(t\Osc(\VVV))\langle \PPPP_t^T\wwww_m,\bmnu\rangle\\&\le \exp(t\Osc(\VVV))\int_{X^T}(C_{m,\BB_0,h,T}+4e^{-m\alpha_1(t-T)}\|\uu\|^{2m}_{X^T})\bmnu(\dd \uu),
		\end{align*}
		from which, by choosing $m,l$ sufficiently large so that
		\[4e^{-m\alpha_1(t-T)}\exp(t\Osc(\VVV))\le \frac{1}{2},\qquad l\ge 2\exp(t\Osc(\VVV))C_{m,\BB_0,h,T},\]
		we have $G_{t,\VVV}\bmnu\in D_{m,l}$. 
		
		In view of the Leray--Schauder theorem, there is a measure $\bmnu\in D_{m,l}$ satisfying \eqref{EE1} with 
		\[c_{t,\VVV}:=\PPPP_t^{T,\VVV*}\bmnu(X^T)>0,\]
		provided that $G_{t,\VVV}(D_{m,l})$ is pre-compact. To this end, we use the estimate \eqref{UTPE2} to get
		\begin{align*}
			\int_{X^T}&\log\left(1+\log\left(1+\|\uu\|_{X^T_{\infty}}\right)\right)G_{t,\VVV}\bmnu(\dd \uu)\\&\lesssim_{t,\VVV}\int_{X^T}\E_{\uu}\log\left(1+\log\left(1+\|\uu_t\|_{X^T_{\infty}}\right)\right)\bmnu(\dd \uu)\\&\lesssim_{t,\VVV,\BB_1,h,T}\int_{X^T}\wwww_1(\uu)\bmnu(\dd \uu)\lesssim_{t,\VVV,\BB_1,h,T}\langle \wwww_m,\bmnu\rangle<\infty
		\end{align*}
		for any $\bmnu\in D_{m,l}$. This, together with the Prokhorov theorem and the Chebyshev inequality, implies the pre-compactness of $G_{t,\VVV}(D_{m,l})$ and thus the existence of an eigenvector for $t\ge T+1$.
		
		\noindent\textit{Step 3: Existence of an eigenvector for $t<T+1$.} Let $k\ge 1$ be such that $kt\ge T+1$, and let $\mu_{kt,\VVV}\in \PP(X^T)$ denote the measure satisfying \eqref{EE1} with $c_{kt,\VVV}>0$. Consider the following Bogolyubov averaging:
		\[\mu_{t,\VVV}:=\sum_{l=0}^{k-1}\left(c_{kt,\VVV}^{-\frac{1}{k}}\PPPP^{T,\VVV*}_{t}\right)^l\mu_{kt,\VVV}.\]
		Utilizing \eqref{EE1} combined with the above definition, we see that 
		\[\PPPP_{t}^{T,\VVV*}\mu_{t,\VVV}=c_{kt,\VVV}^{\frac{1}{k}}\mu_{t,\VVV}.\] 
		This completes the proof.
	\end{proof}
 
	\subsection{Multiplicative ergodic theorem}
	In this subsection, we apply Proposition~\ref{propositionE2} combined with the results established previously to obtain a detailed description of the large-time behaviour of the Feynman--Kac semigroup. This result will be used to verify the conditions in the Kifer-type criterion Proposition~\ref{proposition-kifer}.
	
	Using Proposition~\ref{propositionGC}, we see that $\{\PPPP_t^{T,\VVV}\}_{t\ge0}$ forms a semigroup of bounded linear operators in $C_{\wwww,m}(X^T)$ for sufficiently large $m$. As introduced in Section~\ref{level2}, we say $c\in \R$ is an eigenvalue of $\{\PPPP_t^{T,\VVV}\}_{t\ge0}$, if there is a function $\fff\in C_{\wwww,m}(X^T)$ such that 
	\[\PPPP_t^{T,\VVV}\fff=c^t\fff,\qquad\forall t>0.\]
	Similarly, for the dual semigroup, $c\in\R$ is said to be an eigenvalue, if there is a measure $\bmnu\in \MM^+(X^T)$ satisfying
	\[\PPPP^{T,\VVV*}_t\bmnu=c^t\bmnu,\qquad\forall t>0.\]
	We recall that $\VV^T$ is the space of functions in the form \eqref{UF0}, and write 
    \[\PP_{\wwww,m}(X^T):=\{\bmnu\in\PP(X^T)|\langle \wwww_m,\bmnu\rangle<\infty\}.\]
	\begin{theorem}
		For any $\VVV\in \VV^T$, there is an integer $m=m(\VVV)\ge 1$ such that the following assertions hold.
		\begin{enumerate}
			\item There is a unique eigenvalue $c=c_{\VVV}>0$ of $\{\PPPP^{T,\VVV*}_t\}_{t\ge0}$ associated with a unique eigenvector $\mu_{\VVV}\in \PP(X^T)$. Moreover, $\mu_{\VVV}$ satisfies the exponential moment estimate \eqref{EE3}.
			\item The number $c_{\VVV}$ is the unique principal eigenvalue of $\{\PPPP_t^{T,\VVV}\}_{t\ge0}$ and there is a unique eigenfunction $ h_{\VVV}\in C_{\wwww,m}(X^T)$ associated with $c_{\VVV}$ and normalized by the condition $\langle  h_{\VVV},\mu_{\VVV}\rangle=1.$ Moreover, $h_{\VVV}$ is positive on $X^T$.
			\item For any $\fff\in C_{\wwww,m}(X^T)$, $\bmnu\in \PP_{\wwww,m}(X^T)$, and $R>0$, we have 
			\begin{gather}
				\label{MET1}
				c_{\VVV}^{-t}\PPPP^{T,\VVV}_t\fff\to \langle \fff,\mu_{\VVV}\rangle  h_{\VVV},\qquad \mbox{in $C_b(\overline{B_{X^T}(R)})\mcap L^1(X^T,\mu_{\VVV})$},\\
				\label{MET2}
				c_{\VVV}^{-t}\PPPP^{T,\VVV*}_t\bmnu\to \langle h_{\VVV},\bmnu\rangle\mu_{\VVV},\qquad \mbox{in $\MM^+(X^T)$},
			\end{gather}
			as $t\to+\infty$. Moreover, for any $\gamma,M>0$, the convergence
			\begin{align}
				\label{MET2-1}c_{\VVV}^{-t}\E_{\bmnu}\left(\fff(\uu_t)\exp\left(\int_0^t\VVV(\uu_s)\dd s\right)\right)\to\langle\fff,\mu_{\VVV}\rangle\langle h_{\VVV},\bmnu\rangle,
			\end{align}
			as $t\to+\infty$, holds uniformly in 
			\begin{align}
				\label{MET2-2}\bmnu\in\bmLambda(\gamma,M):=\left\{\bmnu\in\PP(X^T)\Big|\int_{X^T}\exp\left(\gamma\|\uu\|_{X^T}^2\right)\bmnu(\dd \uu)\le M\right\}.
			\end{align}
		\end{enumerate}
	\end{theorem}
	\begin{proof} \textit{Step 1. Checking the conditions of Proposition~\ref{propositionE2}.} In view of Propositions~\ref{propositionGC}--\ref{propositionDC}, \ref{propositionUI}, \ref{propositionUF}, and \ref{propositionUTPE}, it remains to show that $\VV^T$ is a determining family\footnote{Let us recall that a family $\CC\subset C_b(X^T)$ is determining, if any two measures $\mu,\nu\in\MM^+(X^T)$ satisfying $\langle f,\mu\rangle=\langle f,\nu\rangle$ for any $f\in \CC$ coincide.}. Indeed, by the Dini theorem, cf. Theorem 1.7.10 in \cite{Bog2007}, we have $\ppP_N\uu \to \uu$ in $X^T$, which implies that for any $\F\in L_b(X^T)$, $\F(\ppP_N\uu)\to \F(\uu)$ in the sense of buc-convergence\footnote{Here, we recall that the definition of buc-convergence is given in \eqref{buc}.} and pointwise convergence, as $N\to+\infty$. It is well-known that any $\VVV\in C_b(X^T)$ can be approximated by functions $\F_n\in L_b(X^T)$ in the sense of these types of convergence. Combining this and the dominated convergence theorem, we see that $\VV^T$ is a determining family. 
		
		Now, by applying Propositions~\ref{propositionE2} and \ref{propositionEE}, we have the existence part of the first two statements, as well as the convergence \eqref{MET1} for any $\fff\in \CC$ and $R>0$. The uniqueness part follows directly from \eqref{MET1} and \eqref{MET2}, which together with \eqref{MET2-1} are established in the next four steps.
		
		\noindent\textit{Step 2: Proof of \eqref{MET1} for $\fff\in C_b(X^T)$.} We claim that the convergence \eqref{MET1} holds for all functions $\fff$ that can be approximated, in the sense of buc-convergence and pointwise convergence, by a sequence $\{\fff_n\}_{n\ge 1}$ satisfying \eqref{MET1}. Indeed, let us fix $R>0$ and set 
		\[\Delta_t(\fff)=\left\|c_{\VVV}^{-t}\PPPP^{T,\VVV}_t\fff-\langle \fff,\mu_{\VVV}\rangle  h_{\VVV}\right\|_{0,R},\qquad \|\fff\|_{0,R}:=\sup_{\uu\in \overline{B_{X^T}(R)}}|\fff(\uu)|.\]
		Then, 
		\begin{align*}
			\Delta_t(\fff)\le \Delta_t(\fff_n)+|\langle \fff-\fff_n,\mu_{\VVV}\rangle|\| h_{\VVV}\|_{0,R}+c_{\VVV}^{-t}\|\PPPP^{T,\VVV}_t(\fff-\fff_n)\|_{0,R}.
		\end{align*}
		Since $\Delta_t(\fff_n)\to 0$, as $t\to+\infty$, for any fixed integer $n\ge 1$, and 
		\[\lim_{n\to+\infty}|\langle \fff-\fff_n,\mu_{\VVV}\rangle|=0,\]
		it remains to establish
		\begin{align}
			\label{MET3}
			\lim_{n\to+\infty}\sup_{t\ge T+1}c_{\VVV}^{-t}\|\PPPP^{T,\VVV}_t(\fff-\fff_n)\|_{0,R}=0.
		\end{align}
		To this end, let us take any integer $\rho\ge 1$. Notice that 
		\begin{align}
			\label{MET4}
			c_{\VVV}^{-t}\|\PPPP^{T,\VVV}_t&(\fff-\fff_n)\|_{0,R}\notag\\&\le c_{\VVV}^{-t}\|\PPPP^{T,\VVV}_t(\I_{X^{T}_{\rho}}(\fff-\fff_n))\|_{0,R}+c_{\VVV}^{-t}\|\PPPP^{T,\VVV}_t(\I_{(X^{T}_{\rho})^c}(\fff-\fff_n))\|_{0,R}\notag\\&=:J_1(n,t,\rho)+J_2(n,t,\rho).
		\end{align}
		For $J_1$, we apply \eqref{GC5} to obtain
		\begin{align}
			\label{MET5}
			J_1(n,t,\rho)&\le c_{\VVV}^{-t}\|\fff-\fff_n\|_{\rho}\|\PPPP^{T,\VVV}_t\mathbf{1}\|_{0,R}\lesssim_R c_{\VVV}^{-t}\|\fff-\fff_n\|_{\rho}\|\PPPP^{T,\VVV}_t\mathbf{1}\|_{L^\infty_{\wwww,m}}\notag\\&\lesssim_{R} c_{\VVV}^{-t}\|\fff-\fff_n\|_{\rho}\|\PPPP^{T,\VVV}_t\mathbf{1}\|_{R_0}.
		\end{align}
		Convergence \eqref{MET1} with $\fff=\mathbf{1}$ gives
		\begin{align}\label{MET5-1}
			\sup_{t\ge 0}c_{\VVV}^{-t}\|\PPPP^{T,\VVV}_t\mathbf{1}\|_{R_0}<\infty.
		\end{align}
		Combining this with \eqref{MET5}, we have 
		\begin{align}
			\label{MET6}
			\lim_{n\to+\infty}\sup_{t\ge0}J_1(n,t,\rho)=0,\qquad\forall \rho\ge 1.
		\end{align}
        To estimate $J_2$, let us write 
		\[\gggg(\uu):=\log(1+\log(1+\|\uu\|_{X^T_\infty})),\qquad g(x):=\log(1+\log(1+x)).\]
		Notice that
		\begin{align}
			\label{MET7}
			J_2(n,t,\rho)\lesssim_{\fff} \frac{1}{g(\rho)}c_{\VVV}^{-t}\|\PPPP_t^{T,\VVV}\gggg\|_{0,R}.
		\end{align}
		Applying the Markov property combined with the estimates \eqref{GC2} and \eqref{UTPE2}, for $\uu\in \overline{B_{X^T}(R)}$ and $t\ge T+1$, we have 
		\begin{align*}
			\PPPP_t^{T,\VVV}\gggg(\uu) &\lesssim_{\VVV,T}\E_{\uu}\left(\exp\left(\int_0^{t-T-1}\VVV(\uu_s)\dd s\right)\E_{\uu_{t-T-1}}\gggg(\uu_{T+1})\right)\\&\lesssim_{\VVV,T,\BB_1,h}\E_{\uu}\left(\exp\left(\int_0^{t-T-1}\VVV(\uu_s)\dd s\right)\wwww_m(\uu_{t-T-1})\right)\\&\lesssim_{\VVV,T,\BB_1,h,R}\|\PPPP_{t-T-1}^{T,\VVV}\wwww_m\|_{L^\infty_{\wwww,m}}\lesssim_{\VVV,T,\BB_1,h,R}\|\PPPP_{t}^{T,\VVV}\mathbf{1}\|_{R_0}.
		\end{align*}
		This, together with \eqref{MET5-1} and \eqref{MET7}, gives
		\begin{align}
			\label{MET8}
			\lim_{\rho \to+\infty}\sup_{n\ge 1}\sup_{t\ge T+1}J_2(n,t,\rho)=0.
		\end{align}
		Combining \eqref{MET4}, \eqref{MET6}, and \eqref{MET8}, we obtain \eqref{MET3} and thus the claimed result. Utilizing this combined with the approximation results established in the previous step, we see that \eqref{MET1} holds for $\fff\in C_b(X^T)$. 
		
		\noindent\textit{Step 3: Proof of \eqref{MET1} for $\fff\in C_{\wwww,m}(X^T)$.} Let us define
		\[\fff_n:=\fff^+\wedge n-\fff^-\wedge n.\]
		Notice that 
		\[\fff_n\in C_b(X^T),\qquad |\fff_n|\le |\fff|,\qquad\forall n\ge 1,\]
		and $\fff_n\to\fff$ in $C_{\wwww,m+1}(X^T)$, as $n\to\infty$, which implies 
		\[\lim_{n\to+\infty}|\langle\fff-\fff_n,\mu_{\VVV} \rangle|\to0.\]
		Therefore, as in the previous step, it suffices to establish \eqref{MET3}. To see this, we apply \eqref{GC2} with $m+1$ to have
		\begin{align*}
			c_{\VVV}^{-t} \|\PPPP^{T,\VVV}_t(\fff-\fff_n)\|_{0,R}&\lesssim_R \|\fff-\fff_n\|_{L^{\infty}_{\wwww,m+1}} c_{\VVV}^{-t} \|\PPPP^{T,\VVV}_t\wwww_{m+1}\|_{L^\infty_{\wwww,m+1}}\\&\lesssim_{R} \|\fff-\fff_n\|_{L^{\infty}_{\wwww,m+1}} c_{\VVV}^{-t} \|\PPPP^{T,\VVV}_t\mathbf{1}\|_{R_0},
		\end{align*}
		which, together with \eqref{MET5-1}, implies \eqref{MET3} and thus \eqref{MET1} for $\fff\in C_{\wwww,m}(X^T)$.
		
		\noindent\textit{Step 4: Proof of \eqref{MET2} and \eqref{MET2-1}.} Let $\fff\in C_{\wwww,m}(X^T)$. By using \eqref{GC2},
		\begin{align*}
			c_{\VVV}^{-t}\PPPP^{T,\VVV}_t\fff(\uu)&\le c_{\VVV}^{-t}\|\fff\|_{L^{\infty}_{\wwww,m}}\PPPP^{T,\VVV}_t\wwww_m(\uu)\\&\le \|\fff\|_{L^{\infty}_{\wwww,m}}\wwww_{m}(\uu)c_{\VVV}^{-t}\|\PPPP_t^{T,\VVV}\wwww_m\|_{L^\infty_{\wwww,m}}\\&\lesssim\|\fff\|_{L^{\infty}_{\wwww,m}}\wwww_{m}(\uu)\sup_{t\ge0 }\left(c_{\VVV}^{-t}\|\PPPP_t^{T,\VVV}\mathbf{1}\|_{R_0}\right),\qquad\forall \uu\in X^T,
		\end{align*}
		which, together with the fact $h_{\VVV}\in C_{\wwww,m}(X^T)$, implies
		\begin{align}\label{MET10}
			|\langle &c_{\VVV}^{-t}\PPPP_t^{T,\VVV}\fff,\bmnu\rangle-\langle h_{\VVV},\bmnu\rangle\langle \fff,\mu_{\VVV}\rangle|\notag\\&\lesssim_{\VVV,\fff} \left|\langle \I_{B_{X^T}(R)}\left(c_{\VVV}^{-t}\PPPP_t^{T,\VVV}\fff-\langle\fff,\mu_{\VVV}\rangle h_{\VVV}\right),\bmnu\rangle\right|\notag\\&\quad+|\langle\I_{B_{X^T}(R)^c}\wwww_m,\bmnu\rangle|+|\langle\fff,\mu_{\VVV}\rangle||\langle \I_{B_{X^T}(R)^c} h_{\VVV},\bmnu\rangle|\notag\\&\lesssim_{\VVV,\fff} \sup_{\uu\in B_{X^T}(R)}\left|c_{\VVV}^{-t}\PPPP_t^{T,\VVV}\fff(\uu)-\langle\fff,\mu_{\VVV}\rangle h_{\VVV}(\uu)\right|+\langle\I_{B_{X^T}(R)^c}\wwww_m,\bmnu\rangle.
		\end{align}
		The first term on the right hand-side of \eqref{MET10} goes to $0$, as $t\to+\infty$, because of \eqref{MET1}. As $\bmnu\in\PP_{\wwww,m}(X^T)$, the second term can be made arbitrarily small by choosing $R>0$ sufficiently large. Therefore, the convergence \eqref{MET2} holds.
		
		To get \eqref{MET2-1}, we note that for $\bmnu\in\bmLambda(\gamma,M)$,
		\begin{align*}
			\langle\I_{B_{X^T}(R)^c}\wwww_m,\bmnu\rangle&\le \frac{1}{1+R^2}\langle\wwww_1\wwww_m,\bmnu\rangle\\&\lesssim_{\VVV,\gamma} \frac{1}{1+R^2}\int_{X^T}\exp\left(\gamma\|\uu\|_{X^T}^2\right)\bmnu(\dd \uu)\lesssim_{\VVV,\gamma}\frac{M}{1+R^2},
		\end{align*}
		which, together with taking supremum in $\bmnu$ in both sides of \eqref{MET10}, yields the convergence \eqref{MET2-1}.
	\end{proof}

	\section{Proof of Theorem~\ref{MT2}}\label{proofMT2}
	First, we reduce the proof to a uniform LDP for appropriately shifted empirical measures. Let $(u_t,\PPPPPP_u)$ be the Markov family associated with \eqref{I1}, and let $\bmzeta_t^T$ be defined in \eqref{MR7}. We introduce a shifted version of $\bmzeta_t^T$:
	\begin{align}
		\label{MT1-1}
		\bar\bmzeta_t^T:=\frac{1}{t}\int_0^t \delta_{u_{[s-T,s]}}\dd s,
	\end{align}
    where 
    \[u_{t}\equiv u_0,\qquad t\in[-T,0].\]
	Using similar arguments as in Section A.2 in \cite{JNPS15}, for any $\gamma,M>0$, $\{\bmzeta_t^T\}_{t>0}$ and $\{\bar\bmzeta_t^T\}_{t>0}$ are uniformly exponentially equivalent in $\sigma\in \Lambda(\gamma,M)$ with the set $\Lambda(\gamma,M)$ defined in \eqref{MR0}. Notice that the distribution of $\{u_{[s-T,s]}\}_{s\ge 0}$ under $\PPPPPP_{u}$ coincides with that of $\{\uu_t\}_{t\ge 0}$ under $\PPPPPP_{\uu}$ with $\uu\equiv u$ and $(\uu_t,\PPPPPP_{\uu})$ being the $T$-Markov family defined by \eqref{FK1-1} and \eqref{FK1}. Therefore, it suffices to establish a uniform LDP for 
	\[\tilde\bmzeta_t^T:=\frac{1}{t}\int_0^t\delta_{\uu_s}\dd s\]
	in $\bmnu\in \PP_c(X^T)\mcap\bmLambda(\gamma,M)$,
	where $\bmLambda(\gamma,M)$ is defined in \eqref{MET2-2} and $\PP_c(X^T)$ denotes the set of all Borel probability measures $\bmnu$ on $X^T$ satisfying 
	\[\bmnu\{\uu\in X^T| \uu\equiv \uu(T)\}=1.\]
    
	Next, we apply the Kifer-type criterion Proposition~\ref{proposition-kifer} to get a uniform LDP for $\{\tilde\bmzeta_t^T\}_{t>0}$. To this end, we introduce
	\[\Theta:=\{\theta=(t,\bmnu)|t\ge0, \bmnu\in \PP_c(X^T)\mcap \bmLambda(\gamma,M)\},\qquad r(\theta)=t\]
	with an order relation $\prec$ such that $(t_1,\bmnu_1)\prec (t_2,\bmnu_2)$ if and only if $t_1\le t_2$. It is straightforward to check that the convergence in $\theta\in\Theta$ is equivalent to the convergence as $t\to+\ty$, which is uniform in $\bmnu\in\PP_c(X^T)\mcap \bmLambda(\gamma,M)$. We divide the verification of the conditions in Proposition~\ref{proposition-kifer} into three steps.
    
	\noindent\textit{Step 1: Exponential tightness along any subsequence.} Combining Proposition~\ref{proposition4-4} and the uniform exponential equivalence between $\{\bmzeta_t^T\}_{t>0}$ and $\{\tilde\bmzeta_t^T\}_{t>0}$, we see that for any sequence $\{t_n\}_{n\ge 1}$ satisfying $t_n\to+\infty$, as $n\to+\infty$, the family $\{\tilde\bmzeta_{t_n}^T\}_{n\ge 1}$ is uniformly exponentially tight in $\bmnu\in\PP_c(X^T)\mcap \bmLambda(\gamma,M)$, which verifies the first condition.
	
	\noindent\textit{Step 2: Existence of pressure.} Taking $\fff=\mathbf{1}$ in \eqref{MET2-1}, we see that 
	\begin{align}\label{MT1-2}
		Q_T(\VVV)=\log c_{\VVV},\qquad\VVV\in\VV^T.
	\end{align}
    For the general case $\VVV\in C_b(X^T)$, we notice that the set of functions $\VVV$ for which the pressure function $Q_T$ exists is closed under buc-convergence, and any function in $C_b(X^T)$ can be buc-approximated by functions in $\VV^T$. This, together with \eqref{MT1-2}, implies the existence of pressure function $Q_T$ in the general case.
	
	\noindent\textit{Step 3: Uniqueness of equilibrium.} Let us show that for any $\VVV\in \VV^T$, there is only one equilibrium state $\nu_{\VVV}$. To this end, we introduce the following generalized Markov semigroup 
	\[\SSS_t^{T,\VVV,\F}\fff:=c_{\VVV}^{-t} h_{\VVV}^{-1}\PPPP_t^{T,\VVV+\F}( h_{\VVV}\fff),\qquad \fff,\F\in C_{b}(X^T),\]
	and $\SSS_t^{T,\VVV}:=\SSS_t^{T,\VVV,\mathbf{0}}$. Combining \eqref{MET1} and the argument used in the previous step, we see that the pressure function is well-defined:
	\begin{align}
		\label{MT1-3}Q_T^{\VVV}(\F):=\lim_{t\to+\infty}\frac{1}{t}\log\SSS_t^{T,\VVV,\F}\mathbf{1},\qquad \F\in C_b(X^T).
	\end{align}
	Let $I_T,I_T^{\VVV}$ denote the Legendre transform\footnote{For the definition of Legendre transform, we refer to the formula \eqref{kifer2}.} of $Q_T,Q^{\VVV}_T$, respectively. Using \eqref{MET1}, \eqref{MT1-2}, and \eqref{MT1-3}, we obtain
	\[Q^{\VVV}_T(\F)=Q_T(\VVV+\F)-Q_T(\VVV),\qquad \forall \F\in\VV^T,\]
	which implies\footnote{The supremum in the definition of Legendre transform can be taken over $\VV^T$, which can be proved as in Lemma~\ref{lemma1}.}
	\[I^{\VVV}_T(\bmnu)=I_T(\bmnu)+Q_T(\VVV)-\langle\VVV,\bmnu\rangle.\]
	Therefore, any equilibrium state $\nu_{\VVV}$ is a zero of $I^{\VVV}_T$. 
	
	We claim that $I^{\VVV}_T$ has a unique zero 
    \begin{align}\label{MT1-4}\nu_{\VVV}= h_{\VVV}\mu_{\VVV}.\end{align}
    To see this, we introduce the generator $\LLLLL_{\VVV}$ of the semigroup $\{\SSS^{T,\VVV}_t\}_{t\ge 0}$ defined on the domain
	\begin{align*}
		\mathcal{D}(\LLLLL_{\VVV}):=\left\{\fff\in C_b(X^T)\Big|\exists \gggg\in C_b(X^T),\ \!\SSS^{T,\VVV}_t\fff-\fff=\int_0^t\SSS^{T,\VVV}_s\gggg\dd s,\ \!\forall t\ge 0\right\}
	\end{align*}
	with $\LLLLL_{\VVV}\fff:=\gggg$ for $\fff\in \mathcal{D}(\LLLLL_{\VVV})$. We need the following auxiliary lemmas. 
	\begin{lemma}\label{LEMMAMT1-1}
		For $\F,\fff\in C_b(X^T)$,
		\[\SSS_t^{T,\VVV,\F}\fff-\SSS^{T,\VVV}_t \fff=\int_0^t\SSS^{T,\VVV}_{t-s}(\F\SSS^{T,\VVV,\F}_s\fff)\dd s,\qquad \forall t\ge0.\]
	\end{lemma}
	\begin{lemma}\label{LEMMAMT1-2}
		$\DD_+(\LLLLL_{\VVV}):=\{\fff\in \DD(\LLLLL_{\VVV})| \inf_{X^T}\fff>0\}$ is a determining family.
	\end{lemma}
	\begin{lemma}\label{LEMMAMT1-3}
		$\{\SSS^{T,\VVV}_t\}_{t\ge0}$ has a unique invariant measure given by \eqref{MT1-4}.
	\end{lemma}
	Lemma~\ref{LEMMAMT1-1} follows in the same way as in Lemma~\ref{lemma3}, while Lemma~\ref{LEMMAMT1-2} can be established as Lemma A.7 in \cite{MN18}. Moreover, Lemma~\ref{LEMMAMT1-3} is a direct result of \eqref{MET1}, see Lemma 5.9 in \cite{JNPS18}. Thus, we omit the proofs.
	
	Let us show that for any $\fff\in \DD_+(\LLLLL_{\VVV})$,
	\begin{align}
		\label{MT1-5}
		Q^{\VVV}_T\left(\F_\fff\right)=0,
	\end{align}
	where $\F_\fff:=-\LLLLL_{\VVV}\fff/\fff$. Indeed, by Lemma~\ref{LEMMAMT1-1} and the Gronwall inequality, we have 
	\[\SSS_t^{T,\VVV,\F_{\fff}}\fff\equiv \fff,\qquad \forall t>0,\]
	which implies 
	\[\lim_{t\to+\infty}\frac{1}{t}\log\SSS_t^{T,\VVV,\F_{\fff,}}\fff=0.\]
	As $c\le \fff\le C$ for some constants $c,C>0$, we obtain \eqref{MT1-5}.
	
	To conclude the proof, let us assume that $\nu_{\VVV}$ is a zero of $I^{\VVV}_T$. Then, for any $\fff\in \DD_+(\LLLLL_{\VVV})$,
	\[0=I^{\VVV}_T(\nu_{\VVV})\ge\langle \F_{\fff},\nu_{\VVV}\rangle=-\int_{X^T}\frac{\LLLLL_{\VVV}\fff}{\fff}\nu_{\VVV}(\dd\uu).\]
	Combining this with $\LLLLL_{\VVV}\mathbf{1}\equiv0$ and $\mathbf{1}+c\fff\in\DD_+(\LLLLL_{\VVV})$ for sufficiently small $c$, we see that
	\[\psi(c):=\int_{X^T}\frac{\LLLLL_{\VVV}(\mathbf{1}+c\fff)}{(\mathbf{1}+c\fff)}\nu_{\VVV}(\dd\uu)\]
	has a local minimum $c=0$. Therefore, 
	\[0=\psi'(0)=\int_{X^T}\LLLLL_{\VVV}\fff\nu_{\VVV}(\dd \uu),\qquad\forall \fff\in \DD_+(\LLLLL_{\VVV}).\]
    This implies
	\begin{align*}
	    \int_{X^T}\SSS_t^{T,\VVV}\fff\nu_{\VVV}(\dd \uu)-\int_{X^T}\fff\nu_{\VVV}(\dd \uu)=0,\qquad \forall t\ge0,
	\end{align*}
    which, together with Lemma~\ref{LEMMAMT1-2}, implies that $\nu_{\VVV}$ is an invariant measure of the semigroup $\{\SSS^{T,\VVV}_t\}_{t\ge0}$. Finally, an application of Lemma~\ref{LEMMAMT1-3} shows that $\nu_{\VVV}= h_{\VVV}\mu_{\VVV}$, thus verifying the uniqueness of equilibrium state.

	\section{Proof of Theorem~\ref{MT3}}\label{proofMT3}
	The LD upper and lower bounds with a good rate function $\III$ can be established by the Dawson--G\"artner type argument as in Section A.3 in \cite{JNPS15}. Thus, we omit the proof, and turn to the Donsker--Varadhan entrophy formula \eqref{MR9}. The proof is divided into three steps.
		
	\noindent\textit{Step 1.} Let $I$ denote the rate function given by the Donsker--Varadhan variational formula \eqref{MR4}, and let $\PPPP_t$ denote the transition operator defined in \eqref{Markov1}. We claim that 
	\begin{align}\label{DVEF1}
	    I(\lambda)\le\lim_{t\to 0^+}\frac{1}{t}\sup_{\substack{f\in C_b(H)\\\inf_Hf>0}}\int_H \log \frac{f}{\PPPP_t f}\mathrm{d}\lambda.
	\end{align}
	To see this, we denote by $J(\lambda)$ the right hand-side of the above inequality. For any $f\in \DD(\LLLL)$ satisfying $\inf_H f>0$, by using the continuity of the map $s\mapsto \frac{-\PPPP_s\LLLL f}{\PPPP_s f}$ and the dominated convergence theorem, we have
		\begin{align*}
			\int_H\frac{-\LLLL f}{f}\dd \lambda&=\int_H\lim_{t\to 0^+}\frac{1}{t}\int_0^t\frac{-\PPPP_s\LLLL f}{\PPPP_s f}\dd s\dd \lambda\\&=\lim_{t\to 0^+}\frac{1}{t}\int_H\int_0^t\frac{-\PPPP_s\LLLL f}{\PPPP_s f}\dd s\dd \lambda\\&=\lim_{t\to 0^+}\frac{1}{t}\int_H(-\log \PPPP_t f+\log f)\dd \lambda\le J(\lambda).
		\end{align*}
        Taking supremum in $f$, we obtain \eqref{DVEF1}.
        
		\noindent \textit{Step 2.} Let $\PPPP_{t}^T$ denote the transition operator defined in \eqref{FK3}. Repeating the derivation of \eqref{DVEF1} and the proof of Theorem~\ref{MT1}, we see that the rate function $I_T$ in Theorem~\ref{MT2} satisfies
		\[I_T(\lambda_{[0,T]})\le \lim_{t\to 0^+}\frac{1}{t}\sup_{\substack{\fff\in C_b(X^{T})\\\inf_{X^T}\fff>0}}\int_{X^{T}} \log \frac{\fff}{\PPPP_{t}^T \fff}\mathrm{d}\lambda_{[0,T]},\qquad \forall\lambda_{[0,T]}\in \PP(X^{T}).\]
		Hence, by the Dawson--G\"artner theorem, cf. Theorem A.3. in \cite{JNPS15}, 
		\[\III(\bm\lambda)\le \sup_{T>0}\lim_{t\to 0^+}\frac{1}{t}\sup_{\substack{\fff\in C_b(X^{T})\\\inf_{X^T}\fff>0}}\int_{X^{T}} \log \frac{\fff}{\PPPP_{t}^T \fff}\mathrm{d}\pi^*_{[0,T]}\bmlambda,\qquad\forall \bmlambda\in \PP(X).\]
        In particular, for $\bm\lambda\in \PP_s(X)$, by applying the translation-invariance of $\bmlambda$ and change of test functions, we have 
		\begin{align*}
			\III(\bm\lambda)&\le \sup_{T>0}\lim_{t\to 0^+}\frac{1}{t}\sup_{\fff\in \BB^+(X^{T})}\int_{X^{\R}}\left(\fff-\log\PPPP_{t}^Te^{\fff}\right)(u_{[t-T,t]})\mathrm{d}\bmlambda,
		\end{align*}
		where $\BB^+(X^{T})$ stands for the set of all bounded measurable functions $\fff$ on $X^{T}$ satisfying $\inf_{X^{T}}\fff>0$. Moreover, by applying the conditional expectation,
		\begin{align*}
			\int_{X^{\R}}\fff( u_{[t-T,t]})\mathrm{d}\bmlambda=\int_{X^{\R}}\int_{X^{t}}\fff(u_{[t-T,0]};v_{[0,t]})\pi_{[0,t]}^*\bmlambda(u_{(-\infty,0]},\dd v_{[0,t]})\dd \bmlambda,
		\end{align*}
        while, on the other hand, by utilizing the definition of $\PPPP_t^T$ and the translation-invariance of $\bmlambda$ again, we obtain
		\begin{align*}
			\int_{X^{\R}}&\log\PPPP_{t}^Te^\fff( u_{[t-T,t]})\dd\bmlambda\\&=\int_{X^{\R}}\log\left( \int_{X^{t}} e^{\fff(u_{[2t-T,t]};v_{[0,t]})}\pi^*_{[0,t]}\mathbb{P}_{u(t)}(\dd v_{[0,t]})\right)\dd\bmlambda \\&=\int_{X^{\R}}\log \left(\int_{X^{t}} e^{\fff(u_{[t-T,0]};v_{[0,t]})}\pi^*_{[0,t]}\mathbb{P}_{u(0)}(\dd v_{[0,t]})\right)\dd\bmlambda.
		\end{align*}
		Let us define
		\begin{align*}
			F_{\fff,T,t}(u_{(-\infty,0]})&:=\int_{X^{t}}\fff(u_{[t-T,0]};v_{[0,t]})\pi_{[0,t]}^*\bmlambda(u_{(-\infty,0]},\dd v_{[0,t]})\\&\quad-\log \left(\int_{X^{t}} e^{\fff(u_{[t-T,0]};v_{[0,t]})}\pi^*_{[0,t]}\mathbb{P}_{u(0)}(\dd v_{[0,t]})\right),
		\end{align*}
		so
		\[\III(\bmlambda)\le \sup_{T>0}\lim_{t\to 0^+}\frac{1}{t}\sup_{\fff\in \BB^+(X^{T})}\int_{X^{\R}}F_{\fff,T,t}(u_{(-\infty,0]})\dd \bmlambda.\]
		By the definition \eqref{relative-entropy} of the relative entropy, 
		\[F_{\fff,T,t}(u_{(-\infty,0]})\le \ent(\pi_{[0,t]}^*\bmlambda(u_{(-\infty,0]},\cdot)|\pi^*_{[0,t]}\mathbb{P}_{u(0)}).\]
		This, together with Proposition~\ref{propositionC1}, implies 
		\begin{align*}
			\III(\bmlambda)\le \lim_{t\to 0^+}\frac{1}{t}\int_{X^{\R}}\ent(\pi^*_{[0,t]}\bmlambda(u_{(-\infty,0]},\cdot)|\pi^*_{[0,t]}\mathbb{P}_{u(0)})\dd \bmlambda=\II(\bm\lambda),
		\end{align*}   
		where $\II(\bm\lambda)$ denotes the right hand-side of \eqref{MR9}. Notice that if $\bmlambda\notin \PP_s(X)$, then $\II(\bm\lambda)=+\infty$. Hence, we conclude that
		\begin{align}
		    \label{DVE1}\III(\bmlambda)\le \II(\bm\lambda).\end{align}
		\textit{Step 3.} It remains to prove \eqref{DVE1} in the other direction. By contradiction, we assume that there is a probability measure $\bm\lambda\in \PP(X)$ such that $\III(\bmlambda)< \II(\bm\lambda)$. For $\delta>0$, let $A$ be an open set such that $\bmlambda\in A$ and  
		\[\inf_{\bm\sigma\in \bar{A}}\II(\bm\sigma)\ge (\II(\bm\lambda)-\delta)\wedge \frac{1}{\delta}.\]
		The existence of such an open set $A$ is due to the lower semi-continuity of $\II$. Utilizing this, the LD lower bounds with the rate function $\III$, and Proposition~\ref{proposition4-3}, we have  
		\[\III(\bm\lambda)\ge \inf_{\bm\sigma\in A}\III(\bm\sigma)\ge \inf_{\bm\sigma\in \bar{A}}\II(\bm\sigma)\ge (\II(\bm\lambda)-\delta)\wedge \frac{1}{\delta}.\]
		Sending $\delta\to0^+$, we obtain $\III(\bm\lambda)\ge \II(\bm\lambda)$, which is a contradiction and thus completes the proof.
	\appendix
	\section{A priori estimates}
	The following lemma gathers standard a priori estimates for the solutions of the stochastic Navier--Stokes system.
	\begin{lemma}\label{lemmaA1}
		Let $u(t)$ be the solution of \eqref{I1} issued from $u_0\in H$. Suppose that $h\in U$ and $\BB_1<\infty$. Then, there are positive constants $\gamma_0,K$ depending on $h$ and 
        $\BB_0:=\sum_{j=1}^{\infty}b_j^2$ such that the following estimates hold.
		\begin{enumerate}
			\item \textbf{Energy equality.} For any $t\ge0$, 
			\begin{align}
				\label{A0}
				\E\|u(t)\|^2+2 \E\int_0^t\|\nabla u(s)\|^2\mathrm{d}s=2\E\int_0^t\langle u(s),h\rangle\dd s+\|u_0\|^2+\BB_0 t.
			\end{align}
			
			\item \textbf{Moment estimates for the energy.} For any $\kappa\in(0,\gamma_0)$,
			\begin{align}
				\label{A1}&\mathbb{P}\left\{\sup_{t\ge 0}(\EE(t)-Kt)\ge \|u_0\|^2+\rho\right\}\le e^{-\gamma_0\rho},\qquad \forall \rho\ge0,\\
				\label{A2}&\E e^{\kappa\EE(t)}\lesssim_{\kappa,\BB_0,h}e^{\kappa(Kt+\|u_0\|^2)},
			\end{align}
			where 
			\begin{align*}\EE(t):=\|u(t)\|^2+\int_0^t\|u(s)\|^2_1\mathrm{d}s.
			\end{align*}
			\item \textbf{Moment estimates for the $L^2$-norm.} For any $\kappa\in(0,\gamma_0)$ and $m\ge 1$, 
			\begin{align}
				\label{A3}&\E e^{\kappa\|u(t)\|^2}\le e^{-\kappa t}e^{\kappa\|u_0\|^2}+C_{\kappa,\BB_0,h},\\\label{A4}
				&\E\|u(t)\|^{2m}\le e^{-m\alpha_1 t}\|u_0\|^{2m}+C_{m,\BB_0,h}.
			\end{align}
			\item \textbf{Moment estimates for the $L_t^\infty L_x^2$-norm.} For any $\kappa\in(0,\gamma_0)$ and $m\ge 1$,
			\begin{align}
                \label{A4-3}
				&\E\sup_{s\in[t,t+T]}(\|u(s)\|^2\exp(\kappa\|u(s)\|^2))\lesssim_{\kappa,\BB_0,h,t,T}e^{\kappa\|u_0\|^2},\\\label{A4-1}&\E \exp\left(\kappa\|u(\cdot)\|_{C([t,t+T];H)}^2\right)\le 4 e^{-\kappa t}e^{\kappa\|u_0\|^2}+C_{\kappa,\BB_0,h,T},\\\label{A4-2}
		        &\E\|u(\cdot)\|_{C([t,t+T];H)}^{2m}\le 4e^{-m\alpha_1 t}\|u_0\|^{2m}+C_{m,\BB_0,h,T} .
			\end{align}
			\item \textbf{Doubly-logarithmic moment estimate for the $L^\infty_tH_x^1$-norm.} For any $t>0$,
			\begin{align}
				\label{A5}
				\E\log\left(1+\log\left(1+t\|u(\cdot)\|_{C([t,t+T];U)}^2\right)\right)\lesssim_{\BB_1,h,T} \|u_0\|^2+1+t.
			\end{align}
		\end{enumerate}
	\end{lemma}
	\begin{proof}
		The estimates \eqref{A0}--\eqref{A4} are proved in \cite{KS12} and \cite{Ner19} in the case of periodic boundary conditions, and the proof is similar in the case of Dirichlet boundary conditions.
		
		The proof of \eqref{A4-3}--\eqref{A4-2} is based on \eqref{A3}, \eqref{A4}, and (6.22) in \cite{Ner19} combined with the Markov property. Let us present only the proof of \eqref{A4-2}, since the other two follow similarly. First, let us establish \eqref{A4-2} for $t=0$. Indeed, by using the It\^o formula, we have 
		\begin{align*}
			\dd \|u(t)\|^{2m}&=m\|u(t)\|^{2(m-1)}\left(-2\|\nabla u\|^2\dd t+2\langle u,h\rangle+\BB_0\right)\dd t\\&\quad+2m\|u(t)\|^{2(m-1)}\sum_{j\ge 1}b_j\langle u,e_j\rangle\dd \beta_j\\&\quad+2m(m-1)\|u\|^{2(m-2)}\sum_{j\ge1}b_j^2\langle u,e_j\rangle^2\dd t,
		\end{align*}
		from which, by using the Young inequality, we derive
		\begin{align*}
			\|u(t)\|^{2m}-\|u_0\|^{2m}+m\alpha_1&\int_0^t \|u\|^{2m}\dd s\\&\le C_{m,\BB_0,h}t+2m\int_0^t \|u\|^{2(m-1)}\sum_{j\ge 1}b_j\langle u,e_j\rangle\dd \beta_j.
		\end{align*}
		Therefore,
		\begin{align*}
			\sup_{t\le T}\|u(t)&\|^{2m}+m\alpha_1\int_0^T \|u\|^{2m}\dd s\\&\le 2\|u_0\|^{2m}+C_{m,\BB_0,h}T+4m\sup_{t\le T}\int_0^t \|u\|^{2(m-1)}\sum_{j\ge 1}b_j\langle u,e_j\rangle\dd \beta_j,
		\end{align*}
		which, by using the Burkholder--Gundy--Davis inequality, leads to
		\begin{align}\label{A5-1}
			\E\sup_{t\le T}\|u(t)&\|^{2m}+m\alpha_1\E\int_0^T \|u\|^{2m}\dd s\notag\\&\le 2\|u_0\|^{2m}+C_{m,\BB_0,h}T+4m\E\left(\int_0^T\|u\|^{4(m-1)}\sum_{j\ge 1}b_j^2\langle u,e_j\rangle^2\dd s\right)^{\frac{1}{2}}.
		\end{align}
		Notice that by the Young inequality,
		\begin{align*}
			&\left(\int_0^T\|u\|^{4(m-1)}\sum_{j\ge 1}b_j^2\langle u,e_j\rangle^2\dd s\right)^{\frac{1}{2}}\le \BB_0^{\frac{1}{2}}\left(\int_0^T\|u\|^{4m-2}\dd s\right)^{\frac{1}{2}}\\&\qquad\qquad\qquad\qquad\qquad\le \BB_0^{\frac{1}{2}}\left(\int_0^T\|u\|^{2m-2}\dd s\right)^{\frac{1}{2}}\left(\sup_{t\le T}\|u(t)\|^{2m}\right)^{\frac{1}{2}}\\&\qquad\qquad\qquad\qquad\qquad\le\frac{1}{8m}\sup_{t\le T}\|u(t)\|^{2m}+C_{m,\BB_0}\int_0^T\|u\|^{2m-2}\dd s\\&\qquad\qquad\qquad\qquad\qquad\le\frac{1}{8m}\sup_{t\le T}\|u(t)\|^{2m}+\frac{\alpha_1}{8}\int_0^T\|u\|^{2m}\dd s+C_{m,\BB_0}T.
		\end{align*}
		This, combined with \eqref{A5-1}, implies
		\begin{align}
			\label{A5-2}
			\E\sup_{t\le T}\|u(t)&\|^{2m}+m\alpha_1\E\int_0^T \|u\|^{2m}\dd s\le 4\|u_0\|^{2m}+C_{m,\BB_0,h}T.
		\end{align}
		To prove \eqref{A4-2}, we combine \eqref{A5-2}, \eqref{A4}, and the Markov property to obtain
		\begin{align*}
			\E\|u(\cdot)\|_{C([t,t+T];H)}^{2m}&=\E\left(\E\left(\|u(\cdot)\|_{C([t,t+T];H)}^{2m}\Big|\mathscr{F}_t\right)\right)\\&=\E\left(\E_{u_t}\left(\|u(\cdot)\|_{C([0,T];H)}^{2m}\right)\right)\\&\le \E\left(4\|u(t)\|^{2m}+ C_{m,\BB_0,h,T}\right)\\&\le 4e^{-m\alpha_1t}\|u_0\|^{2m}+ C_{m,\BB_0,h,T}
		\end{align*}
		as desired. 
		
		The estimate \eqref{A5} in the case of multiplicative noise is obtained in \cite{KV14}. In the additive case, the proof is simpler and we give it below for completeness. Let us write
		\begin{align}\label{A6}
			u(t)=v(t)+z(t),
		\end{align}
		where $v(t)$ is the solution of the following NS-type equation:
		\begin{align}\label{A7}
			\begin{cases}
				\partial_t v+Lv+B(v+z)=h,\\
				v(0)=u_0,
			\end{cases}
		\end{align}
		and $z(t)$ solves the following stochastic Stokes system:
		\begin{align}\label{A8}
			\begin{cases}
				\partial_t z+Lz= \eta,\\
				z(0)=0.
			\end{cases}
		\end{align}
		Applying the inner product in $U$ to the both sides of \eqref{A7}, we obtain
		\begin{align*}
			\frac{1}{2}\frac{\dd}{\dd t}\|L^{\frac{1}{2}}v(t)\|^2+\|Lv\|^2=\langle h,L v\rangle-\langle B(v+z),L v\rangle.
		\end{align*}
		Utilizing the boundedness of the Leray projection $\Pi$ combined with the Ladyzhenskaya inequality, we derive
		\begin{align*}
			-\langle B(v+z),L v\rangle&\lesssim \|L v\|\|v+z\|_{L^4}\|\nabla(v+z)\|_{L^4}\notag\\&\lesssim\|v\|_2\|v+z\|^{\frac{1}{2}}\|v+z\|_1\|v+z\|_2^{\frac{1}{2}}\notag\\&\lesssim \|v\|_2^{\frac{3}{2}}\|v+z\|^{\frac{1}{2}}\|v+z\|_1+\|v\|_2\|v+z\|^{\frac{1}{2}}\|v+z\|_1\|z\|_2^{\frac{1}{2}}\notag\\&\le \frac{1}{8} \|Lv\|^2+C\|v+z\|_1^4\|v+z\|^2+C\|z\|^2_2.
		\end{align*}
		Therefore, we have 
		\begin{align}\label{A8-1}
			\frac{\dd}{\dd t}\|v(t)\|_1^2+\|v\|_{2}^2\lesssim\|h\|^2+\|v+z\|_1^4\|v+z\|^2+\|z\|^2_2.
		\end{align}
		Let us establish estimates for $\|v\|$. By applying the inner product in $H$ on both sides of \eqref{A7}, the cancellation property of the nonlinearity $B$, and the Ladyzhenskaya inequality, we have 
		\[\frac{\dd}{\dd s}\|v(s)\|^2 +\|\nabla v\|^2 \lesssim \|h\|^2+\|z\|_1^2\|v\|^2+\|z\|^2\|z\|_1^2,\]
		which gives 
		\begin{align}
			\label{A8-2}
			\|v(t)\|^2+\int_0^t\|v\|_1^2\dd s\lesssim_h (\|u_0\|^2+t+1)\exp\left(C\|z_{[0,t]}\|^2_{\XX^0}\right),
		\end{align}
		where for $j\in\N$,
		\begin{align}
			\label{A8-3}
			\|z_{[0,t]}\|^2_{\XX^j}:=\sup_{s\le t}\|z(s)\|^2_j+\int_0^t\|z\|^2_{j+1}\dd s.
		\end{align}
		To close the estimate for $v$, let us combine \eqref{A8-1} with \eqref{A8-2}, and derive
		\begin{align*}
			\frac{\dd}{\dd t}(t\|v(t)\|_1^2)+t\|v\|_{2}^2\lesssim t\|h\|^2+t(\|v+z\|_1^4\|v+z\|^2+\|z\|^2_2)+\|v(t)\|_1^2,
		\end{align*}
		which, together with the Gronwall inequality and \eqref{A8-2}, leads to 
		\begin{align}\label{A8-4}
			t\|v(t)\|_1^2+\int_0^t s\|v\|_{2}^2\dd s\lesssim_h \exp\left(C_h \FFFF_1(t,u_0,z_{[0,t]})\right),
		\end{align}
		where 
		\begin{align}
			\label{A8-5}
			\FFFF_1(t,u_0,z_{[0,t]}):=(\|u_0\|^4+t^2+1)\exp(C\|z_{[0,t]}\|^2_{\XX^1}).
		\end{align}
		By first taking double logarithm and then expectation, we arrive at
		\begin{align}
			\label{A8-6}
			\E\log&\left(1+\log\left(1+t\|v(\cdot)\|^2_{C([t,t+T];U)}\right)\right)\notag\\&\qquad\qquad\qquad\qquad\qquad\lesssim_{h,T}1+t+\E\|u_0\|^2+\E\|z_{[0,t+T]}\|^2_{\XX^1}.
		\end{align}
		It remains to estimate the solutions of the stochastic Stokes system \eqref{A8}. By the It\^o formula, 
		\begin{align*}
			\dd \|L^\frac{1}{2}z(t)\|^2+2\|Lz\|^2\dd t= 2\sum_{j\ge 1} b_j\alpha_j\langle z,e_j\rangle \dd \beta_j+\BB_1\dd t.
		\end{align*}
		An application of the Burkholder--Davis--Gundy inequality yields
		\begin{align}\label{A8-8}
			\E\sup_{s\le t}\|z(s)\|^2_1+\E\int_0^t\|z\|^2_2\dd s\lesssim_{\BB_1}t+1.
		\end{align}
		Combining this with \eqref{A8-6}, we have 
		\begin{align}
			\label{A8-7}
			\E\log\left(1+\log\left(1+t\|v(\cdot)\|_{C([t,t+T];U)}^2\right)\right)\lesssim_{\BB_1,h,T} \E\|u_0\|^2+1+t.
		\end{align}
		Notice that for $g(x):=\log(1+\log(1+x))$, 
		\begin{align}\label{A8-9}
			g(xy)+g(x+y)\le 2(g(x)+g(y))
		\end{align}
		and $g(x)\le x$ for any $x,y\ge 0$. This, together with \eqref{A6}, \eqref{A8-8}, and \eqref{A8-7}, implies \eqref{A5} as desired.
	\end{proof}
	\begin{corollary}\label{corollaryA1}
		Under the conditions of Lemma~\ref{lemmaA1}, for any $t>0$,
		\begin{align}\label{A10}\E \log\left(1+\log\left(1+t\|u(\cdot)\|^2_{C^{\frac{1}{4}}([t,t+T];U^*)}\right)\right)\lesssim_{\BB_1,h,T}1+t+\frac{1}{t}+\|u_0\|^2.
		\end{align}
	\end{corollary} 
	\begin{proof}Let us note that
		\[u(t)=u_0+\int_0^t(h-B(u)-Lu)\dd s+W(t)\]
		with
		\begin{align}\label{A12}
			W(t):=\sum_{j\ge 1}b_je_j\beta_j(t).     
		\end{align} 
		Then,
		\begin{align*}
			&\|u(\cdot)\|_{C^{\frac{1}{4}}([t,t+T];U^*)}\\&\qquad\ \le \|u_0\|+\left\|\int_0^{\cdot}(h-B(u)-Lu)\dd s\right\|_{C^{1}([t,t+T];U^*)}+\|W(\cdot)\|_{C^{\frac{1}{4}}([t,t+T];H)},
		\end{align*}
		where 
		\begin{align*}
			&\left\|\int_0^{\cdot}(h-B(u)-Lu)\dd s\right\|_{C^{1}([t,t+T];U^*)}\\&\qquad\ \le
			\left\|u(\cdot)-u_0-W(\cdot)\right\|_{C([t,t+T];H)}+\|h-B(u)-Lu\|_{C([t,t+T];U^*)}\\&\qquad\ \lesssim \|u_0\|+\|h\|+1+\|u(\cdot)\|_{C([t,t+T];U)}^2+\|W(\cdot)\|_{C^{\frac{1}{4}}([t,t+T];H)}.
		\end{align*}
		Hence, 
		\begin{align*}
			&\|u(\cdot)\|_{C^{\frac{1}{4}}([t,t+T];U^*)}\\&\qquad\ \lesssim \|u_0\|+\|h\|+1+\|u(\cdot)\|_{C([t,t+T];U)}^2+\|W(\cdot)\|_{C^{\frac{1}{4}}([t,t+T];H)}.
		\end{align*}
		Combining this with \eqref{A8-9}, we derive
		\begin{align}\label{A11}
			g\left(t\|u(\cdot)\|^2_{C^{\frac{1}{4}}([t,t+T];U^*)}\right)&\lesssim_h g\left(t\|u(\cdot)\|^2_{C([t,t+T];U)}\right)\notag\\&\qquad+t+\frac{1}{t}+\|u_0\|^2+\|W(\cdot)\|^2_{C^{\frac{1}{4}}([t,t+T];H)}.
		\end{align} 
            It remains to establish moment estimates on the H\"older norm of the Wiener process $W$. Using the embedding 
		\[W^{\frac{3}{8},8}(0,T;H)\subset C^{\frac{1}{4}}([0,T];H),\]
		and Lemma 2.1 in \cite{FG1995}, we obtain
		\begin{align*}
			\mathbb{E}\|W(\cdot)\|^2_{C^{\frac{1}{4}}([t,t+T];H)}\lesssim 1+\mathbb{E}\|W(\cdot)\|^8_{W^{\frac{3}{8},8}(0,t+T;H)}\lesssim_{\BB_0,T} t+1.
		\end{align*}
		This, together with \eqref{A5} and \eqref{A11}, implies \eqref{A10}.
	\end{proof}
	For $t,T>0$, we define the solution operator $\RRRR_t^{t+T}$ of \eqref{A7} by 
	\begin{align}\label{A13-1}
		\RRRR_t^{t+T}:\ &H\times \XX^1(0,t+T)\to C([0,T];U)\mcap C^{1}([0,T];U^*) \notag\\& (u_0,z_{[0,t+T]})\mapsto v_{[t,t+T]},
	\end{align}
	where for $j\in\N$ and $t>0$, 
	\begin{align}
		\label{A13}
		\XX^j(0,t):=C_0([0,t];U^j)\mcap L^2(0,t;U^{j+1})
	\end{align}
	with the norm $\|\cdot\|_{\XX^j}$ defined in \eqref{A8-3} and $C_0([0,t];U^j)$ denoting the space of all continuous functions valued in $U^j$ and vanishing at $t=0$. 
	\begin{proposition}\label{propositionlip}
		Suppose that $h\in U$ and $\BB_1<\infty$. Then, for any $t,T>0$, $\RRRR_t^{t+T}$ is locally Lipschitz.
	\end{proposition}
	\begin{proof}Let $u_0^i\in H$, $z^i\in \XX^1(0,t+T)$, and let $v^i$ denote the solution of \eqref{A7} given the initial data $u_0^i$ and the trajectory $z^i$. Let us write 
		\[v:=v^1-v^2,\qquad z:=z^1-z^2.\]
		Then, the equation for $v$ is given by
		\begin{align}
			\label{L13}\partial_t v+B(v+z,v^1+z^1)+B(v^2+z^2,v+z)+Lv=0.
		\end{align}
		Applying the inner product in $H$ and the cancellation property of the convection term $B$, we have 
		\begin{align*}
			\frac{1}{2}\frac{\dd }{\dd s}\|v(s)\|^2+\|\nabla v\|^2=-\langle B(v+z,v^1+z^1),v\rangle-\langle B(v^2+z^2,z),v\rangle.
		\end{align*}
		Using the embedding $H^{\frac{1}{2}}\subset L^4$, the Ladyzhenskaya inequality, and the Cauchy--Schwartz inequality, we derive
		\begin{align*}
			-\langle B(v+z,v^1+z^1),v\rangle&\lesssim \|v+z\|_{L^4}\|v\|_{L^4}\|v^1+z^1\|_1 \\&\lesssim \|v\|\|v\|_{1}\|v^1+z^1\|_1 +\|z\|_{\frac{1}{2}}\|v\|_{\frac{1}{2}}\|v^1+z^1\|_1\\&\le \frac{1}{8}\|\nabla v\|^2+C\|v\|^2\|v^1+z^1\|_1^2+ C\|z\|_1^2\|v^1+z^1\|_1^2,
		\end{align*}
		and 
		\begin{align*}
			-\langle B(v^2+z^2,z),v\rangle&\le \|v\|_1\|v^2+z^2\|_{\frac{1}{2}}\|z\|_{\frac{1}{2}}\\&\le \frac{1}{8}\|\nabla v\|^2+C\|z\|_1^2\|v^2+z^2\|_1^2.
		\end{align*}
		This implies
		\begin{align*}
			\frac{\dd }{\dd s}\|v(s)\|^2+\|v\|_1^2\lesssim \|v\|^2\|v^1+z^1\|_1^2+ \|z\|_1^2\|v^1+z^1\|^2_1+\|z\|_1^2\|v^2+z^2\|_1^2.
		\end{align*}
		An application of the Gronwall inequality gives
		\begin{align}
			\label{L15}
			\|v(s)\|^2+\int_0^s \|v\|_1^2\dd r&\lesssim \left(\|v_0\|^2+\|z_{[0,s]}\|^2_{\XX^1}\right)\notag\\&\ \times \exp\left(C\sum_{i=1}^2\left(\|v^i_{[0,s]}\|^2_{\XX^0}+\|z^i_{[0,s]}\|^2_{\XX^0}\right)\right).
		\end{align}
		Next, we move to the $H^1$-estimate. Taking the inner product in $U$, we get 
		\begin{align*}
			\frac{1}{2}\frac{\dd }{\dd t}\|L^{\frac{1}{2}} v(t)\|^2+\|L v\|^2=-\langle B(v+z,v^1+z^1),Lv\rangle-\langle B(v^2+z^2,v+z),Lv\rangle,
		\end{align*}
		where, by using the interpolation inequalities
		\begin{align}
			\label{L14}
			\|f\|_1\lesssim \|f\|^\frac{1}{2}\|f\|^\frac{1}{2}_2,\qquad \|f\|_{L^\infty}\lesssim \|f\|^\frac{1}{2}\|f\|^\frac{1}{2}_2,
		\end{align}
		and the Cauchy--Schwartz inequality, we derive
		\begin{align*}
			\langle B(v+z,v^1+z^1),Lv\rangle &\lesssim \|v+z\|^\frac{1}{2}\|v+z\|_2^\frac{1}{2}\|v^1+z^1\|^\frac{1}{2}\|v^1+z^1\|^\frac{1}{2}_2\|Lv\|\\&\lesssim \frac{1}{8}\|Lv\|^2+C\|v+z\|^2\|v^1+z^1\|^2\|v^1+z^1\|^2_2+C\|z\|_2^2,
		\end{align*}
		and 
		\begin{align*}
			\langle B(v^2+z^2,v+z),Lv\rangle&\lesssim\|v+z\|^\frac{1}{2}\|v+z\|_2^\frac{1}{2}\|v^2+z^2\|^\frac{1}{2}\|v^2+z^2\|^\frac{1}{2}_2\|Lv\|\\&\lesssim \frac{1}{8}\|Lv\|^2+C\|v+z\|^2\|v^2+z^2\|^2\|v^2+z^2\|^2_2+C\|z\|_2^2.
		\end{align*}
		Therefore, 
		\begin{align*}
			\frac{\dd }{\dd s}(s\|v(s)\|_1^2)&+s\|v\|_2^2\\&\lesssim \|v\|_1^2+\sum_{i=1}^2s\left(\|v^i+z^i\|^2\|v^i+z^i\|^2_2\right)\|v+z\|^2+s\|z\|_2^2.
		\end{align*}
		Combining this and \eqref{L15}, we obtain
		\begin{align}
			\label{L16}
			s\|v(s)\|^2_1+\int_0^s r&\|v\|_2^2\dd r\notag\\\notag&\lesssim (s+1)\left(\|v_0\|^2+\|z_{[0,s]}\|^2_{\XX^1}\right)\left(1+\sum_{i=1}^2\int_0^sr\|v^i+z^i\|^2_2\dd r\right)\\&\quad\times\exp\left(C\sum_{i=1}^2\left(\|v^i_{[0,s]}\|^2_{\XX^0}+\|z^i_{[0,s]}\|^2_{\XX^0}\right)\right).
		\end{align}
		From \eqref{A8-2} and \eqref{A8-4}, we derive
		\[\left(\|v^i_{[0,s]}\|^2_{\XX^0}+\int_0^sr\|v^i\|^2_2\dd r\right)\lesssim_h \exp(C_h\FFFF_1(s,u^i_0,z^i_{[0,s]})),\qquad i=1,2,\]
        where $\FFFF_1$ is given in \eqref{A8-5}. Plugging this estimate into \eqref{L16}, we obtain
        \begin{align*}\|v(\cdot)\|^2_{C([t,t+T];U)}\lesssim_{h,t,T,M}\left(\|u^1_0-u_0^2\|^2+\|z^1_{[0,t+T]}-z^2_{[0,t+T]}\|^2_{\XX^1}\right)
		\end{align*}
        for
		\[\|u_0^i\|\le M,\qquad \|z^i_{[0,t+T]}\|_{\XX^1}\le M,\qquad i=1,2.\]
		This, together with the equation \eqref{L13} and \eqref{A8-4}, implies 
		\begin{align*}\|\partial_tv(\cdot)\|^2_{C([t,t+T];U^*)}\lesssim_{h,t,T,M}\left(\|u^1_0-u_0^2\|^2+\|z^1_{[0,t+T]}-z^2_{[0,t+T]}\|^2_{\XX^1}\right)
		\end{align*}
		and thus completes the proof.
	\end{proof}
	\section{Markov semigroups on weighted phase spaces}\label{AppenB}
	Let $\{\PPPP_t\}_{t\ge0}$ denote the Markov semigroup defined in \eqref{Markov1}. For $V\in C_b(H)$, we denote by $\{\PPPP^V_t\}_{t\ge0}$ the Feynman--Kac semigroup defined in \eqref{MR6}. For $\kappa>0$, we introduce the weighted space
	\begin{align}\label{B0}C_{\www,\kappa}(H):=\left\{f\in C(H)\Big| \|f\|_{C_{\www,\kappa}}:=\sup_{u\in H}\frac{|f(u)|}{e^{\kappa \|u\|^2}}<\infty\right\}.\end{align}
	The goal of this section is twofold. First, we show that $\{\PPPP^V_t\}_{t\ge0}$ forms a semigroup of bounded linear operators in the weighted space $C_{\www,\kappa}(H)$ for appropriately chosen $\kappa>0$ and any $V\in C_b(H)$. Second, we introduce the resolvent operator $\RR_\alpha$ of the Markov semigroup $\{\PPPP_t\}_{t\ge0}$ and establish some approximation properties for $\RR_\alpha$. These properties are essential for the proof of Theorem~\ref{MT1}.
	\begin{proposition}\label{propositionB1}
		Let $V\in C_b(H)$ and $\kappa:=\frac{\gamma_0}{4}$, where $\gamma_0$ is the constant in Lemma~\ref{lemmaA1}. Then, $\{\PPPP_t^V\}_{t\ge 0}$ forms a semigroup of bounded linear operators in $C_{\www,\kappa}(H)$ satisfying 
		\begin{enumerate}
			\item for any $T>0$, $\sup_{0\le t\le T}\|\PPPP^V_t\|_{\LL(C_{\www,\kappa}(H))}\le \exp(T\|V\|_{\infty})C_{\kappa,\BB_0,h}<\infty$;
			\item for any $u\in H$ and $f\in C_{\www,\kappa}(H)$, $t\mapsto \PPPP^V_t f(u)$ is continuous.
		\end{enumerate}
	\end{proposition}
	\begin{proof}
		\textit{Step 1.} Let us denote by $u(t;u_0)$ the solution of \eqref{I1} issued from $u_0\in H$. For any $V\in C_b(H)$, $f\in C_{\www,\kappa}(H)$, and $R,T>0$, we claim that the family of random variables 
        \[\FF_{V,f,R,T}:=\left\{\exp\left(\int_0^t V(u(s;u_0))\dd s\right) f(u(t;u_0))\Big|0\le t\le T, u_0\in B_H(R)\right\}\]
        is uniformly integrable. Indeed, by the exponential moment estimate \eqref{A3},
		\begin{align*}
			\sup_{0\le t\le T, u_0\in B_H(R)}\E&\left(\exp\left(2\int_0^t V(u(s;u_0))\dd s\right)|f(u(t;u_0))|^2\right)\\&\le \exp(2T\|V\|_{\infty})\|f\|_{C_{\www,\kappa}}^2\sup_{t\ge0, u_0\in B_H(R)}\mathbb{E} e^{2\kappa\|u(t;u_0)\|^2}\\&\le \exp(2T\|V\|_{\infty})\|f\|_{C_{\www,\kappa}}^2(e^{2\kappa R^2}+C_{\kappa,\BB_0,h})<\infty,
		\end{align*}
		which implies the uniform integrability of $\FF_{V,f,R,T}$ as claimed.
		
		\noindent\textit{Step 2.} Let us show that $\PPPP^V_t$ is a bounded linear operator in $C_{\www,\kappa}(H)$ for any $t\ge0$. Indeed, by utilizing the continuity of the solution $u(\cdot,u_0)$ with respect to the initial data $u_0$, see Proposition 2.4.7 in \cite{KS12}, for any $f\in C_{\www,\kappa}(H)$ and $V\in C_b(H)$, we have
		\[\lim_{u\to u_0} \exp\left(\int_0^t V(u(s;u))\dd s\right) f(u(t;u))=\exp\left(\int_0^t V(u(s;u_0))\dd s\right) f(u(t;u_0))\]
		$\mathbb{P}$-almost surely. Therefore, by utilizing the uniform integrability of the family $\FF_{V,f,R,T}$ combined with the Vitali convergence theorem, cf. Theorem 4.5.4. in \cite{Bog2007}, we have the continuity of $u\mapsto\PPPP^V_t f(u)$. Moreover, by using the exponential moment estimate \eqref{A3} again,
		\begin{align*}
			\|\PPPP^V_t f\|_{C_{\www,\kappa}}&=\sup_{u\in H}\frac{|\PPPP^V_tf(u)|}{e^{\kappa\|u\|^2}}\le \exp(t\|V\|_{\infty})\sup_{u\in H}\left(\frac{\mathbb{E}|f(u(t;u))|}{e^{\kappa\|u\|^2}}\right)\\&\le  \exp(t\|V\|_{\infty})\|f\|_{C_{\www,\kappa}}\sup_{u\in H}\frac{\mathbb{E}e^{\kappa \|u(t;u)\|^2}}{e^{\kappa\|u\|^2}}\\&\le \exp(t\|V\|_{\infty})C_{\kappa,\BB_0,h}\|f\|_{C_{\www,\kappa}}.
		\end{align*}
		Hence, $\PPPP^V_t$ is a bounded linear operator in $C_{\www,\kappa}(H)$ for any $t\ge0$, and the first statement holds. The continuity of $t\mapsto \PPPP^V_tf(u)$ for fixed $u\in H$ follows similarly from the uniform integrability of $\FF_{V,f,R,T}$ and the Vitali convergence theorem, and we choose to omit the details.
	\end{proof}
	In what follows, we shall always take $\kappa$ as in Proposition~\ref{propositionB1}. For $\alpha>0$, let us introduce the resolvent operator 
	\begin{align}\label{B1}
		\RR_{\alpha}f(u):=\int_0^{+\infty}e^{-\alpha t}\PPPP_tf(u)\dd t,\qquad f\in C_{\www,\kappa}(H),
	\end{align}
	which is well-defined by Proposition~\ref{propositionB1}.
	
	\begin{proposition}
		\label{propositionB2}
		The following properties hold for the resolvent operator $\RR_{\alpha}$.
		\begin{enumerate}
			\item For each $\alpha>0$, $\RR_{\alpha}$ is a bounded linear operator in $C_{\www,\kappa}(H)$ satisfying
			\[\alpha\|\RR_{\alpha} f\|_{C_{\www,\kappa}}\lesssim_{\kappa,\BB_0,h}\|f\|_{C_{\www,\kappa}}.\]
			\item For any $f\in C_{\www,\kappa}(H)$ and $u\in H$, 
			\[\lim_{\alpha\to+\infty}\alpha\RR_{\alpha}f(u)=f(u).\]
			\item For any $f\in C_{\www,\kappa}(H) $, $u\in H$ and $ \alpha,t>0$, 
			\begin{align}
				\PPPP_t\RR_{\alpha}f(u)=\int_0^{+\infty}e^{-\alpha s}\PPPP_{t+s}f(u)\dd s= \RR_{\alpha}\PPPP_tf(u),\label{B2}\\
				\PPPP_t\RR_{\alpha}f(u)-\RR_{\alpha}f(u)=\int_0^t\PPPP_s(\alpha\RR_{\alpha}f-f)(u)\dd s.\label{B3}
			\end{align}
			\item If a sequence $\{f_n\}_{n=1}^{\infty}\subset C_{\www,\kappa}(H)$ satisfies
			\begin{align}\label{B3_1}
			    \sup_{n\ge 1}\|f_n\|_{C_{\www,\kappa}}<\infty
			\end{align}
			and 
			\begin{align}\label{B3_2}\lim_{n\to+\infty}f_n(u)=f(u),\qquad \forall u\in H,
			\end{align}
			with $f\in C_{\www,\kappa}(H)$, then for any $\alpha>0$, 
			\[\lim_{n\to+\infty}\RR_{\alpha}f_n(u)=\RR_{\alpha}f(u),\qquad \forall u\in H.\]
		\end{enumerate}
	\end{proposition}
	\begin{proof}Let us prove the first statement. Note that 
		\[e^{-\alpha t}|\PPPP_t f(u)|\le e^{-\alpha t}e^{\kappa\|u\|^2}\| f\|_{C_{\www,\kappa}}\sup_{t\ge0}\|\PPPP_t\|_{\LL(C_{\www,\kappa}(H))}.\]
		By  Proposition~\ref{propositionB1} and the dominated convergence theorem, we have the continuity of $\RR_\alpha f$:
		\[\lim_{u\to u_0}\RR_{\alpha} f(u)=\RR_{\alpha}f(u_0).\]
        Moreover, we have
		\begin{align*}
			|\RR_{\alpha}f(u)|&\le \int_0^{+\infty} e^{-\alpha t}|\PPPP_tf(u)|\dd t\\& \le \left(\int_0^{+\infty} e^{-\alpha t}\dd t\right)e^{\kappa\|u\|^2}\| f\|_{C_{\www,\kappa}}\sup_{t\ge0}\|\PPPP_t\|_{\LL(C_{\www,\kappa}(H))}\\& \lesssim_{\kappa,\BB_0,h} \alpha^{-1}e^{\kappa\|u\|^2}\| f\|_{C_{\www,\kappa}}.
		\end{align*}
		Thus, the first statement holds. 
  
        The second statement follows from a change of variables, Proposition~\ref{propositionB1}, and the dominated convergence theorem: 
        \begin{align*}
            \lim_{\alpha\to+\infty}\alpha\RR_{\alpha}f(u)=\lim_{\alpha\to+\infty}\int_0^{+\infty}e^{-t}\PPPP_{t/\alpha}f(u)\dd t=f(u).
        \end{align*}
        
        We turn to the proof of the third statement. By the Fubini theorem, we derive
		\begin{align*}
			\PPPP_t \RR_{\alpha} f(u)&=\mathbb{E}_u\left(\int_0^{+\infty}e^{-\alpha s}\PPPP_sf(u_t)\dd s\right)\\&=\int_0^{+\infty}e^{-\alpha s}\mathbb{E}_u\PPPP_sf(u_t)\dd s=\int_0^{+\infty}e^{-\alpha s}\PPPP_{t+s}f(u)\dd s=\RR_{\alpha}\PPPP_tf(u).
		\end{align*}
		To prove \eqref{B3}, we apply the Fubini theorem combined with a change of variables:
		\begin{align*}
			\PPPP_t \RR_{\alpha} f(u)=\int_0^{+\infty}e^{-\alpha s}\PPPP_{t+s}f(u)\dd s=e^{\alpha t}\int_t^{+\infty}e^{-\alpha s}\PPPP_{s}f(u)\dd s.
		\end{align*}
		This, together with the Newton--Leibniz formula, gives
		\begin{align*}
			\PPPP_t \RR_{\alpha} f(u)-\RR_{\alpha} f(u)&=\int_0^t\frac{\dd }{\dd s}\left(e^{\alpha s}\int_s^{+\infty}e^{-\alpha r}\PPPP_{r}f(u)\dd r\right)\dd s\\
            &=\int_0^t \left(\alpha \PPPP_s\RR_{\alpha} f(u)-\PPPP_s f(u)\right)\dd s\\&=\int_0^t \PPPP_s\left(\alpha \RR_{\alpha} f-f\right)(u)\dd s
		\end{align*}
		as desired. 
		
		Finally, we prove the last statement. For each fixed $u\in H$ and $t\ge0$, by using \eqref{A3}, we have
		\[\sup_{n\ge 1}\E|f_n(u(t;u))|^2\le \left(\sup_{n\ge 1}\|f_n\|^2_{C_{\www,\kappa}}\right)\left(e^{2\kappa\|u\|^2}+C_{\kappa,\BB_0,h}\right)<\infty,\]
		which implies the uniform integrability of $\{f_n(u(t;u))\}_{n\ge 1}$. Notice that 
		\[\lim_{n\to+\infty}f_n(u(t;u))=f(u(t;u))\]
		$\mathbb{P}$-almost surely, so by the Vitali convergence theorem,
		\[\lim_{n\to+\infty}\PPPP_tf_n(u)=\PPPP_tf (u)\]
		for any $u\in H$ and $t\ge 0$. Moreover, since
		\begin{align*}
			e^{-\alpha t}|\PPPP_t f_n(u)|&\le e^{-\alpha t}e^{\kappa\|u\|^2} \sup_{n\ge 1}\|\PPPP_t f_n\|_{C_{\www,\kappa}}\\&\lesssim_{\kappa,\BB_0,h}e^{-\alpha t}e^{\kappa\|u\|^2} \sup_{n\ge 1}\|f_n\|_{C_{\www,\kappa}},
		\end{align*}
		by the dominated convergence theorem, we conclude 
		\begin{align*}
			\lim_{n\to+\infty}\RR_{\alpha} f_n(u)&=\lim_{n\to+\infty}\int_0^{+\infty} e^{-\alpha t}\PPPP_t f_n(u)\dd t\\&=\int_0^{+\infty} e^{-\alpha t}\PPPP_t f(u)\dd t=\RR_{\alpha} f(u)
		\end{align*}
		for any $\alpha>0$ and $u\in H$. 
	\end{proof}
	Let
	\begin{align}\label{B4}
		\mathcal{D}_{\www,\kappa}(\LLLL):=\left\{f\in C_{\www,\kappa}(H)\Big|\exists g\in C_{\www,\kappa}(H),\ \PPPP_tf-f=\int_0^t\PPPP_sg\dd s,\ \forall t\ge 0\right\},
	\end{align}
	and define the extended generator $\LLLL f:=g$ for $f\in \DD_{\www,\kappa}(\LLLL)$.
	\begin{proposition}\label{propositionB3}
		The following properties hold.
		\begin{enumerate}
			\item For any $f\in C_{\www,\kappa}(H)$ and $\alpha>0$, we have $\RR_{\alpha}f \in\DD_{\www,\kappa}(\LLLL)$ and
			\begin{align}\label{B5}
				\LLLL\RR_{\alpha}f=\alpha\RR_{\alpha}f-f.
			\end{align}
			\item For any $f\in \DD_{\www,\kappa}(\LLLL)$ and $\alpha>0$, 
			\begin{align}\label{B6}
				\LLLL\RR_{\alpha}f=\RR_{\alpha}\LLLL f.
			\end{align}
			\item If a sequence $\{f_n\}_{n=1}^{\infty}\subset C_{\www,\kappa}(H)$ converges to $f\in C_{\www,\kappa}(H)$ in the sense of \eqref{B3_1} and \eqref{B3_2}, then
			\[\lim_{n\to+\infty}\LLLL\RR_{\alpha}f_n(u)=\LLLL\RR_{\alpha}f(u)\]
			for any $\alpha>0$ and $u\in H$.
		\end{enumerate} 
	\end{proposition}
	\begin{proof}
		The first statement is valid because of \eqref{B3}, while the third follows from the assertion \textit{4} in Proposition~\ref{propositionB2} and \eqref{B5}. As for the second statement, we use \eqref{B2}, the definitions \eqref{B1} and \eqref{B4}, and the Fubini theorem to get
		\begin{align*}
			\PPPP_t\RR_{\alpha}f(u)-\RR_{\alpha}f(u)&=\RR_{\alpha}(\PPPP_tf-f)(u)\\&=\int_0^{+\infty}e^{-\alpha r}\PPPP_r\left(\int_0^t\PPPP_s\LLLL f\dd s\right)(u)\dd r\\&=\int_0^{+\infty}e^{-\alpha r}\left(\int_0^t\PPPP_{s+r}\LLLL f(u)\dd s\right)\dd r\\&=\int_0^t\int_0^{+\infty}e^{-\alpha r}\PPPP_{s+r}\LLLL f(u)\dd r\dd s\\&=\int_0^t\RR_{\alpha}\PPPP_s \LLLL f(u)\dd s=\int_0^t\PPPP_s \RR_{\alpha}\LLLL f(u)\dd s.
		\end{align*}
		This implies \eqref{B6}.
	\end{proof}
    The following approximation result plays a crucial role in the proof of the Donsker--Varadhan variational formula \eqref{MR4}. 
    \begin{corollary}\label{corollaryB1}
        Let $f\in\DD_{\www,\kappa}(\LLLL)$ satisfy 
	\begin{align}\label{B7}
		f>0,\qquad \frac{\LLLL f}{f}\in C_b(H).
	\end{align}
    Then, we have 
    \begin{align}\label{B8}
		\int_H\frac{-\LLLL f}{f}\dd \lambda=\lim_{\alpha\to+\infty}\int_H\frac{-\LLLL\RR_{\alpha}f}{\RR_\alpha f}\dd \lambda,
	\end{align}
    and
    \begin{align}\label{B9}
		\int_H-\frac{\LLLL\RR_{\alpha}f}{\RR_{\alpha }f}\dd \lambda\le \liminf_{N\to+\infty}\int_H-\frac{\LLLL\RR_{\alpha}f_N}{\RR_{\alpha }f_N}\dd \lambda
	\end{align}
    for any $\alpha>0$, where $f_N:=(f\wedge N)\vee \frac{1}{N}$.
    \end{corollary}
    \begin{proof}
        Let $f\in\DD_{\www,\kappa}(\LLLL)$ satisfy \eqref{B7}. To prove \eqref{B8}, we apply Propositions~\ref{propositionB2} and \ref{propositionB3} to get
        \[\lim_{\alpha\to+\infty}\frac{\LLLL\RR_{\alpha}f(u)}{\RR_{\alpha }f(u)}=\lim_{\alpha\to+\infty}\frac{\alpha\RR_{\alpha}\LLLL f(u)}{\alpha\RR_{\alpha }f(u)}=\frac{\LLLL f(u)}{f(u)}\]
	for any $u\in H$. Moreover, 
	\begin{align*}
		|\RR_{\alpha}\LLLL f(u)|&\le \int_0^{+\infty}e^{-\alpha t}|\PPPP_t \LLLL f(u)|\dd t\\&\le \left\|\frac{\LLLL f}{f}\right\|_{\infty}\int_0^{+\infty}e^{-\alpha t}\PPPP_tf(u)\dd t=\left\|\frac{\LLLL f}{f}\right\|_{\infty}\RR_{\alpha}f(u),
	\end{align*}
	which implies 
	\[\left\|\frac{-\LLLL\RR_{\alpha}f}{\RR_\alpha f}\right\|_{\infty}\le \left\|\frac{\LLLL f}{f}\right\|_{\infty}\]
	for any $\alpha>0$. Therefore, an application of the dominated convergence theorem implies \eqref{B8}. 
    
    We turn to the proof of \eqref{B9}. To this end, let us notice that $f_N$ converges to $f$ pointwisely on $H$ and
	\[\sup_N\|f_N\|_{C_{\www,\kappa}}\le 1\vee\|f\|_{C_{\www,\kappa}}<\infty.\]
	By Propositions \ref{propositionB2} and \ref{propositionB3}, this implies
	\[\lim_{N\to+\infty}-\frac{\LLLL\RR_{\alpha}f_N(u)}{\RR_{\alpha }f_N(u)}=-\frac{\LLLL\RR_{\alpha}f(u)}{\RR_{\alpha }f(u)}\]
	for any $u\in H$. Besides, from Proposition~\ref{propositionB3}, it follows that 
	\begin{align*}
		-\frac{\LLLL\RR_{\alpha}f_N}{\RR_{\alpha }f_N}=-\frac{\alpha\RR_{\alpha}f_N-f_N}{\RR_{\alpha }f_N}\ge -\alpha.
	\end{align*}
	Hence, by applying the Fatou lemma, we obtain \eqref{B9}. 
    \end{proof}
	\section{Basic properties of the Donsker--Varadhan entropy}\label{propertyDVE}
    In this section, we gather some basic properties of the Donsker--Varadhan entropy used throughout the paper. While these results were originally established in \cite{DV1983} under seven technical assumptions, we show that their validity extends beyond the original setting and adapts to the case of the stochastic Navier--Stokes system. Let $\BH(\bmlambda)$ denote the Donsker--Varadhan entrophy defined in \eqref{ET1}. For $t>0$, we introduce
	\begin{align}\label{C1}
		\BH(t,\bmlambda):=\begin{cases}
			\int_{\DDDD^{\R}}\ent\left(\pi_{[0,t]}^*\bmlambda(u_{(-\infty,0]},\cdot)\Big| \pi_{[0,t]}^*\mathbb{P}_{u(0)}\right)\bmlambda(\dd u_{(-\infty,\infty)})\\&\!\!\!\!\!\!\!\mbox{if $\bmlambda\in \PP_s(\DDDD^\R)$},\\
			+\infty &\!\!\!\!\!\!\! \mbox{otherwise},
		\end{cases}
	\end{align}
    where for $r\in\R$ and $\JJ\subset\R$, $\pi^*_r$ and $\pi_{\JJ}^*$ are the push-forward operators defined in \eqref{pushforward1} and \eqref{pushforward2}, respectively. We introduce $\BB(\mathscr{F}^0_t)$ as the space of all bounded $\mathscr{F}^0_t$-measurable functions $\Psi:\DDDD^\R\to\R$, and recall that each $\Psi\in \BB(\mathscr{F}^0_t)$ has the following representation: 
    \[\Psi(u_{(-\ty,\ty)})=\Phi\circ\pi_{[0,t]}(u_{(-\ty,\ty)}),\qquad\forall u_{(-\ty,\ty)}\in\DDDD,\]
    where $\Phi:\DDDD^{[0,t]}\to\R$ is some bounded Borel-measurable function, see Theorem 2.12.3 in \cite{Bog2007}. For simplicity, we do not distinguish the function $\Psi $ and the corresponding representation $\Phi$ in what follows. Let us define
    \begin{align}
		\label{C2}\bar\BH(t,\bmlambda):=\begin{cases}\sup_{\Psi\in\BB(\mathscr{F}^0_t)}\left(\int_{\DDDD^\R}\left(\Psi(u_{[0,t]})-\log\int_{\DDDD} e^{\Psi}\dd \mathbb{P}_{u(0)}\right)\bmlambda(\dd u_{(-\infty,\infty)})\right)\\\qquad\qquad\qquad\qquad\qquad\quad\qquad\qquad\qquad\qquad\qquad\mbox{if $\bmlambda\in \PP_s(\DDDD^\R)$},&\\
			+\infty \qquad\qquad\qquad\qquad\qquad\qquad\qquad\qquad\qquad\quad\qquad\mbox{otherwise}.&
		\end{cases}
	\end{align}
    We need the following properties of $\BH(\bmlambda)$, $\BH(t,\bmlambda)$, and $\bar\BH(t,\bmlambda)$. 
	\begin{proposition} \label{propositionC1} For any $\bmlambda\in \PP(\DDDD^{\R})$, we have
		\begin{align}\label{C2-1}
		    \BH(t,\bmlambda)=t\BH(\bmlambda),\qquad\bar\BH(t,\bmlambda)\le t\BH(\bmlambda),\qquad \forall t>0,
		\end{align}
    and \begin{align}\label{C2-2}\lim_{t\to+\infty}\frac{\bar\BH(t,\bmlambda)}{t}=\BH(\bmlambda).
    \end{align} Furthermore, if $\bmlambda\in \PP_s(\DDDD^{\R})$, then
		\begin{align}\label{C3}
			\bar{\BH}(t,\bmlambda)&=\sup_{\Psi\in\YY_t}\left(\int_{\DDDD^\R}\left(\Psi(u_{[0,t]})-\log\int_{\DDDD} e^{\Psi}\dd \mathbb{P}_{u(0)}\right)\bmlambda(\dd u_{(-\infty,\infty)})\right)\\&=\sup_{\Psi\in\ZZ_t}\int_{\DDDD^\R}\Psi\dd\bmlambda,\label{C4}
		\end{align}    
     where $\YY_t$ denotes the space of functions $\Psi:\DDDD^{\R}\to \R$ of the form 
		\[\Psi=V\circ \pi_{t_1,t_2,\ldots,t_n}\]
		for some integer $n\ge 1$, partition $0\le t_1<t_2<\ldots<t_n\le t$, and function $V\in C_b(H^n)$, and $\ZZ_t$ denotes the set of functions $\Psi\in\YY_t$ such that 
		\begin{align}
			\label{C13}\int_{\DDDD}e^{\Psi}\dd\mathbb{P}_u\le 1,\qquad \forall u\in H.
		\end{align}
	\end{proposition}
    \begin{proof}
    The proof of \eqref{C2-1} and \eqref{C2-2} can be found in Theorems 3.1 and 3.6 in \cite{DV1983}, while \eqref{C3} can be established by using the definition \eqref{C2} combined with Dynkin's monotone-class argument. Thus, we omit the details and turn to \eqref{C4}. Let us denote by $\tilde{\BH}(t,\bmlambda)$ the right hand-side of \eqref{C4}. Utilizing \eqref{C3}, we have 
		\[\bar{\BH}(t,\bmlambda)\ge \tilde\BH(t,\bmlambda).\]
		To prove the above inequality in the other direction, let us take any $\Psi\in\YY_t$ and write
		\[\bar{\Psi}(u):=\log \int_{\DDDD}e^{\Psi}\dd\PPPPPP_u.\]
		Notice that
		$\bar{\Psi}\in C_b(H)$ and $\Psi-\bar{\Psi}\circ\pi_0\in\YY_t.$
		Moreover, we have 
		\begin{align*}
			\int_{\DDDD}e^{\Psi-\bar{\Psi}\circ\pi_0}\dd\PPPPPP_u=e^{-\bar{\Psi}(u)}\int_{\DDDD}e^{\Psi}\dd\PPPPPP_u=1,\qquad\forall u\in H.
		\end{align*}
        Therefore, we see that $\Psi-\bar{\Psi}\circ\pi_0\in\ZZ_t$. This implies
		\begin{align*}
			\int_{\DDDD^{\R}}\Psi\dd\bmlambda&=\int_{\DDDD^{\R}}(\Psi-\bar{\Psi}\circ\pi_0) \dd\bmlambda+\int_{\DDDD^{\R}}\bar{\Psi}\circ\pi_0 \dd\bmlambda\\&\le \tilde{\BH}(t,\bmlambda)+\int_{\DDDD^{\R}}\bar{\Psi}\circ\pi_0 \dd\bmlambda,
		\end{align*}
        and thus
		\begin{align*}
			\int_{\DDDD^{\R}}\left(\Psi-\log \int_{\DDDD}e^{\Psi}\dd\PPPPPP_{u(0)}\right)\dd\bmlambda\le \tilde{\BH}(t,\bmlambda).
		\end{align*}
		Taking supremum in $\Psi\in\YY_t$ and using \eqref{C3}, we obtain \eqref{C4}.      
    \end{proof}
    \begin{proposition}\label{lscDVE}
        $\BH(\bmlambda)$ is lower semi-continuous in $\bmlambda$.
    \end{proposition}
    \begin{proof}
        By Proposition 2.4.7 in \cite{KS12}, the map $u\mapsto \PPPPPP_u$ is continuous. Combining this and Theorem 3.3 in \cite{DV1983}, we have the lower semi-continuity of $\BH(\bmlambda)$.
    \end{proof}
    The main result of this section is the following LD upper bound.
    \begin{proposition}
		\label{propositionC4}
		Let $F\subset\PP_s(\DDDD^\R)$ be a closed set such that the family $\{\pi_0^*\bmlambda|\bmlambda\in F\}$ is tight in $\PP(H)$. Then, 
		\begin{align}
			\label{C9}
			\limsup_{t\to+\infty}\frac{1}{t}\log\left(\sup_{u\in H}\PPPPPP_{u}\{\bmzeta^{\per}_t\in F\}\right)\le-\inf_{\bmlambda\in F}\BH(\bmlambda).
		\end{align}
	\end{proposition}
	The proof is based on the following three auxiliary lemmas.
    \begin{lemma}\label{lemmaC5-0}
        A sequence $\{\bmlambda_n\}_{n\ge 1}\subset\PP_s(\DDDD^\R)$ is tight if and only if $\{\pi_I^*\bmlambda_n\}_{n\ge 1}$ is tight in $\PP(\DDDD^I)$ for some interval $I\subset\R$.
    \end{lemma}
	\begin{lemma}\label{lemmaC5}
		For any $A\subset\PP_s(\DDDD^\R)$, we have
		\begin{align}
			\label{C10}
			\limsup_{t\to+\infty}\frac{1}{t}\log\left(\sup_{u\in H}\PPPPPP_{u}\{\bmzeta^{\per}_t\in A\}\right)\le-\sup_{t>0}\sup_{\Psi\in \ZZ_t}\inf_{\bmlambda\in A}\frac{1}{t}\int_{\DDDD^{\R}}\Psi\dd\bmlambda.
		\end{align}
	\end{lemma}
	\begin{lemma}\label{lemmaC6}
		Let $A\subset\PP_s(\DDDD^\R)$ be a compact set. Then, for any $\epsilon>0$, there are open sets $A_1,A_2\ldots,A_l$ such that $A\subset\mcup_{i=1}^l A_i$ and
		\begin{align}\label{C16}
			-\inf_{1\le i\le l}\sup_{t>0}\sup_{\Psi\in\ZZ_t}\inf_{\bmlambda\in A_i}\frac{1}{t}\int_{\DDDD^{\R}}\Psi\dd\bmlambda\le-\inf_{\bmlambda\in A}\BH(\bmlambda)+\epsilon.
		\end{align}
		In particular, 
		\begin{align}\label{C17}
			\limsup_{t\to+\infty}\frac{1}{t}\log\left(\sup_{u\in H}\PPPPPP_{u}\{\bmzeta^{\per}_t\in A\}\right)\le-\inf_{\bmlambda\in A}\BH(\bmlambda).
		\end{align}
	\end{lemma}
    \begin{proof}[Proof of Lemma~\ref{lemmaC5-0}]
        It suffices to prove the tightness of $\{\bmlambda_n\}_{n\ge 1}$ by using the tightness of $\{\pi_I^*\bmlambda_n\}_{n\ge 1}$ for some interval $I\subset\R$. Without loss of generality, we take $I:=[0,1]$. We need the following result, which follows from Theorem 16.8 in \cite{Bil1999}.
        \begin{lemma}\label{TIGHT1}
            A sequence $\{\bmlambda_n\}_{n\ge1}\subset \PP(\DDDD^\R)$ is tight if and only if the following two conditions hold.
            \begin{enumerate}
                \item For each integer $m\ge 1$, 
                \[\lim_{a\to+\infty}\limsup_{n\to\infty}\bmlambda_n\left\{\uu\in \DDDD^\R\Big|\sup_{|t|\le m}\|\uu(t)\|\ge a\right\}=0.\]
                \item For each integer $m\ge 1$ and $\epsilon>0$,
                \[\lim_{\delta\to 0^+}\limsup_{n\to\infty}\bmlambda_n\left\{\uu\in \DDDD^\R\Big|w'_m(\uu,\delta)\ge\epsilon\right\}=0,\]
                where 
                \[w'_m(\uu,\delta):=\inf\max_{1\le i\le k}\sup_{s,r\in [t_{i-1,}t_i)}\|\uu(s)-\uu(r)\|,\]
                and the infimum is taken over all partitions $-m=t_0<t_1<\ldots<t_k=m$ such that $t_i-t_{i-1}>\delta$ for all $1< i<k$\footnote{The endpoints $i=1,k$ are excluded, due to the definition of the Skorohod topology; see Theorem 16.4 in \cite{Bil1999} and the subsequent discussion.}.
            \end{enumerate}
        \end{lemma}
    Let us prove the tightness of $\{\bmlambda_n\}_{n\ge 1}$ by verifying the above conditions. For simplicity, we confine ourselves to the verification of the second condition, as the other follows similarly. For integer $-m<j\le m$, let us introduce
    \[w'_{[j-1,j]}(\uu,\delta):=\inf\max_{1\le i\le k_j}\sup_{s,r\in [t^j_{l-1,}t^j_l)}\|\uu(s)-\uu(r)\|,\]
    and the infimum is taken over all partitions $j-1=t^j_0<\ldots<t_{k_j}^j=j$ satisfying $t^j_l-t^j_{l-1}>\delta$ for all $1\le l\le {k_j}$. For any $\epsilon'>0$, we take the partition of $[j-1,j]$ such that 
    \begin{align*}
        \max_{1\le i\le k_j}\sup_{s,r\in [t^j_{l-1,}t^j_l)}\|\uu(s)-\uu(r)\|\le w'_{[j-1,j]}(\uu,\delta)+\epsilon',\qquad\forall -m< j\le m.
    \end{align*}
    By taking maximum in $j$ and then sending $\epsilon'$ to $0$, we see that 
    \[w'_m(\uu,\delta)\le \max_{-m<j\le m}w'_{[j-1,j]}(\uu,\delta),\]
    which, together with the translation-invariance of $\{\bmlambda_n\}_{n\ge 1}$, the tightness of $\{\pi_{[0,1]}^*\bmlambda_n\}_{n\ge 1}$, and Theorem 13.2 in \cite{Bil1999}, implies
    \begin{align*}
        \lim_{\delta\to 0^+}\limsup_{n\to\infty}\bmlambda_n&\left\{\uu\in \DDDD^\R\Big|w'_m(\uu,\delta)\ge\epsilon\right\}\\&\le 2m\lim_{\delta\to 0^+}\limsup_{n\to\infty}\bmlambda_n\left\{\uu\in \DDDD^\R\Big|w'_{[0,1]}(\uu,\delta)\ge\epsilon\right\}=0.
    \end{align*}
    This verifies the second condition in Lemma \ref{TIGHT1} and thus completes the proof.
    \end{proof}
	\begin{proof}[Proof of Lemma~\ref{lemmaC5}] Let us take any $t>0$ and $\Psi\in \BB(\mathscr{F}^0_t)$ satisfying $\E_u e^{\Psi}\le 1$ for any $u\in H$. We claim that 
    \begin{align}
		\label{C11}
        \int_{\DDDD}\exp\left(\frac{1}{t}\int_0^s\Psi(\theta_ru_{[0,t]})\dd r\right)\dd\PPPPPP_{u}\le1
    \end{align}
    for any $s\ge0$ and $u\in H$. Indeed, for any $s\in [0,t]$ and $u\in H$, by the Jensen inequality, and the Markov property, we get
    \begin{align*}
		\int_{\DDDD}&\exp\left(\frac{1}{t}\int_0^s\Psi(\theta_ru_{[0,t]})\dd r\right)\dd\PPPPPP_{u}\le \frac{1}{t}\int_0^t\int_{\DDDD}\exp\left(\I_{\{r\le s\}}\Psi(\theta_ru_{[0,t]})\right)\dd\PPPPPP_{u}\dd r\notag\\&\qquad\quad\le \frac{1}{t}\int_0^t\left(\I_{\{r\le s\}}\E_u\exp\left(\Psi(\theta_ru_{[0,t]})\right)+\I_{\{r> s\}}\right)\dd r\notag\\&\qquad\quad=\frac{1}{t}\int_0^t\left(\I_{\{r\le s\}}\E_u\left(\E_{u(r)}\exp\left(\Psi(u_{[0,t]})\right)\right)+\I_{\{r> s\}}\right)\dd r\le 1.
	\end{align*}
    For $s>t$, we write
    \[\bar\Psi(r,u_{[0,t]}):=\sum_{k=0}^{[s/t]-1}\Psi(\theta_{r+kt}u_{[0,t]})+\I_{\{0\le r\le s-[s/t]t\}}\Psi(\theta_{r+[s/t]t}u_{[0,t]}),\]
    so we have 
    \[\int_0^s\Psi(\theta_ru_{[0,t]})\dd r=\int_0^t\bar\Psi(\theta_ru_{[0,t]})\dd r.\]
    Then, an application of the Jensen inequality gives
		\begin{align*}
			\int_{\DDDD}\exp\left(\frac{1}{t}\int_0^s\Psi(\theta_ru_{[0,t]})\dd r\right)\dd\PPPPPP_{u}\le \frac{1}{t} \int_0^t\int_{\DDDD}\exp\left(\bar\Psi(r,u_{[0,t]})\right)\dd\PPPPPP_{u}\dd r,
		\end{align*}
        where by successively applying the Markov property, 
        \begin{align*}
			\int_{\DDDD}\exp\left(\bar\Psi(r,u_{[0,t]})\right)\dd\PPPPPP_{u}&=\E_u\left(\E_u\left(\exp\left(\bar\Psi(r,u_{[0,t]})\right)\Big| \mathscr{F}_{r+[s/t]t}\right)\right)\\&\le\E_u\left(\exp\left(\sum_{k=0}^{[s/t]-1
			}\Psi(\theta_{r+kt}u_{[0,t]})\right)\right)\le\ldots\le 1.
		\end{align*}
        This implies \eqref{C11} as claimed. To conclude the proof, let us notice that 
        \begin{align*}
			\left|\int_{0}^s\Psi(\theta_ru_{\per,s})\mathrm{d}r-\int_0^s\Psi(\theta_ru_{[0,t]})\dd r\right|\le 2t\sup_{\DDDD^{\R}}|\Psi|.
		\end{align*}
        Combining this with \eqref{C11}, we derive
		\begin{align*}
			\int_{\DDDD}\exp\left(\frac{s}{t}\int_{\DDDD^{\R}}\Psi\mathrm{d}\bmzeta^{\per}_s\right)\dd \PPPPPP_u&=\int_{\DDDD}\exp\left(\frac{1}{t}\int_{0}^s\Psi(\theta_ru_{\per,s})\mathrm{d}r\right)\dd \PPPPPP_u\\&\le \exp\left(2\sup_{\DDDD^{\R}}|\Psi|\right)
		\end{align*}
		for any $s>0$ and $u\in H$. Hence, for any $A\subset \PP_s(\DDDD^\R)$, 
		\begin{align*}
			\PPPPPP_{u}\{\bmzeta^{\per}_s\in A\}&\le \exp\left(-\inf_{\bmlambda\in A}\frac{s}{t}\int_{\DDDD^{\R}}\Psi\mathrm{d}\bmlambda\right)\E_u\exp\left(\frac{s}{t}\int_{\DDDD^{\R}}\Psi\mathrm{d}\bmzeta^{\per}_s\right)\\&\le \exp\left(-\inf_{\bmlambda\in A}\frac{s}{t}\int_{\DDDD^{\R}}\Psi\mathrm{d}\bmlambda+2\sup_{\DDDD^{\R}}|\Psi|\right),
		\end{align*}
		which implies
		\begin{align}\label{C15}
			\limsup_{s\to+\infty}\frac{1}{s}\log\left(\sup_{u\in H}\PPPPPP_{u}\{\bmzeta^{\per}_s\in A\}\right)\le-\inf_{\bmlambda\in A}\frac{1}{t}\int_{\DDDD^{\R}}\Psi\dd\bmlambda.
		\end{align} 
		In particular, this estimate holds for any $t>0$ and $\Psi\in \ZZ_t$. By taking infimum in $t,\Psi$ in both sides, we obtain \eqref{C10}.
	\end{proof}
	\begin{proof}[Proof of Lemma~\ref{lemmaC6}]
		Let $\bmlambda_0\in A$. Utilizing \eqref{C2-2} and \eqref{C4}, for any $\epsilon>0$, there is a time $t_0>0$ and a function $\Psi_0\in \ZZ_{t_0}$ such that 
		\[\BH(\bmlambda_0)\le \frac{1}{t_0}\int_{\DDDD^{\R}}\Psi_0\dd\bmlambda_0+\frac{1}{2}\epsilon.\]
        Let us introduce the following continuous functional
		\[\bmlambda\mapsto\frac{1}{t_0}\int_{\DDDD^{\R}}\Psi_0\dd\bmlambda,\qquad \PP_s(\DDDD^\R)\to \R.\]
        Then, there is a neighborhood $A_0$ of $\bmlambda_0$ such that 
		\[\inf_{\bmlambda\in A}\BH(\bmlambda)\le \inf_{\bmlambda\in A_0}\frac{1}{t_0}\int_{\DDDD^{\R}}\Psi_0\dd\bmlambda+\epsilon\le \sup_{t>0}\sup_{\Psi\in\ZZ_t}\inf_{\bmlambda\in A_0}\frac{1}{t}\int_{\DDDD^{\R}}\Psi\dd\bmlambda+\epsilon.\]
		All these neighborhoods form an open cover of the compact set $A$. Hence, we can find open sets $A_1,A_2,\ldots,A_l$ such that $A\subset \mcup_{i=1}^l A_i$ and 
		\[\inf_{\bmlambda\in A}\BH(\bmlambda)\le \sup_{t>0}\sup_{\Psi\in\ZZ_t}\inf_{\bmlambda\in A_i}\frac{1}{t}\int_{\DDDD^{\R}}\Psi\dd\bmlambda+\epsilon\]
		for any $i=1,2,\ldots,l$. This leads to \eqref{C16} as desired.
		
		To prove \eqref{C17}, let us take any $\epsilon>0$, and choose $\{A_i\}_{i=1}^l$ as the open cover of $A$ satisfying \eqref{C16}. Then, by Lemma~\ref{lemmaC5}, we derive
		\begin{align*}
			\limsup_{t\to+\infty}\frac{1}{t}\log\left(\sup_{u\in H}\PPPPPP_{u}\{\bmzeta^{\per}_t\in A\}\right)&\le \sup_{1\le i\le l}\limsup_{t\to+\infty}\frac{1}{t}\log\left(\sup_{u\in H}\PPPPPP_{u}\left\{\bmzeta^{\per}_t\in  A_i\right\}\right)\\&\le -\inf_{1\le i\le l}\sup_{t>0}\sup_{\Psi\in \ZZ_t}\inf_{\bmlambda\in A_i}\frac{1}{t}\int_{\DDDD^{\R}}\Psi\dd\bmlambda\\
			&\le-\inf_{\bmlambda\in A}\BH(\bmlambda)+\epsilon,
		\end{align*} 
		which implies \eqref{C17} by taking $\epsilon\to0^+$. 
	\end{proof}
	Now we are in a position to prove Proposition~\ref{propositionC4}.
	\begin{proof}[Proof of Proposition~\ref{propositionC4}]The proof is divided into three steps.
 
    \noindent\textit{Step 1.} Let us set $F_{\MM}:=\{\pi_0^*\bmlambda|\bmlambda\in F\}$, and take any sequences $\epsilon_n,\tilde\epsilon_n$ such that $\epsilon_n,\tilde\epsilon_n\to0$, as $n\to\infty$.
    By the tightness of $F_{\MM}$, for any $n\ge 1$, there is a compact sets $K_n\subset H$ such that 
	\begin{align}
		\label{C18}\mu(K_n)\ge 1-\epsilon_n
	\end{align}
	for any $\mu\in F_{\MM}$. As $\{\PPPPPP_u\}_{u\in K_n}$ is tight, we can find a compact sets $\KK_{n}\subset \DDDD^{[0,1]}$ such that 
		\begin{align}
			\label{C19}
            \PPPPPP_u(\KK_{n})\ge 1-\tilde\epsilon_n
		\end{align}
		for any $u\in K_n$. Repeating the argument of the derivation of \eqref{ET3}, we see that
		\begin{align*}
			\E_{u}\left(\exp\left(c\I_{K_n}(u(0))\I_{\KK_{n}^c}(u_{[0,1]})\right)\right)\le 1+(e^{c}-1)\tilde\epsilon_n
		\end{align*}
        for any $c>0$ and $u\in H$. Applying \eqref{C11} with
        \[\Psi(u_{[0,1]}):=c\I_{K_n}(u(0))\I_{\KK^c_{n}}(u_{[0,1]})-\log(1+(e^c-1)\tilde\epsilon_n),\]
        we get
		\begin{align}\label{C20}
			\int_{\DDDD}\exp\left(c\int_0^t\I_{K_n}(u(s))\I_{\KK_{n}^c}(u_{[s,s+1]})\dd s\right)\dd\PPPPPP_{u}\le \exp\left(t\log\left(1+(e^c-1)\tilde\epsilon_n\right)\right)
		\end{align}
		for any $t\ge 0$ and $u\in H$. Notice that
		\begin{align*}
			\int_0^t\I_{K_n}(u(s))&\I_{\KK_{n}^c}(u_{[s,s+1]})\dd s\ge \int_0^t\I_{\KK_{n}^c}(u_{[s,s+1]})\dd s-\int_0^t\I_{K_n^c}(u(s))\dd s\\&\ge\int_0^t\I_{\KK_n^c}\circ\pi_{[0,1]}(\theta_su_{\per,t})\dd s-\int_0^t\I_{K_n^c}\circ\pi_0(\theta_su_{\per,t})\dd s-1\\&=t\pi_{[0,1]}^*\bmzeta^{\per}_t(\KK_{n}^c)-t\pi^*_0\bmzeta^{\per}_t(K_n^c)-1.
		\end{align*}
		This, together with \eqref{C20}, implies 
		\begin{align*}
			\mathbb{E}_{u}\exp\left(ct\left(\pi_{[0,1]}^*\bmzeta^{\per}_t(\KK_{n}^c)-\pi^*_0\bmzeta^{\per}_t(K_n^c)-\frac{1}{t}\right)\right) \le \exp\left(t\log\left(1+(e^c-1)\tilde\epsilon_n\right)\right).
		\end{align*}
		Moreover, utilizing \eqref{C18}, we see that \[\pi_0^*\bmzeta^{\per}_t(K_n^c)\le \epsilon_n,\qquad \forall n\ge 1\]
        on the event $\{\bmzeta^{\per}_t\in F\}$. Therefore, 
		\begin{align}\label{C21}
			\exp(t\log&\left(1+(e^c-1)\tilde\epsilon_n\right))\notag\\&\ge \mathbb{E}_{u}\left(\I_{\{\bmzeta^{\per}_t\in F\}}\exp\left(ct\left(\pi_{[0,1]}^*\bmzeta^{\per}_t(\KK_{n}^c)-\epsilon_n-\frac{1}{t}\right)\right)\right)\notag\\&\ge e^{c\epsilon_n t}\mathbb{P}_{u}\left\{\bmzeta^{\per}_t\in F\mcap\left\{\bmlambda\Big|\pi_{[0,1]}^*\bmlambda(\KK^c_{n})> 2\epsilon_n+\frac{1}{t}\right\}\right\}.
		\end{align}
		Let us choose 
        \[\epsilon_n=\frac{1}{n},\qquad 
        \tilde\epsilon_n=e^{-an^2},\qquad c=an^2\] in \eqref{C21}, where $a>0$. Then, 
		\[\mathbb{P}_{u}\left\{\bmzeta^{\per}_t\in F\mcap\left\{\bmlambda\Big|\pi_{[0,1]}^*\bmlambda(\KK^c_{n})> 2\epsilon_n+\frac{1}{t}\right\}\right\}\le e^{t\log2-ant},\]
		which implies 
		\[\mathbb{P}_{u}\left\{\bmzeta^{\per}_t\in F\mcap F^c_t\right\}\le e^{t\log2}\frac{e^{-at}}{1-e^{-at}},\]
		where 
		\[F_t:=\mcap_{n=1}^{\infty}\left\{\bmlambda\in\PP_s(\DDDD^{\R})\Big|\pi_{[0,1]}^*\bmlambda(\KK^c_n)\le \frac{2}{n}+\frac{1}{t}\right\}.\]
		This yields
		\begin{align}
			\label{C22}
			\limsup_{t\to+\infty}\frac{1}{t}\log\left(\sup_{u\in H}\PPPPPP_{u}\left\{\bmzeta^{\per}_t\in F\mcap F^c_t\right\}\right)\le \log 2-a.
		\end{align}
		\noindent\textit{Step 2.} Let us write 
		\[F_\infty:=\mcap_{t>0}F_t.\]
		By the Portmanteau theorem, $F_\infty$ is a closed set. Let $G\subset\PP_s(\DDDD^{\R})$ be an open set such that
		\[F\mcap F_{\infty}\subset G.\]
		We claim that there is $t_0>0$ such that $F\mcap F_{t_0}\subset G$. Indeed, by contradiction, let us assume that for any $k\ge 1$,
		\[(F\mcap F_k)\mcap G^c\neq \emptyset.\]
		For each $k\ge 1$, we choose $\bmlambda_k\in (F\mcap F_k)\mcap G^c$. Then, for any $k,n\ge 1$, 
		\[\pi_{[0,1]}^*\bmlambda_k(\KK^c_n)\le \frac{2}{n}+\frac{1}{k}.\]
		This, together with Lemma~\ref{lemmaC5-0}, implies the tightness of $\{\bmlambda_k\}_{k=1}^{\infty}$. Therefore, we obtain a contradiction, since any limit point of $\{\bmlambda_k\}_{k=1}^{\infty}$ must belong to 
		\[G^c\mcap (F\mcap F_{\infty})=\emptyset.\]
		
		\noindent\textit{Step 3.} To conclude the proof, we note that by the Prokhorov theorem, $F\mcap F_{\infty}$ is a compact set. Then, by applying Lemma~\ref{lemmaC6}, for any $\epsilon>0$, there is an open cover $\{A_i\}_{i=1}^l$ of $F\mcap F_{\infty}$ satisfying \eqref{C16}. As shown in the previous step, we can find $t_0>0$ such that 
		\[F\mcap F_{t_0}\subset \mcup_{i=1}^lA_i.\]
		Therefore, 
		\begin{align*}
			\PPPPPP_{u}\left\{\bmzeta^{\per}_t\in F\right\}&\le \PPPPPP_{u}\left\{\bmzeta^{\per}_t\in F\mcap F_{t_0}\right\}+\PPPPPP_{u}\left\{\bmzeta^{\per}_t\in F\mcap F_{t_0}^c\right\}\\&\le \sum_{i=1}^l\PPPPPP_{u}\left\{\bmzeta^{\per}_t\in A_i\right\}+\PPPPPP_{u}\left\{\bmzeta^{\per}_t\in F\mcap F_{t}^c\right\}
		\end{align*}
		for any $t\ge t_0$ and $u\in H$. This, together with \eqref{C10}, \eqref{C16}, and \eqref{C22}, implies 
		\begin{align*}
			\limsup_{t\to+\infty}\frac{1}{t}\log&\left(\sup_{u\in H}\PPPPPP_{u}\left\{\bmzeta^{\per}_t\in F\right\}\right)\\&\le \max\left\{-\inf_{1\le i\le l}\sup_{t>0}\sup_{\Psi\in\ZZ_t}\inf_{\bmlambda\in A_i}\frac{1}{t}\int_{\DDDD^{\R}}\Psi\dd\bmlambda,\log2-a\right\}\\&\le \max\left\{-\inf_{\bmlambda\in F}\BH(\bmlambda)+\epsilon,\log 2-a\right\}.
		\end{align*}
		As $F$ is independent of $a$, by first choosing $a$ sufficiently large and then taking $\epsilon\to0^+$, we obtain \eqref{C9}.
	\end{proof}
	
	\section{An improved version of Kifer's criterion}
	In this section, we provide an improvement of the Kifer-type criterion introduced in Section 3 in \cite{JNPS18} by weakening the exponential growth condition (3.4) therein to the exponential tightness along any subsequence.
 
    To formulate the result, we recall the following basic settings. Let $\XXXX$ be a Polish space, and let $\PP(\XXXX)$ denote the space of all Borel probability measures on $\XXXX$ endowed with the topology of weak convergence. Consider a family of random probability measures $\{\zeta_{\theta}\}_{\theta\in\Theta}$ on $\XXXX$ that are defined on some probability spaces $(\Omega_{\theta},\mathscr{F}_{\theta},\PPPPPP_{\theta})$, where $\Theta:=(\Theta,\prec)$ is a directed set possessing a cofinal sequence. In what follows, the dependence of the probability spaces on $\theta$ shall be omitted, which will not lead to any confusion. Let $r(\theta)$ be a positive function satisfying $\lim_{\theta\in\Theta}r(\theta)=+\infty$. Let us define the pressure function
	\begin{align}\label{kifer1}
		Q(V):=\lim_{\theta\in\Theta}\frac{1}{r(\theta)}\log\int_{\Omega}\exp(r(\theta)\langle V,\zeta_{\theta}\rangle)\dd\PPPPPP,\qquad V\in C_b(\XXXX),
	\end{align}
    if the above limit exists, and introduce the Legendre transform of $Q$:
	\begin{align}\label{kifer2}I(\lambda):=\sup_{V\in C_b(\XXXX)}(\langle V,\lambda\rangle-Q(V)),\end{align}
	where $\lambda\in\PP(\XXXX)$. A measure $\sigma_V\in \PP(\XXXX)$ satisfying the relation 
	\[Q(V)=\langle V,\sigma_V\rangle-I(\sigma_V)\]
	is called an equilibrium state for $V$. The basic properties of $Q$ and $I$ can be found in \cite{JNPS18}.
	\begin{proposition}\label{proposition-kifer}
		Suppose that the following two conditions holds for $\{\zeta_{\theta}\}_{\theta\in\Theta}$.
		\begin{enumerate}
			\item \textbf{Exponential tightness along any subsequence.} For any subsequence $\{\theta_n\}_{n\ge 1}$ such that $r(\theta_n)\to+\infty$, as $n\to+\infty$, $\{\zeta_{\theta_n}\}_{n\ge 1}$ is exponentially tight, that is, for any $l>0$, there is a compact set $\KK_l\subset \PP(\XXXX)$ such that 
			\begin{align}\label{kifer2-1}
				\limsup_{n\to+\infty}\frac{1}{r(\theta_n)}\log\PPPPPP\{\zeta_{\theta_n}\in \KK_l^c\}\le -l.
			\end{align}
			\item \textbf{Existence of pressure.} The limit \eqref{kifer1} exists for any $V\in C_b(\XXXX)$.
		\end{enumerate}
		Then, the relation \eqref{kifer2} defines a good rate function, and the following upper bound holds:  for any closed set $F\subset\PP(\XXXX)$,
		\begin{align}
			\label{kifer3}
			\limsup_{\theta\in\Theta}\frac{1}{r(\theta)}\log\PPPPPP\{\zeta_{\theta}\in F\}\le-\inf_{\lambda\in F}I(\lambda).
		\end{align}
		Furthermore, we have the following lower bound: for any open set $G\subset\PP(\XXXX)$,
		\begin{align}
			\label{kifer4}\liminf_{\theta\in\Theta}\frac{1}{r(\theta)}\log\PPPPPP\{\zeta_{\theta}\in G\}\ge-\inf_{\lambda\in G}I(\lambda),
		\end{align}
		provided that the following condition is also satisfied for $\{\zeta_\theta\}_{\theta\in\Theta}$.
		\begin{enumerate}[start=3]
			\item \textbf{Uniqueness of equilibrium.} There exists a vector space $\VV\subset C_b(\XXXX)$ such that the restrictions of the functions in $\VV$ on any compact subset $K\subset \XXXX$ form a dense subspace of $C(K)$, and that, for each $V\in\VV$, there is a unique\footnote{Let us notice that if $I$ is a good rate function, then the set of equilibrium states for any $V\in C_b(\XXXX)$ is non-empty.} equilibrium state.
		\end{enumerate}
	\end{proposition}
	\begin{proof}
		It is clear that if for each subsequence $\{\theta_n\}_{n\ge 1}$ such that $r(\theta_n)\to+\infty$, as $n\to+\infty$, the LDP holds for $\{\zeta_{\theta_n}\}_{n\ge 1}$ with a rate function $I$ independent of the subsequence $\{\theta_n\}_{n\ge 1}$, then the whole family $\{\zeta_\theta\}$ satisfies the LDP with the same rate function $I$. Therefore, let us fix such a subsequence $\{\zeta_{\theta_n}\}_{n\ge 1}$.
		
		\noindent\textit{Step 1: Covering property.} Compared to the proof of Theorem 3.3 in \cite{JNPS18}, the main difficulty in our situation is to establish the following covering property for the level sets of $I$: for any $a\ge0$, there is a sequence of compact sets $\{\MMM_{a,l}\}_{l\ge 1}\subset \PP(\XXXX)$ such that 
		\begin{align}\label{kifer4-1}
			L_a:=\{\lambda\in\PP(\XXXX)|I(\lambda)\le a\}\subset \mcap_{l\ge 1}\MMM_{a,l},
		\end{align}
		and 
		\begin{align}
			\label{kifer5}\limsup_{n\to+\infty}\frac{1}{r(\theta_n)}\log\PPPPPP\{\zeta_{\theta_n}\in \MMM_{a,l}^c\}\le -\frac{l}{2}
		\end{align}
		for any integer $l\ge 1$. This covering property is established by using a crucial estimate derived in \cite{CLXZ2024}. To formulate the estimate, for each integer $l\ge 1$, let $\KK_l$ be the compact set satisfying \eqref{kifer2-1}. As we are considering a sequence of random probability measures, we can assume, up to enlarging the compact set $\KK_l$, that
		\begin{align}\label{kifer5-1}
			\PPPPPP\{\zeta_{\theta_n}\in\KK_l^c\}\le e^{-\frac{l}{2}r(\theta_n)}
		\end{align}
		for all $n\ge 1$. Since $\KK_l$ is compact, for each integer $m\ge 1$, there is a compact set $K_{l,m}$ such that 
		\[\sigma(K_{l,m}^c)\le \frac{1}{m}\]
		for any $\sigma\in\KK_l$. Moreover, in view of Corollary 3.4 in \cite{CLXZ2024}, we have
		\begin{align}\label{kifer6}
			I(\lambda)\ge l\left(\lambda(K_{l,m}^c)-\frac{1}{m}\right)
		\end{align}
        for any $\lambda\in \PP(\XXXX)$ and any $l,m\ge 1$. Let us define
		\[\MMM_{a,l}:=\mcap_{j\ge l}\mcap_{m\ge 1}\left\{\lambda\in\PP(\XXXX)\Big| \lambda(K_{j,m}^c)\le \frac{a}{j}+\frac{1}{m}\right\}.\]
		By the Portmanteau theorem and the Prokhorov theorem, for any $a\ge0$ and any $l\ge 1$, $\MMM_{a,l}$ is a compact subset of $\PP(\XXXX)$. Furthermore, by the definition of $\MMM_{a,l}$ and \eqref{kifer6}, we have 
		\[L_a\subset \MMM_{a,1}= \mcap_{l\ge 1}\MMM_{a,l},\]
        which implies \eqref{kifer4-1}. As for \eqref{kifer5}, let us note that for any $l\ge 1$,
		\[\mcap_{j\ge l}\KK_j\subset\MMM_{a,l}.\]
		This, together with \eqref{kifer5-1}, implies 
		\[\PPPPPP\{\zeta_{\theta_n}\in \MMM_{a,l}^c\}\le \sum_{j\ge l}\PPPPPP\{\zeta_{\theta_n}\in \KK_{j}^c\}\le\sum_{j\ge l}e^{-\frac{j}{2}r(\theta_n)}=\frac{e^{-\frac{l}{2}r(\theta_n)}}{1-e^{-\frac{r(\theta_n)}{2}}}\]
		and therefore \eqref{kifer5} as desired.
		
		\noindent\textit{Step 2: Goodness of $I$ and LD upper bound.} Since $I$ is lower semi-continuous, for any $a\ge 0$, $L_a$ is a closed subset of $\PP(\XXXX)$. Combining this with \eqref{kifer4-1}, we obtain the compactness of $L_a$ and therefore the goodness of the rate function $I$. The LD upper bound \eqref{kifer3} follows in the same way as in \cite{JNPS18}, and the proof is omitted.
		
		\noindent\textit{Step 3: LD lower bound.} The LD lower bound \eqref{kifer4} is established by using the covering property \eqref{kifer4-1} combined with the estimate \eqref{kifer5} and the following lemma established in \cite{JNPS18}, see Proposition 3.4 therein.
		\begin{lemma}\label{lemma-kifer}
			Let $V_1,V_2,\ldots,V_m\in\VV$ and define 
			\begin{align}\label{kifer7}
				f_m:&\PP(\XXXX)\to\R^m
				\notag\\& \lambda\mapsto (\langle V_1,\lambda\rangle,\langle V_2,\lambda\rangle,\ldots,\langle V_m,\lambda\rangle)
			\end{align}
			and $\zeta^m_{\theta_n}:=f_m(\zeta_{\theta_n})$. Then, under the conditions of Proposition~\ref{proposition-kifer}, for any closed set $\FF\subset \R^m$,
			\[\limsup_{n\to+\infty}\frac{1}{r(\theta_n)}\log\PPPPPP\{\zeta^m_{\theta_n}\in \FF\}\le-\inf_{\alpha\in \FF}I_m(\alpha),\]
			and for any open set $\WW\subset \R^m$,
			\[\liminf_{n\to+\infty}\frac{1}{r(\theta_n)}\log\PPPPPP\{\zeta^m_{\theta_n}\in \WW\}\ge-\inf_{\alpha\in \WW}I_m(\alpha),\]
			where $I_m(\alpha):=\inf_{\lambda\in f_m^{-1}(\alpha)} I(\lambda)$.
		\end{lemma}
		Let $G\subset\PP(\XXXX)$ be an open set. Without loss of generality, we assume that $\inf_{\lambda\in G}I(\lambda)<\infty$. Then, for any $\epsilon>0$, there is a measure $\lambda_{\epsilon}\in G$ such that
		\begin{align}\label{kifer7-1}
			I(\lambda_{\epsilon})\le \inf_{\lambda\in G}I(\lambda)+\epsilon.
		\end{align}
		Let us set 
        \begin{align}
            \label{kifer7-2}
            a=\inf_{\lambda\in G}I(\lambda)+\epsilon,\qquad l=[2(a+2)]+1,
        \end{align}
        so we have $\lambda_{\epsilon}\in L_a\subset \MMM_{a,l}$. Moreover, let $\{V_k\}_{k\ge 1}\subset \VV$ be a sequence of functions such that $\|V_k\|_{\infty}=1$ for any $k\ge 1$ and 
		\[d(\lambda_1,\lambda_2):=\sum_{k=1}^{\infty}2^{-k}|\langle V_k,\lambda_1\rangle-\langle V_k,\lambda_2\rangle|\]
		metrizes the topology of weak convergence on $\MMM_{a,l}$\footnote{For any given compact set $\KK\subset \PP(\XXXX)$, the existence of such functions is proved in \cite{JNPS18}.}. As $G$ is open, there is an integer $m\ge 1$ and $\delta>0$ such that if $\lambda\in\MMM_{a,l}$ satisfies 
		\begin{align}\label{kifer8}\sum_{k=1}^{m}2^{-k}|\langle V_k,\lambda\rangle-\langle V_k,\lambda_{\epsilon}\rangle|< \delta,\end{align}
		then $\lambda\in G$. Let us endow $\R^m$ with the norm
		\[\|x\|_m:=\sum_{k=1}^m2^{-k}|x_k|,\]
		where $x=(x_1,x_2,\ldots,x_m)$, and define the map $f_m$ by \eqref{kifer7} with the functions $V_1,V_2,\ldots,V_m$ given in \eqref{kifer8}. Let $x_{\epsilon}:=f_m(\lambda_{\epsilon})$. Under the above notations, we have 
		\[f_m^{-1}(B_{\R^m}(x_{\epsilon},\delta))\mcap \MMM_{a,l}\subset G.\]
		Therefore,
		\begin{align*}
			\PPPPPP\{\zeta_{\theta_n}\in G\}&\ge \PPPPPP\{\zeta_{\theta_n}\in G\mcap \MMM_{a,l}\}\ge \PPPPPP\{\zeta_{\theta_n}\in f_m^{-1}(B_{\R^m}(x_{\epsilon},\delta))\mcap \MMM_{a,l}\}
			\\&\ge \PPPPPP\{\zeta^m_{\theta_n}\in B_{\R^m}(x_{\epsilon},\delta)\}-\PPPPPP\{\zeta_{\theta_n}\in \MMM_{a,l}^c\}.
		\end{align*}
		Utilizing Lemma~\ref{lemma-kifer}, the relations \eqref{kifer5}, \eqref{kifer7-1}, and \eqref{kifer7-2}, there is an integer $n_0\ge 1$ such that if $n\ge n_0$, then
		\begin{align*}
			\PPPPPP\{\zeta^m_{\theta_n}\in B_{\R^m}(x_{\epsilon},\delta)\}&\ge \exp\left(r(\theta_n)\left(-\inf_{\alpha\in B_{\R^m}(x_{\epsilon},\delta)}I_m(\alpha)-1\right)\right)\\&\ge \exp\left(r(\theta_n)\left(-I(\lambda_{\epsilon})-1\right)\right)\\&\ge \exp\left(r(\theta_n)\left(-\frac{l}{2}+1\right)\right)\ge \sqrt{e}\PPPPPP\{\zeta_{\theta_n}\in \MMM_{a,l}^c\},
		\end{align*}
		which implies 
		\[\PPPPPP\{\zeta_{\theta_n}\in G\}\ge \left(1-\frac{1}{\sqrt{e}}\right)\PPPPPP\{\zeta^m_{\theta_n}\in B_{\R^m}(x_{\epsilon},\delta)\}.\]
		This, together with Lemma~\ref{lemma-kifer} and \eqref{kifer7-1} again, leads to
		\begin{align*}
			\liminf_{n\to+\infty}\frac{1}{r(\theta_n)}\log\PPPPPP\{\zeta_{\theta_n}\in G\}&\ge \liminf_{n\to+\infty}\frac{1}{r(\theta_n)}\log\PPPPPP\{\zeta^m_{\theta_n}\in B_{\R^m}(x_{\epsilon},\delta)\}\\
			&\ge-\inf_{\alpha\in B_{\R^m}(x_{\epsilon},\delta)}I_m(\alpha)\ge -I(\lambda_{\epsilon})\ge -\inf_{\lambda\in G}I(\lambda)-\epsilon.
		\end{align*}
		As $\epsilon$ is arbitrary, the LD lower bound \eqref{kifer4} holds. This completes the proof.
	\end{proof}
	\section{Large-time asymptotics 
		of generalized Markov semigroups: improved abstract criterion}
	We present a new version of the abstract criteria given in Theorem A.4 in \cite{MN18} and Theorem 4.1 in \cite{JNPS18}. In these works, sufficient conditions are provided, which ensure a multiplicative ergodic theorem for generalized Markov semigroups. The multiplicative ergodic theorem implies the existence of the pressure function and the uniqueness of equilibrium required in Kifer's criterion. However, as discussed in Section~\ref{scheme}, these criteria do not apply to the weak disspative case considered in the present paper. The difficulty is overcome by using a stronger irreducibility property ensured by the parabolic smoothing effect, cf. (A.14) in \cite{MN18} and \eqref{E0}.
	
	To formulate the result, let us start with some definitions. Let $\XXXX$ be a separable Banach space endowed with the norm $\|\cdot\|$, let $\MM^+(\XXXX)$ be the set of the non-negative Borel measures on $\XXXX$, and let $\{P_t(u,\cdot), u\in\XXXX, t\ge0\}$ be a generalized Markov family of transition kernels, cf. Definition A.3 in \cite{MN18}. We define two associated semigroups by the following relations:
	\begin{align*}
		\PPPP_t&: C_b(\XXXX)\to C_b(\XXXX),&\PPPP_t f(u)&:=\int_{\XXXX} f(v)P_t(u,\dd v),\\
		\PPPP_t^*&:\MM^+(\XXXX)\to\MM^+(\XXXX), &\PPPP^*_t\mu(A)&:=\int_{\XXXX}P_t(u,A)\mu(\dd u).
	\end{align*}
	We assume the following structure for the phase space $\XXXX$: there is an increasing sequence of compact subsets $\{\XXXX_{R}\}_{R=1}^{\infty}$ such that \[\XXXX_{\infty}:=\mcup_{R=1}^{\infty}\XXXX_{R}\]
	is dense in $\XXXX$ and for each $R\ge 1$,
	\begin{align}
		\label{E0-1}\XXXX_{R}\subset \overline{B_\XXXX(R)}.
	\end{align}
	Let us define the following weight function
	\begin{align}
		\label{E0-2}\wwww(u):=\GG(\|u\|)
	\end{align}
	where $\GG:[0,+\infty)\to[1,+\infty)$ is a continuous increasing function satisfying
	\[\lim_{y\to+\infty}\GG(y)=+\infty.\]
	We denote by $C_{\wwww}(\XXXX)$ (respectively, $L^{\infty}_{\wwww}(\XXXX)$) the weighted space of continuous (measurable) functions $f:\XXXX\to\R$ such that the norm
	\[\|f\|_{L^{\infty}_{\wwww}}:=\sup_{u\in \XXXX}\frac{|f(u)|}{\wwww(u)}<\infty.\]
	Let $\PP_{\wwww}(\XXXX)$ be the set of measures $\mu\in\PP(X)$ such that $\langle \wwww,\mu\rangle<\infty$. We say a family $\CC\subset C_b(\XXXX)$ is determining, if any two measures $\mu,\nu\in\MM^+(\XXXX)$, satisfying 
	\[\langle f,\mu\rangle=\langle f,\nu\rangle\]
	for any $f\in \CC$, must coincide. 
	
	We have the following improved version of Theorem 4.1 in \cite{JNPS18}.
	\begin{proposition}\label{propositionE1}
		Let $\{P_k(u,\cdot)|u\in \XXXX,k\in\N\}$ be a generalized Markov family of transition kernels satisfying the following conditions.
		\begin{enumerate}
			\item \textbf{Concentration property.} There is an integer $l_0\ge 1$ such that $P_{l_0}(u,\cdot)$ is concentrated on $\XXXX_{\infty}$ for any $u\in \XXXX$.
			\item \textbf{Growth condition.} There is a weight function $\wwww$ satisfying \eqref{E0-2} and an integer $R_0\ge 1$ such that 
			\begin{align*}\sup_{k\ge 0}\frac{\|\PPPP_k\wwww\|_{L^{\infty}_{\wwww}}}{\|\PPPP_k\mathbf{1}\|_{R_0}}<\infty,
			\end{align*}
			where $\mathbf{1}$ is the function on $\XXXX$ identically equal to $1$, and $\|\cdot\|_R$ denotes the $L^{\infty}$-norm on $\XXXX_{R}$.
			\item \textbf{Uniform irreducibility.} For any integers $\rho,R\ge 1$ and $r>0$, there is an integer $l=l(\rho,R,r)\ge 1$ and a positive number $p=p(\rho,R,r)$ such that 
			\begin{align*}
				P_l(u,B_{\XXXX_{\rho}}(\hat u,r))\ge p, 
			\end{align*}
			where $u\in \overline{B_{\XXXX}(R)}$, $\hat{u}\in \XXXX_{\rho}$, and $B_{\XXXX_{\rho}}(\hat u,r):=\XXXX_{\rho}\mcap B_\XXXX(\hat u,r).$
			\item \textbf{Uniform Feller property.} There is a determining family $\CC\subset C_b(\XXXX)$ and an integer $R_0\ge 1$ such that $\mathbf{1}\in\CC$ and the sequence $\{\|\PPPP_k\mathbf{1}\|^{-1}_R\PPPP_k f\}_{k\ge 0}$ is uniformly equicontinuous on $\XXXX_R$ for any $f\in\CC$ and $R\ge R_0$.
			\item \textbf{Uniform tail probability estimate.} There is an integer $m\ge 1$ such that for any integer $\RR\ge 1$,
			\begin{align}\label{E0-3}\sup_{u\in \overline{B_\XXXX(\RR)}}\int_{\XXXX_R^c}\wwww(v)P_m(u,\dd v)\to0,\end{align}
			as $R\to+\infty$.
		\end{enumerate}
		Then, the following properties hold.
		\begin{enumerate}
			\item There is at most one measure $\mu\in \PP_{\wwww}(\XXXX)$ satisfying the relation
			\[\PPPP^*\mu=c\mu\]
            for some number $c\ge0$, as well as the following tail probability estimate:
			for each positive integer $\rho$ and $r>0$,
			\begin{align}
				\label{E0}
				p^{-1}(\rho,R,r)\GG(R)\int_{\overline{B_\XXXX(R)}^c}\wwww\dd \mu\to0,
			\end{align}
			as $R\to+\infty$. Here, $\GG$ is the function in \eqref{E0-2}.
			\item If such a measure $\mu$ exists, then the corresponding eigenvalue $c$ is positive, the support of $\mu$ coincides with $\XXXX$, and there is a positive function $h\in C_{\wwww}(\XXXX)$ such that 
			\[\langle h,\mu\rangle=1,\qquad\PPPP h=c h.\]
			Moreover, for any $f\in\CC$ and integer $R\ge 1$, we have 
			\begin{align}\label{E1}
				c^{-k}\PPPP_kf\to\langle f,\mu\rangle h
			\end{align}
			in $C_b(\overline{B_\XXXX(R)})\mcap L^1(\XXXX,\mu)$, as $k\to+\infty$.
		\end{enumerate}
	\end{proposition}
	\begin{proof}[Sketch of the proof]
		The proof of Theorem 4.1 in \cite{JNPS18} remains applicable to this case, with the only adjustment required in Step 6. More precisely, by utilizing the stronger uniform irreducibility condition, one can derive the estimate (4.23) in \cite{JNPS18} with $\XXXX_R$ replaced by $\overline{B_{\XXXX}(R)}$. This, combined with the new condition \eqref{E0}, leads to the same contradiction.
	\end{proof}
	The following result is an improved version of Theorem A.4 in \cite{MN18}.
	\begin{proposition}\label{propositionE2}
		Let $\{P_t(u,\cdot)|u\in \XXXX,t\ge0\}$ be a generalized Markov family of transition kernels satisfying the following conditions.
		\begin{enumerate}
			\item \textbf{Concentration property.} There is a number $l_0> 0$ such that $P_{t}(u,\cdot)$ is concentrated on $\XXXX_{\infty}$ for any $u\in \XXXX$ and $t\ge l_0$.
			\item \textbf{Growth condition.} There is a weight function $\wwww$ satisfying \eqref{E0-2} and an integer $R_0\ge 1$ such that 
			\begin{align}\sup_{t\ge 0}\frac{\|\PPPP_t\wwww\|_{L^{\infty}_{\wwww}}}{\|\PPPP_t\mathbf{1}\|_{R_0}}<\infty,\label{E4}\\\label{E5}
				\sup_{t\in[0,1]}\|\PPPP_t\mathbf{1}\|_{L^{\infty}(\XXXX)}<\infty,
			\end{align}
			where $\mathbf{1}$ is the function on $\XXXX$ identically equal to $1$, and $\|\cdot\|_R$ denotes the $L^{\infty}$-norm on $\XXXX_{R}$.
			\item \textbf{Time-continuity.} For any function $f\in C_{\wwww}(\XXXX)$ and $u\in \XXXX$, the function $t\mapsto \PPPP_tf(u)$ is continuous from $[0,\infty)$ to $\R$.
			\item \textbf{Dissipativity condition.} For any integer $R\ge 1$ and $T>0$, there is an integer $\RR=\RR(R,T)\ge 1$ such that 
			\begin{align}\label{E2}
				\inf_{t\in[0,T]}\inf_{u\in \overline{B_{\XXXX}(R)}}P_t(u,\overline{B_\XXXX(\RR)})>0.
			\end{align}
			\item \textbf{Uniform irreducibility.} For any integers $\rho,R\ge 1$ and $r>0$, there are positive numbers $l=l(\rho,R,r)$ and $p=p(\rho,R,r)$ such that 
			\begin{align}\label{E3}
				P_l(u,B_{\XXXX_{\rho}}(\hat u,r))\ge p, 
			\end{align}
			where $u\in \overline{B_{\XXXX}(R)}$, $\hat{u}\in \XXXX_{\rho}$, and $B_{\XXXX_{\rho}}(\hat u,r):=\XXXX_{\rho}\mcap B_\XXXX(\hat u,r).$			\item \textbf{Uniform Feller property.} There is a determining family $\CC\subset C_b(\XXXX)$ and an integer $R_0\ge 1$ such that $\mathbf{1}\in\CC$ and the sequence $\{\|\PPPP_t\mathbf{1}\|^{-1}_R\PPPP_t f\}_{t\ge 0}$ is uniformly equicontinuous on $\XXXX_R$ for any $f\in\CC$ and $R\ge R_0$.
			\item \textbf{Uniform tail probability estimate.} There is a number $l_0>0$ such that for any integer $\RR\ge 1$ and $t\ge l_0$,
			\begin{align*}\sup_{u\in \overline{B_\XXXX(\RR)}}\int_{\XXXX_R^c}\wwww(v)P_t(u,\dd v)\to0,\end{align*}
			as $R\to+\infty$.
		\end{enumerate}
		Then, the following properties hold.
		\begin{enumerate}
			\item For any $t>0$, there is at most one measure $\mu_t\in \PP_{\wwww}(\XXXX)$ satisfying 
			\begin{align}
				\label{E7-1}\PPPP_t^*\mu_t=c_t\mu_t
			\end{align}
			and the tail probability estimate \eqref{E0}.
			\item If such a measure $\mu_t$ exists for all $t>0$, then it is independent of $t$ (we thus set $\mu:=\mu_t$), the corresponding eigenvalue $c_t$ is of the form $c_t=c^t, c>0$, the support of $\mu$ coincides with $\XXXX$, and there is a positive function $h\in C_{\wwww}(\XXXX)$ such that 
			\begin{align}\label{E7}
				\langle h,\mu\rangle=1,\qquad\PPPP_t h=c^t h.
			\end{align}
			Moreover, for any $f\in\CC$ and integer $R\ge 1$, we have 
			\begin{align}\label{E8}
				c^{-t}\PPPP_tf\to\langle f,\mu\rangle h
			\end{align}
			in $C_b(\overline{B_\XXXX(R)})\mcap L^1(\XXXX,\mu)$, as $t\to+\infty$.
		\end{enumerate}
	\end{proposition}
	\begin{proof}[Sketch of the proof] Let $t>0$ be an arbitrary number, and consider the family of transition kernels $\{\tilde P_k:=P_{kt},k\in\N\}$. The conditions of Proposition~\ref{propositionE1} are satisfied\footnote{The uniform irreducibility of $\{\tilde P_k\}_{k\ge 0}$ follows from \eqref{E2} and \eqref{E3}, and the others follow directly.} for $\{\tilde P_k(u,\cdot), u\in\XXXX, k\in\N\}$. Thus, the first statement holds.
		
		Suppose that the measure $\mu_t$ exists for all $t>0$, then by using \eqref{E7-1} and the Kolmogorov--Chapman relation, we have 
		\[\mu_t=\mu_1=:\mu,\qquad c_t=c_1^t=:c^t\]
		for any $t$ in the set $\mathbb{Q}_+$ of positive rational numbers. Using the time-continuity property and \eqref{E5}, we get 
		\[\PPPP^*_t\mu=c^t\mu\]
		for any $t>0$. Moreover, it is clear that $\supp \mu=\XXXX$, and there is a positive function $h_t\in C_{\wwww}(\XXXX)$ such that $\langle h_t,\mu\rangle=1$ and 
		\[\PPPP_th_t=c^th_t\]
		and 
		\begin{align}\label{E6}
			c^{-tk}\PPPP_{tk}f\to\langle f,\mu\rangle h_t
		\end{align}
		in $C_b(\overline{B_\XXXX(R)})\mcap L^1(\XXXX,\mu)$, as $k\to+\infty$ for any integer $R\ge 1$ and $f\in\CC$. Taking $f=1$ in \eqref{E6}, we see that $h_t=h_1=:h$ for all $t\in\mathbb{Q}_+$, which implies \eqref{E7} by using the time-continuity property. The derivation of \eqref{E8} follows in a similar way as in \cite{MN18} and is omitted. 
	\end{proof}

	\bibliographystyle{alpha}
	\bibliography{reference}
\end{document}